%% file: Liu_Barsukow_blood_flow.tex
\documentclass[review,hidelinks,onefignum,onetabnum]{siamart220329}

\input{ex_shared}

\newcommand{\vla}[1]{{\color{red}#1}}
\newcommand{\bla}[1]{{\color{blue}#1}}

\ifpdf
\hypersetup{
  pdftitle={An Arbitrarily High-Order Fully Well-balanced Hybrid Finite Element-Finite Volume Method for a One-dimensional Blood Flow Model},
  pdfauthor={Y. Liu and W. Barsukow}
}
\fi

\nolinenumbers

\begin{document}

\maketitle

\begin{abstract} We propose an arbitrarily high-order accurate, fully well-balanced numerical method for the one-dimensional blood flow model. The developed method employs a continuous solution representation, combining conservative and primitive formulations. Degrees of freedom are  point values at cell interfaces and moments of conservative variables within cells. \bla{The well-balanced property---ensuring exact preservation of zero and non-zero velocity steady-state solutions while accurately capturing small perturbations---is achieved through two key components. First, in the evolution of the moments, a local reference steady-state solution is obtained and subtracted. Second, the point value update happens in equilibrium variables. Extensive numerical tests are conducted to validate the preservation of various steady-state solutions, robust capturing small perturbations to such solutions, and high-order accuracy for both smooth and discontinuous solutions.}  
\end{abstract}



\section{Introduction}\label{sec1}
One-dimensional (1-D) blood flow models have proven to be valuable tools in order to mathematically understand and numerically address fundamental aspects of pulse wave propagation in the human cardiovascular system; see, e.g., \cite{OA_blood,Avolio,PJ}. \bla{By simplifying the complexity of full three-dimensional models, 1-D models provide a realistic yet computationally efficient alternative for} numerical simulation of blood flow, as shown in \cite{SFPP,SLH}. These models are particularly suited for the description of flow patterns, pressure wave propagation, and average velocities in a single vessel or in networks of vessels \cite{Jordi_2011,Formaggia03,SFPF2, Blanco2014}. \bla{Moreover, the 1-D models still retain the fluid-structure interaction in a simpler framework, enabling the computation of cross-sectional area, cross-sectional averages of velocity, flow and pressure at any time and position along the length of the vessel. This simplicity does not compromise their effectiveness, as demonstrated in various studies including the in-silico analysis, where the predictions of these models were compared with those of more complex models \cite{Grinberg2011,Xiao_2014,Boileau_2015}, in-vitro by assessing 1-D models output with respect to highly controlled experiments \cite{Matthys_2007,Bessems_2008}, and in-vivo by assessing the capacity of these models to reproduce pressure and flow waveforms observed in the clinical context \cite{OPKPNL,Reymond_2009}.} Another key advantage of 1-D models is their lower computational cost, which allows for the study of wave effects in isolated arterial segments or systemic arterial systems (i.e. in the aorta and systemic arteries) \cite{OPKPNL,Pontrelli,Pontrelli2,SFPF2}. \bla{Lastly}, 1-D models can be easily coupled with lumped-parameter models \cite{QTV,SLH} and three-dimensional fluid-structure models \cite{FGNQ0,FGNQ,QMV,BF13,BAUF}, making them a powerful tool for constructing comprehensive models of the human cardiovascular system.  

A well-established partial differential equation (PDE) model for the 1-D blood flow through thin-walled deformable elastic tubes takes the form (see, e.g., \cite{Toro13, GBLT}, the review paper \cite{Toro16} and references therein):  
\begin{equation}\label{1.1}
  \left\{\begin{aligned}
  &A_t+Q_x=0,\\
  &Q_t+\Big(\widehat\alpha\frac{Q^2}{A}\Big)_x+\frac{A}{\rho}p_x=0,
  \end{aligned}\right.
\end{equation}
where $A(x,t)=\pi R^2(x,t)$ is the cross-sectional area of the vessel with $R(x,t)>0$ being the radius and $Q(x,t)=A(x,t)u(x,t)$ is the \vla{flow} with $u(x,t)$ being the averaged velocity of blood at a cross section. $\rho>0$ represents the fluid density which is assumed to be constant and $p(x,t)$ denotes the averaged internal pressure which will be given later to close the description of the system \eref{1.1}. The parameter $\widehat \alpha$ is the momentum-flux correction coefficient that depends on the assumed velocity profile. For simplicity, in this paper we take $\widehat{\alpha}\equiv1$, which corresponds to a blunt velocity profile. 

Note that there are more unknowns than equations in \eref{1.1} and hence a closure condition linking the pressure with the displacement of the vessel will be needed. \vla{In the present work, we adopt the tube law of the form (\cite{PMTP,Toro13})
\begin{equation}\label{1.1b}
  p(x,t)=\mathcal{K}(x)\phi\Big(\frac{A(x,t)}{A_0(x,t)}\Big)+p_{\rm{ext}}(x),
\end{equation}
where $p_{\rm ext}(x)$ is the external pressure, $A_0(x)=\pi (R_0(x))^2$ is the cross-section area in an unloaded configuration (when $p=p_{\rm ext}$) with $R_0(x)$ being its radius. $\mathcal{K}(x)$ is the so-called stiffness coefficient and is a known positive function of the vessel wall Young modulus, the wall thickness and $A_0(x)$. $\phi(a)$ is the transmural pressure, assumed of the form
\begin{equation}\label{phiamn}
 \phi(a)=a^m-a^n.
\end{equation}
Here, $m>0$ and $n\in(-2,0]$ are real numbers and in this work we consider the typical values for collapsible tubes. Namely, $m=\frac{1}{2}, n=0$ for arteries and $m=10, n=-\frac{3}{2}$ for veins. Substituting the tube law \eref{1.1b} into \eref{1.1}, introducing the following two functions
\begin{equation}\label{Phiamn}
  \Phi(a)=\frac{1}{m+1}a^{m+1}-\frac{1}{n+1}a^{n+1},\quad \widetilde{\Phi}(a)=\frac{m}{m+1}a^{m+1}-\frac{n}{n+1}a^{n+1},
\end{equation}
and employing some straightforward manipulations,} we can rewrite the 1-D blood flow model \eref{1.1} in the form of hyperbolic balance laws: 
\vla{
\begin{equation}\label{AQ}
 \begin{aligned}
  &A_t+Q_x=0,\\
  &Q_t+\Big(\frac{Q^2}{A}+\frac{\mathcal{K}A_0}{\rho}\widetilde{\Phi}\big(\frac{A}{A_0}\big)\Big)_x=
  -\frac{A_0}{\rho}\Phi\big(\frac{A}{A_0}\big)\mathcal{K}_x
  +\frac{\mathcal{K}}{\rho}\widetilde{\Phi}(\frac{A}{A_0})(A_0)_x
  -\frac{A}{\rho}(p_{\rm ext})_x.
  \end{aligned}
\end{equation}}



\vla{As one can see, in the system of equations \eref{AQ} expressed by conservative variables, i.e. cross-sectional vessel area $A(x,t)$ and flow $Q(x,t)$, the} momentum balance equation includes source terms that depend on the spatial derivatives of tube law parameters. Mathematically, the blood flow system \eref{AQ}, therefore, belongs to the family of hyperbolic balance laws and admits non-trivial steady-state solutions, which often respect a delicate balance between the flux and the source terms at PDE level. \vla{Standard methods treat the flux derivative and the source differently, and setups that are stationary for the PDE are not stationary in the simulation. This error is reduced upon refining the grid, but often extreme refinement is needed until spurious waves are significantly smaller than the physical waves that one wants to resolve in a near-equilibrium setup. The increase in computational cost is usually unaffordable. Well-balanced methods solve the issue without increasing the computational effort. These methods are specifically designed to preserve the discrete version of steady-state solutions exactly, i.e. to machine accuracy and to capture effectively the nearly steady-state flows on relatively coarse meshes. Numerical comparisons between the well-balanced and standard non well-balanced methods will be conducted in \cref{sec4}, from which we further demonstrate the advantages of the development of well-balanced methods.}


The literature on well-balanced schemes for modeling 1-D blood flow is extensive, with numerous methods having been thoroughly studied over the past decade. For instance, in \cite{DL}, WB first- and second-order finite volume schemes were developed to exactly preserve the zero-velocity steady-state solutions in arteries. In \cite{MPT}, an Arbitrary high-order DERivatives (ADER) finite volume framework was utilized to construct a WB numerical scheme for 1-D blood flow in elastic vessels with varying mechanical properties. In \cite{WLD}, high-order WB finite difference weighted essentially non-oscillatory (WENO) schemes for the blood flow model in arteries were derived based on a specific splitting of the source term into two parts, which were then discretized separately with compatible WENO operators. Additionally, in \cite{LDY}, the hydrostatic reconstruction was used to construct high-order discontinuous Galerkin (DG) and finite volume WENO schemes, which can exactly preserve the zero-velocity steady states in arteries. However, the above-mentioned works only aimed at maintaining the simple zero velocity steady-state solutions. \vla{As explained in \cite{MPT,GDFL}, there is a need to preserve a broader class of steady-state solutions. Zero velocity steady-state solutions are not relevant in medical applications. As a first step towards numerical methods more adapted to practically relevant regimes greater attention should be given to more general non-zero velocity steady-state solutions, which represent ``living-man'' conditions.} Consequently, there has been a growing interest in developing fully WB methods that can preserve both the zero velocity (``blood-at-rest'') and non-zero velocity (``moving-blood'') steady-state solutions; see the definition in Section \ref{sec31}. In \cite{MG15}, the authors introduced an upwind discretization of the source term to create an energy-balanced numerical solver for the blood flow in artery. Likewise, in \cite{GBLT}, a positivity-preserving fully WB scheme was proposed for the 1-D blood flow model with friction. The authors in \cite{BX_blood} proposed a fully WB DG scheme by decomposing the numerical solutions into the steady-state and fluctuation parts for the blood flow in artery. In \cite{CK23_blood}, the authors developed a fully WB central-upwind scheme for the 1-D blood flow in artery based on the flux globalization technique. In \cite{PMTP}, based on a combination of the generalized hydrostatic reconstruction and well-balanced reconstruction operators, the authors proposed high-order fully well-balanced numerical methods for the 1-D blood flow model with discontinuous mechanical and geometrical properties. There are many related works, we have only named a few here.

\bla{In addition to well-balancedness, high-order accurate numerical schemes have a small numerical error at modest grid resolutions already. Examples include finite difference, finite volume (FV), and discontinuous Galerkin (DG) methods. High-order well-balanced FV or DG methods specifically designed for blood flow are discussed in some of the aforementioned works; see, e.g., \cite{LDY,MPT,BX_blood,PMTP}. Finite difference and finite volume methods typically improve accuracy by extending the computational stencil and using higher-order reconstructions, such as Essentially Non-Oscillatory (ENO) and Weighted ENO (WENO) schemes. In contrast, DG methods achieve high accuracy by increasing the number of degrees of freedom (DoFs) within each computational cell. However, these methods also have certain limitations. For instance, the large computational stencil in WENO schemes can complicate parallelization and the implementation of boundary conditions, while the DG method is often memory-intensive. The method proposed here has the advantages of DG methods, but due to shared degrees freedom requires less memory and less computational cost.}

The aim of this paper is to develop an arbitrarily high-order, fully WB, conservative, and positivity-preserving numerical method for the 1-D blood flow model that can handle both the ``blood-at-rest'' and general ``moving-blood'' steady states, which are defined in Section \ref{sec31}. To achieve this goal, we extend the idea introduced in \cite{Abgrall_camc,AL_SW} \bla{and use point values at cell interfaces along with higher moments within each cell as additional DoFs to enhance the accuracy. The higher moments are evolved by a ``moments system'' built upon the balance laws \eref{AQ}, while the point values are updated using} the following formulation of the PDE:
\begin{equation}\label{1.5}
  \left\{\begin{aligned}
  &A_t+(Au)_x=0,\\
  &\vla{u_t+\left(\frac{u^2}{2}+\frac{1}{\rho}\Big(\mathcal{K}\phi\big(\frac{A}{A_0}\big)+p_{\rm ext}\Big)\right)_x=0.}
  \end{aligned}\right.
\end{equation}
\bla{This distinguishes the proposed method from the traditional methods, where rely solely on the cell averages and hyperbolic equations \eref{AQ}.}
\vla{In \cite{AQAU}, \eref{AQ} is referred to as the $(A,Q)$ system, while \eref{1.5} is called the $(A,u)$ system. The former models the physically conserved variables of mass and flow, while there is a priori no physical conservation principle for the velocity. In this work, we refer to these two systems therefore as the conservative and primitive formulations, respectively, even though the primitive formulation for this particular system happens to lead to a conservation law. } 

The developed numerical methods are summarized as follows. We first assume that the numerical solution is globally continuous and can be described by a combination of point values at cell interfaces and moments of the solution, inspired by the Active Flux scheme and in particular the approach introduced in \cite{AB_FE_FV,AB_HOAF}. The point values are expressed in primitive variables $(A,u)$ and are evolved according to the primitive formulation \eref{1.5}, while the moments are related to the conservative variables $(A, Q)$ and their evolution is governed by a ``moments system'' obtained by integrating the conservative formulation \eref{AQ} over the cell. It is important to note that this is made possible by two sources of inspiration. First, the point values of primitive variables can always be transformed from the conservative variables at a specific point. Second, the cell averages are the lowest order moments and are updated conservatively. It has been proved that this new class of scheme combining formulations \eref{AQ} and \eref{1.5} satisfies a Lax-Wendroff-like theorem and guarantees convergence to a weak solution of the PDE \cite{Abgrall_camc}.  

Next, we discretise the ``moments system'' associated with \eref{AQ} and the primitive formulation \eref{1.5} simultaneously using the Active Flux scheme, similar to the new WB version proposed in \cite{AL_SW}, in space. Once we obtain the semi-discrete forms, we apply the standard Runge--Kutta method to update the moments and point values. The updates of the point values and moments must ensure that the resulting method is WB. To achieve this, we use the Gauss--Lobatto point values of the equilibrium variables (defined below in \eref{2.3a}) to approximate the spatial derivatives present in \eref{1.5} and also utilize the local reference steady-state solutions to derive a WB and exact conservation property enhanced quadrature of the source term in \eref{AQ}. Moreover, the usage of moments of the solution gives us as many DoFs as we require, enabling us to construct an arbitrarily high-order method. 

Finally, the proposed method should preserve the positivity of both the point values and cell averages of cross-sectional area $A(x,t)>0$ for $t>0$. To accomplish this, we adopt a simple multi-dimensional optimal order detection (MOOD) paradigm from \cite{CDL,Vilar,AL_SW} equipped with a first-order scheme that has a Local Lax--Friedrichs' flavour for the updates of average and point values. It is worth noting that if only the zeroth moment (cell average) and point values are taken into account, the proposed method is similar to active flux methods \cite{ER_AF1,HKS,Barsukow_AF,ER_AF2,BB} and the recently proposed method in \cite{AL_SW}.

\vla{The novelty of this work is summarized as follows: 
\begin{itemize}
       \item Free of special design for WB numerical fluxes. Traditional FV or DG methods approximate the solution using piecewise polynomials. WB schemes based on these methods often require specific techniques for computing numerical fluxes, such as the (generalized) hydrostatic reconstruction applied in \cite{PMTP,CK23_blood,BX_blood,MPT,Muller_2013HO,LDY}. In contrast, the proposed method assumes a globally continuous variable representation of the solution. This allows us to directly use integration by parts to compute the continuous flux without requiring special WB designs for the numerical fluxes.
           
       \item Arbitrarily high-order. The methods in \cite{AL_SW,BB} achieve third-order accuracy by using two endpoint values and one cell average as DoFs per cell. In this work, we add higher moments as additional DoFs to achieve arbitrarily high-order accuracy. This approach outperforms WENO methods by maintaining a compact stencil, avoiding the larger computational stencil required by WENO. Due to shared degrees of freedom at cell interfaces, it is also more economical than DG methods. Moreover, we improve the source term approximation in \cite{AL_SW}, which is done by an extrapolation technique and limited to even order accuracy as well as challenging to extend when using Gauss--Lobatto quadrature points.

       \item Exact conservation property for subsystems of conservation laws. When the mechanical and geometric parameters ($\mathcal{K}$, $A_0$, and $p_{\rm ext}$) in \eref{AQ} become constants, the source term disappears, reducing the system to a conservation law. For such subsystems, it is essential to design a conservative numerical scheme that not only respects physical laws but also ensures the numerical approximation converges to a weak solution, as guaranteed by the Lax--Wendroff theorem. To address this, a cut-off function is used in \cite{AL_SW} to switch off or limit the non-hydrostatic correction terms (numerical viscosity). This approach is reasonable, as solutions to conservation laws are quite often far from steady states, rendering the non-hydrostatic correction terms unnecessary. However, designing a cut-off function for high-order schemes presents significant challenges. In this work, we take an alternative approach to approximate the source term without compromising the exact conservation property of the scheme while ensuing the well-balanced property. 

       
       \item Positivity preservation and oscillation-control near strong discontinuities. It is known that high-order schemes are prone to numerical oscillations near discontinuities, potentially leading to non-physical solutions (e.g., negative cross-sectional area in blood flow). This may cause nonlinear instability or even code crashes. To address this, we have used a posteriori MOOD paradigm (\cite{CDL,Vilar,AL_SW}) to both the point and average values of the discrete solution. When the high-order solution fails to meet the MOOD criteria, the scheme is downgraded to a lower-order one, and the solution is recomputed in the affected cells. The lowest-order scheme employed is the Local Lax--Friedrichs scheme, which satisfies all MOOD criteria.
       
       \item Comprehensive numerical validation. The proposed method is rigorously validated through a series of 1-D numerical experiments. These include tests of WB property, positivity preservation, and high-order accuracy. Examples contain cases with simplified arterial tube law as well as more complex scenarios involving arteries and veins with space-varying mechanical and geometric parameters. The results demonstrate the accuracy and robustness of the proposed method.
           
     \end{itemize}
}

The rest of the paper is structured as follows. In \cref{sec2}, we provide some notations, introduce the DoFs required for the proposed method, and describe the corresponding polynomial approximation space. In \cref{sec3}, we begin by investigating the steady-state solutions of the 1-D blood flow system. Next, we introduce the WB interpolation of the equilibrium variables, which plays a crucial role in designing the proposed fully WB numerical methods. After that, we present a WB update for the ``moment system'' associated with \eref{AQ}, as well as a WB update for \eref{1.5}. The presentation of numerical results demonstrating the high-order accuracy (with a focus on third-, fourth-, and fifth-order), fully WB property, and the ability to provide good resolution for both smooth and discontinuous solutions is provided in \cref{sec4}.

\section{Preliminaries}\label{sec2}
In this section, we introduce some notations for the sake of simplicity and to define the DoFs, as well as the corresponding interpolation space. These will be used to construct arbitrarily high-order methods.

\subsection{Notations}\label{sec21}
To begin, we rewrite the conservative PDE model \eref{AQ} in the convenient vector form:
\begin{equation}\label{2.1}
  \bm U_t+\bm F(\bm U)_x=\bm S(\bm U, x),
\end{equation}
where
\begin{equation}\label{2.2}
  \bm U=\begin{pmatrix}A\\Q\end{pmatrix}\quad\mbox{and}\quad \bm F(\bm U)=\begin{pmatrix}Q\\\frac{Q^2}{A}+\vla{\frac{\mathcal{K}A_0}{\rho}\widetilde{\Phi}\big(\frac{A}{A_0}\big)}\end{pmatrix}
\end{equation}
\vla{are the conservative variables and flux functions, respectively. The source term is given by $\bm S(\bm U, x)=(0,S^{(2)})^T$, where
\begin{equation*}
   S^{(2)}=-\frac{A_0}{\rho}\Phi\big(\frac{A}{A_0}\big)\mathcal{K}_x
  +\frac{\mathcal{K}}{\rho}\widetilde{\Phi}\big(\frac{A}{A_0}\big)(A_0)_x-\frac{A}{\rho}(p_{\rm ext})_x.
\end{equation*}
The flux Jacobian matrix $J_{\bm U}=\frac{\partial \bm F}{\partial \bm U}$ is expressed as
\begin{equation*}
 J_{\bm U}=\begin{pmatrix}0 & 1\\-u^2+\frac{\mathcal{K}}{\rho}\widetilde{\Phi}'\big(\frac{A}{A_0}\big)& 2u\end{pmatrix}.
\end{equation*}
The eigenvalues of $J_{\bm U}$ are
\begin{equation}\label{evalue}
\lambda_1 = u + \sqrt{\frac{\mathcal{K}}{\rho} \frac{A}{A_0} \phi'\big(\frac{A}{A_0}\big)}, \quad
\lambda_2 = u - \sqrt{\frac{\mathcal{K}}{\rho} \frac{A}{A_0} \phi'\big(\frac{A}{A_0}\big)}.
\end{equation}
For the parameter ranges previously introduced, the system \eref{AQ} is hyperbolic since the matrix is diagonalizable. The flow regime is characterized by the Shapiro number, defined as:
\begin{equation*}
  S_h=\frac{\vert u\vert}{c}, \quad c=\sqrt{\frac{\mathcal{K}}{\rho}\frac{A}{A_0}{\phi}'\big(\frac{A}{A_0}\big)}.
\end{equation*}
A flow state is classified as subcritical if $S_h<1$, critical if $S_h=1$, and supercritical if $S_h>1$. As noted in \cite{GDFL}, under physiological conditions, blood flow is almost always subcritical. However, very specific pathologies may lead to supercritical flows. In the numerical experiments conducted in \cref{sec4}, only subcritical solutions will be considered. Nonetheless, the methods developed in this study are not limited to subcritical flows.}

In addition to the vector form of the conservative formulation, we also rewrite the primitive formulation \eref{1.5} in vector form:
\begin{equation}\label{2.3}
  \bm V_t+\bm E_x=0,
\end{equation}
where
\begin{equation}\label{2.3a}
  \bm V=\begin{pmatrix}A\\u\end{pmatrix},\quad \bm E:=\begin{pmatrix}Q\\E\end{pmatrix}=\begin{pmatrix}Au\\ \vla{\frac{u^2}{2}+\frac{1}{\rho}\Big(\mathcal{K}\phi\big(\frac{A}{A_0}\big)+p_{\rm ext}\Big)}\end{pmatrix} 
\end{equation}
are the primitive and introduced equilibrium variables, respectively. 
We note that the conservative variables $\bm U\in\mathcal{D}_{\bm U}$, where
\begin{equation*}
 \mathcal{D}_{\bm U}=\{\bm U=(A,Q)^\top\in\mathbb{R}^2: A>0\}. 
\end{equation*}
It is easy to find a mapping $\Psi$, which is one-to-one and $C^1$ as well as invertible, so that the primitive variables $\bm V$ can be transformed from the conserved one. Namely, $\bm V=\Psi(\bm U)\in \mathcal{D}_{\bm V}$, where
\begin{equation*}
\mathcal{D}_{\bm V}=\{\bm V=(A,u)^\top\in\mathbb{R}^2: A>0\}.
\end{equation*}

Finally, to describe the numerical methods, we divide the computational domain into a sequence of finite-volume cells denoted by $K_j:=[x_\jmh,x_\jph]$, where $j=1,\cdots,N$, with $N$ being the total number of finite volume cells. Each cell has a uniform size $\dx=x_\jph-x_\jmh$, with the center located at $x_j=(x_\jmh+x_\jph)/2$, and the cell interfaces at $x_\jmh$ and $x_\jph$.

\subsection{Degrees of freedom and interpolation polynomial space}
In this subsection, we declare higher moments to be new DoFs and construct the high-order interpolation polynomial space, inspired by finite element methods. 

Following the approach in \cite[Sections 2 and 3]{AB_FE_FV}, we first consider a basis $(b_\ell)_{\ell\geq0}$ and $b_\ell: [-\frac{\dx}{2},\frac{\dx}{2}]\rightarrow\mathbb{R}$ of some linear space $W\subset C^1$ of differentiable functions. We multiply both sides of \eref{2.1} by $b_\ell(x-x_j)$ and integrate by parts over the cell $K_j$ to obtain
\begin{equation}\label{2.5}
\begin{aligned}
  &\frac{{\rm d}}{{\rm d}t}\int_{K_j}{\bm U}(x)b_\ell(x-x_j)\,{\rm d}x
  +\bm F\big(\bm U(x_{\jph})\big)b_\ell\Big(\frac{\dx}{2}\Big)-\bm F\big(\bm U(x_\jmh)\big)b_\ell\Big(-\frac{\dx}{2}\Big)\\
  &\quad-\int_{K_j}\bm F(\bm U(x))\partial_xb_\ell(x-x_j)\,{\rm d}x=\int_{K_j}\bm S(\bm U, x)b_\ell(x-x_j)\,{\rm d}x
  \end{aligned}
\end{equation}
component-wise. In \eref{2.5}, all of the indexed quantities are time-dependent, but from here on, we omit this dependence for the sake of brevity. In DG methods, the expansion coefficients of components of $\bm U$ with respect to some basis are considered as DoFs, and the flux $\bm F\big(\bm U(x_\jph)\big)$ is replaced by some numerical flux obtained from the solution of a Riemann problem at $x=x_\jph$. Instead, in the proposed new method, we consider the $\ell$-th moments
\begin{equation}\label{2.6}
  \bm U^{(\ell)}_j:=C_\ell\int_{K_j}{\bm U}(x)b_\ell(x-x_j)\,{\rm d}x,\quad \ell\geq0,
\end{equation}
as the new DoFs and $C_\ell$ is a normalization constant that we will determine later. Multiplying both sides of \eref{2.5} by $C_\ell$ and using the definition of moments in \eref{2.6}, we can rewrite \eref{2.5} as
\begin{equation}\label{2.7}
\begin{aligned}
  &\frac{{\rm d}\bm U^{(\ell)}_j}{{\rm d}t}+C_\ell\Big[\bm{F}_\jph b_\ell\Big(\frac{\dx}{2}\Big)-\bm{F}_\jmh b_\ell\Big(-\frac{\dx}{2}\Big)\Big]\\
  &\quad-C_\ell\int_{K_j}\bm F(\bm U(x))\partial_xb_\ell(x-x_j)\,{\rm d}x=C_\ell\int_{K_j}\bm S(\bm U, x)b_\ell(x-x_j)\,{\rm d}x,
  \end{aligned}
\end{equation}
where $\bm{F}_\jph=\bm F\big(\bm U_\jph\big)$. Note that the proposed method is like an Active Flux method which maintains continuity across cell interface and introduces new independent point values $\bm U_\jph=\Psi^{-1}(\bm V_\jph)$ as additional DoFs, where $\bm V_\jph\approx\bm V(x_\jph)$ is governed by \eref{2.3}. 
 
Next, we give the definition of the finite element approximation.
\begin{defn}[\cite{AB_FE_FV}]
Let $(b_\ell)_{\ell\geq0}$ be a basis of some linear space $W\subset C^1$ of functions, and let $(C_\ell)_{\ell\geq0}$ be a sequence of non-zero real numbers. We consider the finite element triple $(K,V,\Sigma)$, where:
\begin{itemize}
  \item $K$ is the interval $[-\frac{\dx}{2},\frac{\dx}{2}]$;
  \item $V$ is the space of real-valued polynomials on $K$ of degree at most $r\geq2$, denoted by $\mathbb{P}^r$;
  \item $\Sigma\subset V'$, the dual space of $V$, consisting of degrees of freedom spanned by the set $\big\{\sigma_{-\frac{1}{2}},\sigma_{\frac{1}{2}},\sigma_0,\sigma_1,\cdots,\sigma_{r-2}\big\}$,
where $\sigma_{-\frac{1}{2}}$, $\sigma_{\frac{1}{2}}$, and $\sigma_\ell$ are defined as
\begin{equation}\label{2.8}
\sigma_{\pm\frac{1}{2}}(v):=v\Big(\pm\frac{\dx}{2}\Big),\quad \sigma_\ell(v):=C_\ell\int_{-\frac{\dx}{2}}^{\frac{\dx}{2}} b_\ell(x)v(x)\,{\rm d}x,  
\end{equation}
for all $\ell=0,1,\cdots,r-2$ and any $v\in V$.
\end{itemize}
\end{defn}

Then, we equip the space $V$ of shape functions with a basis
\begin{equation*}
  \big\{B_{-\frac{1}{2}},B_{\frac{1}{2}},B_0,B_1,\cdots,B_{r-2}\big\},
\end{equation*}
satisfying the following conditions:
\begin{equation}\label{2.8a}
\sigma_\ell(B_s)=\delta_{\ell s},\quad \forall \ell, s\in\Big\{-\frac{1}{2},\frac{1}{2},0,1,\cdots,r-2\Big\},
\end{equation}
where $\delta$ is the Kronecker function.
Note that we will use the same names to refer to the extensions of elements of $\Sigma$ to $L^1(K)$.  Using this notation, the interpolation operator $I:L^1(K)\rightarrow V$ is defined as
\begin{equation}\label{2.8b}
  I(v)(x):=\!\!\!\!\!\!\!\!\! \sum_{\ell\in\{-\frac{1}{2},\frac{1}{2},0,1,\ldots,r-2\}} \!\!\!\!\!\!\!\!\! \sigma_{\ell}(v)B_{\ell}(x),\quad \forall x\in K.
\end{equation}
Similarly, we define the reconstruction operator $R:\mathbb{R}^r\rightarrow V$ as 
\begin{equation}\label{2.8c}
  R\big(a_{-\frac{1}{2}},a_{\frac{1}{2}},a_0,a_1,\ldots,a_{r-2}\big)(x):=\!\!\!\!\!\!\!\!\!\sum_{\ell\in\{-\frac{1}{2},\frac{1}{2},0,1,\ldots,r-2\}} \!\!\!\!\!\!\!\!\! a_{\ell}B_{\ell}(x),\quad \forall x\in K.
\end{equation}

Finally, we give a set of monomial basis functions $(b_\ell)_{\ell\geq0}$, the corresponding normalization constant $C_\ell$, and the shape functions. 
\begin{defn}[\cite{AB_FE_FV,AB_HOAF}]\label{def1}
Let $W=\bigcup_{r\in\mathbb{Z}}\mathbb{P}^{r}$ be the union of polynomial spaces, with $(b_\ell)=(1,x,x^2,\ldots)$ its monomial basis. We define the normalization constant as 
\begin{equation}\label{Al}
  C_\ell:=\frac{(\ell+1)2^{\ell}}{\dx^{\ell+1}}=\left(\frac{1}{\dx},\frac{4}{\dx^2},\frac{12}{\dx^3},\frac{32}{\dx^4},\frac{80}{\dx^5},\ldots\right),
\end{equation}
such that $\sigma_\ell(1)=1$, $\forall \ell\in2\mathbb{Z}$. We can then define the following shape functions (with $\xi:=(x-x_j)/\dx\in[-\frac{1}{2},\frac{1}{2}]$) for third-, fourth-, and fifth-order methods:
\begin{itemize}
  \item when $r=2$, a parabolic polynomial space is considered, and the corresponding basis that satisfies \eref{2.8}--\eref{2.8a} is given by
  \begin{equation}\label{2.8d}
    \left\{\begin{aligned}
    & B_{-\frac{1}{2}}=\frac{1}{4}(2\xi-1)(1+6\xi),\\
    & B_0=-\frac{3}{2}(2\xi-1)(1+2\xi),\\
    & B_\frac{1}{2}=\frac{1}{4}(1+2\xi)(6\xi-1);
    \end{aligned}\right.
  \end{equation}
  \item when $r=3$, a cubic polynomial space is considered, and the corresponding basis that satisfies \eref{2.8}--\eref{2.8a} is given by:
  \begin{equation}\label{2.8e}
    \left\{\begin{aligned}
    & B_{-\frac{1}{2}}=-\frac{1}{4}(2\xi-1)(-1+4\xi+20\xi^2),\\
    & B_0=-\frac{3}{2}(2\xi-1)(1+2\xi),\\
    & B_1=-\frac{15}{2}\xi(2\xi-1)(1+2\xi),\\
    & B_\frac{1}{2}=\frac{1}{4}(1+2\xi)(-1-4\xi+20\xi^2);
    \end{aligned}\right.
  \end{equation}
  \item when $r=4$, a quartic polynomial space is considered, and the corresponding basis that satisfies \eref{2.8}--\eref{2.8a} is given by:
    \begin{equation}\label{2.8f}
    \left\{\begin{aligned}
    & B_{-\frac{1}{2}}=\frac{1}{16}(2\xi-1)(-3-30\xi+60\xi^2+280\xi^3),\\
    & B_0=\frac{15}{16}(2\xi-1)(1+2\xi)(-3+28\xi^2),\\
    & B_1=-\frac{15}{2}\xi(2\xi-1)(1+2\xi),\\
    & B_2=-\frac{35}{16}(2\xi-1)(1+2\xi)(20\xi^2-1),\\
    & B_\frac{1}{2}=\frac{1}{16}(1+2\xi)(3-30\xi-60\xi^2+280\xi^3).
    \end{aligned}\right.
  \end{equation}
\end{itemize}
\end{defn}

\begin{rmk}
One can construct arbitrarily higher order basis functions that satisfy the interpolation conditions \eref{2.8}--\eref{2.8a} and these can be found in \cite[Section 3]{AB_FE_FV}. However, our focus in this paper, as mentioned in Section \ref{sec1}, is on numerical methods of third-, fourth-, and fifth-order accuracy. Thus, we limit our consideration to basis functions of degree $r=2$, $r=3$, and $r=4$.
\end{rmk}
\section{An arbitrarily high-order fully well-balanced method}\label{sec3}
In this section, we present a novel arbitrarily high-order, fully WB active flux method. This method can accurately preserve both the simpler ``blood-at-rest'' steady states and the more complex general ``moving-blood'' steady states of the 1-D blood flow model. Our approach is built upon the two different forms of the studied PDE model, namely \eref{2.1}--\eref{2.2} and \eref{2.3}--\eref{2.3a}. 

\subsection{Steady-state solutions}\label{sec31}
Before introducing the numerical method, we analyze the steady-state solutions of interest \cite{BX_blood,CK23_blood,PMTP}. By definition, a smooth steady-state solution satisfies the time-independent system $\bm F(\bm U)_x=\bm S(\bm U, x)$, i.e., 
\begin{equation}\label{s0}
  \left\{\begin{aligned}
  &Q_x=0,\\
  &\vla{\Big(\frac{Q^2}{A}+\frac{\mathcal{K}A_0}{\rho}\widetilde{\Phi}\big(\frac{A}{A_0}\big)\Big)_x=-\frac{A_0}{\rho}\Phi\big(\frac{A}{A_0}\big)\mathcal{K}_x
  +\frac{\mathcal{K}}{\rho}\widetilde{\Phi}\big(\frac{A}{A_0}\big)(A_0)_x-\frac{A}{\rho}(p_{\rm ext})_x.}
  \end{aligned}\right.
\end{equation} 
After some simple algebraic manipulations or directly observed from \eref{2.3a}, we obtain the general steady-state solutions with non-zero velocity, which are referred to as ``moving-blood'' steady-state solutions:
\begin{equation}\label{s1}
  Q=Au\equiv{\rm Const}, \quad E=\vla{\frac{u^2}{2}+\frac{1}{\rho}\Big(\mathcal{K}\phi\big(\frac{A}{A_0}\big)+p_{\rm ext}\Big)}\equiv \rm Const.  
\end{equation}

If $u\equiv0$, the general steady-state solutions \eref{s1} simplify to the ``blood-at-rest'' steady-state solutions in the literature:
\begin{equation}\label{s3}
  u(x)\equiv0,\quad E=\vla{\frac{1}{\rho}\Big(\mathcal{K}\phi\big(\frac{A}{A_0}\big)+p_{\rm ext}\Big)}=\rm Const.
\end{equation}

The remainder of this section is dedicated to introducing the arbitrarily high-order fully WB numerical method, which is capable of exactly preserving the steady states mentioned above: \eref{s1} and \eref{s3}.

\subsection{WB interpolation of the equilibrium variable}\label{sec23}
We assume that at a certain time $t\geq0$, the $\ell$-th moments ${\bm U}^{(\ell)}_j$ defined in \eref{2.6} and the point values $\bm V_{j\pm\frac{1}{2}}\approx\bm V(x_{j\pm\frac{1}{2}})$ are available. We can immediately obtain $\bm U_{j\pm\frac{1}{2}}:=\Psi^{-1}(\bm V_{j\pm\frac{1}{2}})$ thanks to the invertibility of $\Psi$. Therefore, for a $(r+1)$-th order numerical method, in each cell $K_j$, we have $\ell$-th moments $\{{\bm U}^{(\ell)}_j\}_{\ell=0}^{r-2}$ and two point values placed at cell interfaces and shared by adjacent cells (see Figure \ref{U_recon} for $r=4$). From the Figure, we can observe that every cell has access to $r+1$ pieces of interpolation information. Hence, in every cell, we can use \eref{2.8c} to construct a $r$-th degree polynomial interpolant $\widetilde{\bm U}(x)$:
\begin{equation}\label{2.9}
  \widetilde{\bm U}(x)=\bm U_\jmh B_{-\frac{1}{2}}(\xi)+\bm U_\jph B_{\frac{1}{2}}(\xi)+\sum_{\ell=0}^{r-2}{\bm U}_j^{(\ell)}B_{\ell}(\xi),\quad \xi=\frac{x-x_j}{\dx},\quad\forall x\in K_j,
\end{equation}
where the basis functions are given in \eref{2.8d}, \eref{2.8e}, and \eref{2.8f}, with respect to $r=2$, $r=3$, and $r=4$.
\begin{figure}[ht!]
\centerline{\includegraphics[trim=0.01cm 0.01cm 0.01cm 0.01cm,clip,width=10.5cm]{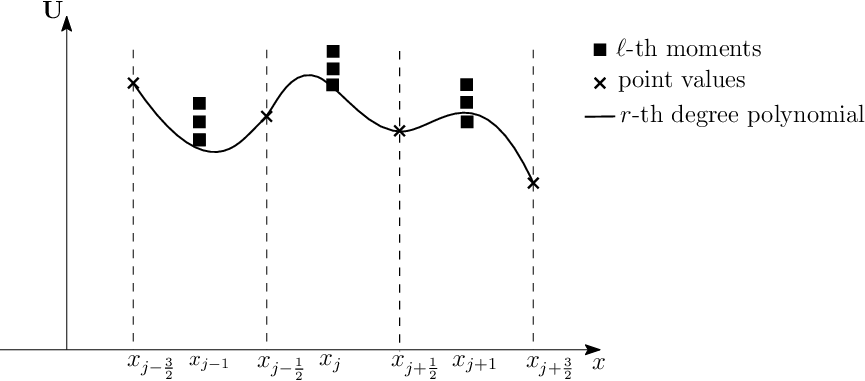}}
\caption{Five degrees of freedom and a unique continuous polynomial interpolant.\label{U_recon}}
\end{figure}

In order to make the developed method fully WB, it is well-known that one has to perform interpolation on the equilibrium variable $\bm E=(Q, E)^\top$; see, e.g., \cite{KLX,CK23_blood,CKLX,XS14,AL_SW,Xu2024}. To this end, we first compute from $\widetilde{\bm U}(x)$ the point values of $\bm E$ at nodes (in $[-\frac{1}{2},\frac{1}{2}]$) and then use them to interpolate $\bm E$ to the relevant order of accuracy. \vla{To compute $\bm E_{k}=(Q_{k},E_{k})$ at the Gauss--Lobatto points $\{X_k\}_{k=1}^{r+1}\in[-\frac{1}{2},\frac{1}{2}]$ inside the cell $K_j$, we first calculate the point values $\bm U_k=(A_k,Q_k)$ using \eref{2.9}:
\begin{equation}\label{Um}
  \bm U_k\approx\widetilde{\bm U}(X_k)=\bm U_\jmh B_{-\frac{1}{2}}(X_k)+\bm U_\jph B_{\frac{1}{2}}(X_k)+\sum_{\ell=0}^{r-2}{\bm U}_j^{(\ell)}B_{\ell}(X_k),
\end{equation}
which yields $A_k$ and $Q_k$. Next, using the formula in \eref{2.3a}, $E_k$ is computed as
\begin{equation}\label{Em}
  E_k=\frac{1}{2}\Big(\frac{Q_k}{A_k}\Big)^2+\frac{1}{\rho}\Bigg(\mathcal{K}_k\phi\Big(\frac{A_k}{(A_0)_k}\Big)+(p_{\rm ext})_k\Bigg),
\end{equation}
where the functions $\mathcal{K}(x)$, $A_0(x)$, and $p_{\rm ext}(x)$ have been projected onto the same polynomial space as $\bm U$. Thus, the point values $\mathcal{K}_k$, $(A_0)_k$, and $(p_{\rm ext})_k$ are computed in the same manner as in \eref{Um}, namely,
\begin{equation}\label{KA0pm}
  \begin{aligned}
   &\mathcal{K}_k=\mathcal{K}_\jmh B_{-\frac{1}{2}}(X_k)+\mathcal{K}_\jph B_{\frac{1}{2}}(X_k)+\sum_{\ell=0}^{r-2}\mathcal{K}_j^{(\ell)}B_{\ell}(X_k),\\
   &(A_0)_k=(A_0)_\jmh B_{-\frac{1}{2}}(X_k)+(A_0)_\jph B_{\frac{1}{2}}(X_k)+\sum_{\ell=0}^{r-2}(A_0)_j^{(\ell)}B_{\ell}(X_k),\\
   &(p_{\rm ext})_k=(p_{\rm ext})_\jmh B_{-\frac{1}{2}}(X_k)+(p_{\rm ext})_\jph B_{\frac{1}{2}}(X_k)+\sum_{\ell=0}^{r-2}(p_{\rm ext})_j^{(\ell)}B_{\ell}(X_k).
  \end{aligned}
\end{equation}}

The interpolant functions of $\bm E$ are
\begin{itemize}
  \item third-order scheme:
  \begin{equation}\label{2.10}
  R_{\bm E}=\bm E_\jmh L_{-\frac{1}{2}}+\bm E_j L_0+\bm E_\jph L_{\frac{1}{2}}
\end{equation}
with
\begin{equation*}
  L_{-\frac{1}{2}}(\xi)=\xi(2\xi-1),\quad L_0(\xi)=(1+2\xi)(1-2\xi),\quad L_{\frac{1}{2}}(\xi)=\xi(2\xi+1);
\end{equation*}
  \item fourth-order scheme:
  \begin{equation}\label{2.10a}
  R_{\bm E}=\bm E_\jmh L_{-\frac{1}{2}}+\bm E_{j-\sqrt{\frac{1}{20}}} L_{-\sqrt{\frac{1}{20}}}+\bm E_{j+\sqrt{\frac{1}{20}}}L_{\sqrt{\frac{1}{20}}}+\bm E_\jph L_{\frac{1}{2}}
\end{equation}
with
\begin{equation*}
\left\{\begin{aligned}
&L_{-\frac{1}{2}}(\xi)=-5\xi^3+\frac{5}{2}\xi^2+\frac{1}{4}\xi-\frac{1}{8},\\
&L_{-\sqrt{\frac{1}{20}}}(\xi)=5\sqrt{5}\xi^3-\frac{5}{2}\xi^2-\frac{5\sqrt{5}}{4}\xi+\frac{5}{8},\\
&L_{\sqrt{\frac{1}{20}}}(\xi)=-5\sqrt{5}\xi^3-\frac{5}{2}\xi^2+\frac{5\sqrt{5}}{4}\xi+\frac{5}{8},\\
&L_{\frac{1}{2}}(\xi)=5\xi^3+\frac{5}{2}\xi^2-\frac{1}{4}\xi-\frac{1}{8};
  \end{aligned}\right.
\end{equation*}
  \item fifth-order scheme:
  \begin{equation}\label{2.11}
  R_{\bm E}=\bm E_\jmh L_{-\frac{1}{2}}+\bm E_{j-\sqrt{\frac{3}{28}}} L_{-\sqrt{\frac{3}{28}}}+\bm E_j L_0+\bm E_{j+\sqrt{\frac{3}{28}}}L_{\sqrt{\frac{3}{28}}}+\bm E_\jph L_{\frac{1}{2}},
\end{equation}
with
\begin{equation*}
\left\{\begin{aligned}
&L_{-\frac{1}{2}}(\xi)=\xi\left(\frac{3}{4}-\frac{3}{2}\xi-7\xi^2+14\xi^3\right),\\
&L_{-\sqrt{\frac{3}{28}}}(\xi)=\frac{7\xi}{3}\left(-\frac{\sqrt{21}}{4}+\frac{7}{2}\xi+\sqrt{21}\xi^2-14\xi^3\right),\\
&L_0(\xi)=1-\frac{40}{3}\xi^2+\frac{112}{3}\xi^4,\\
&L_{\sqrt{\frac{3}{28}}}(\xi)=\frac{7\xi}{3}\left(\frac{\sqrt{21}}{4}+\frac{7}{2}\xi-\sqrt{21}\xi^2-14\xi^3\right),\\
&L_{\frac{1}{2}}(\xi)=\xi\left(-\frac{3}{4}-\frac{3}{2}\xi+7\xi^2+14\xi^3\right).
  \end{aligned}\right.
\end{equation*}
\end{itemize}

\subsection{WB update of moments}\label{sec24}
In this subsection, we describe how to update the $\ell$-th moment over time in an arbitrarily high-order WB manner. This involves solving the semi-discretization \eref{2.7}, i.e.,
\begin{equation}\label{2.13}
\frac{{\rm d}\bm U^{(\ell)}_j}{{\rm d}t}+C_\ell\Big[\bm{F}_\jph b_\ell\big(\frac{\dx}{2}\big)-\bm{F}_\jmh b_\ell\big(-\frac{\dx}{2}\big)\Big]={\bm F}^{(\ell)}_j+{\bm S}^{(\ell)}_j,\quad \ell=0,1,\cdots,r-2,
\end{equation}
where
\begin{equation}\label{2.14}
  \bm{F}_\jph=\bm F(\bm U_\jph)\stackrel{\eref{2.2}}{=}\begin{pmatrix}Q_\jph\\ \frac{Q_\jph^2}{A_\jph}+\vla{\frac{\mathcal{K}_\jph(A_0)_\jph}{\rho}\widetilde{\Phi}\big(\frac{A_\jph}{(A_0)_\jph}\big)}\end{pmatrix},
\end{equation}
${\bm F}^{(\ell)}_j$ and ${\bm S}^{(\ell)}_j$ are the bulk terms involving the flux function and the source term, defined as:
\begin{equation}\label{2.15a}
  {\bm F}^{(\ell)}_j:= C_\ell\int\limits_{K_j}\bm F(\bm U)\partial_xb_\ell(x-x_j)\,{\rm d}x
\end{equation}
and
\begin{equation}\label{2.15}
  {\bm S}^{(\ell)}_j:= C_\ell\int\limits_{K_j}\bm S(\bm U, x)b_\ell(x-x_j)\,{\rm d}x,
\end{equation}
respectively.

\vla{In order to construct fully WB schemes, it is also necessary to design appropriate WB approximations of \eref{2.15a} and \eref{2.15} such that the source term approximations of ${\bm F}^{(\ell)}_j+{\bm S}^{(\ell)}_j$ exactly balance the flux in \eref{2.14} at steady state. Furthermore, when the source term vanishes, the PDE reduces to a hyperbolic conservation law. In such cases, it is crucial to ensure that the numerical approximation of \eref{2.15} maintains exact conservation, or at the very least, consistency. Conservation is critical for updating the zero-th moment, $\bm U_j^{(0)}$, which represents the cell average in FV methods. This requirement is not only to respect physical laws but also to guarantee convergence to a weak solution of the hyperbolic equation, as ensured by the Lax--Wendroff theorem. To satisfy both the WB and exact conservation properties, we reformulate the source term ${\bm S}^{(\ell)}_j$ as
\begin{equation}\label{sue}
  {\bm S}^{(\ell)}_j=C_\ell\int\limits_{K_j}\big(\bm S(\bm U, x)-\bm S(\widehat{\bm U}, x)\big)b_\ell(x-x_j)\,{\rm d}x+C_\ell\int\limits_{K_j}\bm F(\widehat{\bm U})_xb_\ell(x-x_j)\,{\rm d}x,
\end{equation}
where $\widehat{\bm U}$ is a local reference steady-state solution satisfying $\bm F(\widehat{\bm U})_x=\bm S(\widehat{\bm U}, x)$ within each cell $K_j$. The definition of $\widehat{\bm U}$ is provided below. Applying integration by parts to the second term in \eref{sue} yields
\begin{equation}\label{sue2}
\begin{aligned}
  C_\ell\int\limits_{K_j}\bm F(\widehat{\bm U})_xb_\ell(x-x_j)\,{\rm d}x&=C_\ell\Big[\widehat{\bm{F}}_\jph b_\ell\big(\frac{\dx}{2}\big)-\widehat{\bm{F}}_\jmh b_\ell\big(-\frac{\dx}{2}\big)\Big]\\
   &\quad -C_\ell\int\limits_{K_j}\bm F(\widehat{\bm U})\partial_xb_\ell(x-x_j)\,{\rm d}x,
  \end{aligned}
\end{equation}
where $\widehat{\bm F}_\jph=\bm F(\widehat{\bm U}_\jph)$ is the flux evaluated at $x_\jph$ within $K_j$. Combining \eref{2.15a}, \eref{sue}, and \eref{sue2}, the source terms in \eref{2.13} become
\begin{equation}\label{newsource}
\begin{aligned}
  {\bm F}^{(\ell)}_j+{\bm S}^{(\ell)}_j&=C_\ell\Big[\widehat{\bm{F}}_\jph b_\ell\big(\frac{\dx}{2}\big)-\widehat{\bm{F}}_\jmh b_\ell\big(-\frac{\dx}{2}\big)\Big]\\
  &\quad +C_\ell\int\limits_{K_j}\big(\bm F(\bm U)-\bm F(\widehat{\bm U})\big)\partial_xb_\ell(x-x_j)\,{\rm d}x\\
  &\quad +C_\ell\int\limits_{K_j}\big(\bm S(\bm U, x)-\bm S(\widehat{\bm U}, x)\big)b_\ell(x-x_j)\,{\rm d}x.
  \end{aligned}
\end{equation}
The integrals in \eref{newsource} are computed using high-order Gauss--Lobatto quadrature.} 

\vla{The next step is to define the local reference steady-state solution $\widehat{\bm U}$ at the $r+1$ Gauss--Lobatto points. Specifically, $\bm \widehat{\bm U}_k\approx\widehat{\bm U}(X_k)$ for $k=1,\cdots, r+1$, where $X_k\in K_j$ is the $k$-th Gauss--Lobatto point. Recall that in \eref{Um}--\eref{Em}, values $\{(Q_k,E_k)\}_{k=1}^{r+1}$ of the equilibrium variables have already been computed at Gauss-Lobatto points in each cell.} Following the idea presented in \cite{Xu2024,Xu2024a}, we first chose from those local reference equilibrium values $\bm E^{e}=(Q^e, E^e)$ for each cell $K_j$:
\begin{equation}\label{2.15e}
 \vla{ {Q}^{e}=Q_{\iota},\quad {E}^{e}=E_{\iota},}
\end{equation}
\vla{where $1\leq\iota\leq r+1$. $\iota$ is chosen such that they satisfy the following conditions for all $k\in\{1,2,\cdots,r+1\}\backslash\{\iota\}$:
\begin{equation}\label{FF'}
\digamma'(A^*; Q_\iota, E_\iota, \mathcal{K}_k, (A_0)_k)=0,\quad
\digamma(A^*, Q_\iota, E_\iota,\mathcal{K}_k, (A_0)_k,(p_{\rm ext})_k)\leq0.
\end{equation}
Here, 
\begin{equation}\label{Fa}
 \digamma(A):=\digamma(A; Q, E, \mathcal{K}, A_0, p_{\rm ext})=\frac{Q^2}{2A^2}+\frac{\mathcal{K}}{\rho}\Big[\big(\frac{A}{A_0}\big)^{m}-\big(\frac{A}{A_0}\big)^{n}\Big]+\frac{p_{\rm ext}}{\rho}-E,
\end{equation}
and its first derivative with respect to $A$ is 
\begin{equation}\label{Fa'}
  \digamma'(A)=-\frac{Q^2}{A^3}+\frac{\mathcal{K}}{\rho A_0}\Big[m\big(\frac{A}{A_0}\big)^{m-1}-n\big(\frac{A}{A_0}\big)^{n-1}\Big].
\end{equation} 

\vla{The reasoning behind thse conditions is explained in Remark \ref{rmk:iota} below. A pseudocode implementation for determining $\iota$ is provided in \cref{alg:findl}.
\begin{algorithm}
\color{red}
\caption{\vla{Determine $\iota$ using \eref{FF'}}}
\label{alg:findl}
\begin{algorithmic}[1]
\WHILE{$1\leq\iota\leq r+1$}
\STATE{Set flag=.FALSE.}
\FORALL{$1\leq k\leq r+1$ and $k\neq\iota$}
\STATE{Solve $\digamma'(A^*; Q_\iota, E_\iota, \mathcal{K}_k, (A_0)_k)=0$ for $A^*$. }
\STATE{Evaluate $\digamma(A^*, Q_\iota, E_\iota,\mathcal{K}_k, (A_0)_k,(p_{\rm ext})_k)$ using \eref{Fa}.}
\IF {$\digamma(A^*, Q_\iota, E_\iota,\mathcal{K}_k, (A_0)_k,(p_{\rm ext})_k)\leq0$}
\STATE{Set flag=.TRUE.}
\STATE{$k=k+1$}
\ELSE
\STATE{Set flag=.FALSE.}
\STATE{Exit loop on $k$ and continue loop on $\iota$}
\ENDIF
\ENDFOR
\IF{flag==.TRUE.}
\RETURN{$\iota$}
\STATE{Exit loop on $\iota$}
\ENDIF
\ENDWHILE
\end{algorithmic}
\end{algorithm}  }

Once the index $\iota$ is determined, i.e., $\bm E^e=(Q^e, E^e)$ is selected, we} set \vla{$\widehat{Q}_k={Q}^{e}$} and use Newton--Raphson method to solve the following nonlinear equations, derived from \eref{2.3a}, to compute \vla{$\widehat{A}_k$}: 
\begin{equation}\label{2.15bb}
  \vla{E^{e}=\frac{(Q^e)^2}{2(\widehat{A}_k)^2}+\frac{1}{\rho}\Big(\mathcal{K}_k\phi\big(\frac{\widehat{A}_k}{(A_0)_k}\big)+(p_{\rm ext})_k\Big)},\quad k=1,\cdots, r+1,
\end{equation}
where \vla{$\mathcal{K}_k$, $(A_0)_k$, and $(p_{\rm ext})_k$ are defined in \eref{KA0pm}. In rare cases where no $\iota$ satisfies \eref{FF'}, we set $\widehat{A}_k=0$ and $\widehat{Q}_k=0$ for all $k$. These cases may only occur when the flow is far from steady state. }

\vla{
\begin{rmk}\label{rmk:iota}
We illustrate the rationale behind the conditions in \eref{FF'} by examining the properties of the functions $\digamma(A)$ in \eref{Fa} and $\digamma'(A)$ in \eref{Fa'}, under the parameter ranges ($\mathcal{K}(x)>0$, $\rho>0$, $m>0$, and $n\in(-2,0]$). First, we consider $\digamma'(A)$ and the condition $\digamma'(A)=0$ implies
\begin{equation*}
  \digamma'(A)=0\Leftrightarrow \frac{\mathcal{K}}{\rho}\Big[m\frac{A^{m+2}}{(A_0)^m}-n\frac{A^{n+1}}{(A_0)^n}\Big]-Q^2=0.
\end{equation*}

Next, we let 
\begin{equation*}
  g(A)=\frac{\mathcal{K}}{\rho}\Big[m\frac{A^{m+2}}{(A_0)^m}-n\frac{A^{n+2}}{(A_0)^n}\Big]-Q^2,
\end{equation*}
and note that this function satisfies:
\begin{equation*}
  \lim_{A\rightarrow0^+}g(A)=-Q^2\leq0,\quad \lim_{A\rightarrow+\infty}g(A)=+\infty,
\end{equation*}
and
\begin{equation*}
  g'(A)=\frac{\mathcal{K}}{\rho}\Big[m(m+2)\frac{A^{m+1}}{(A_0)^m}-n(n+2)\frac{A^{n+1}}{(A_0)^n}\Big]\geq0.
\end{equation*}
Thus, $g(A)$ is increasing in $A$, ensuring a unique root in $[0,+\infty)$. This also guarantees the existence and uniqueness of the root $A^*\geq0$ of $\digamma'(A)$, which can be efficiently computed using the Newton--Raphson method.

Finally, we analyze the properties of $\digamma(A)$:
\begin{itemize}
\item Case $Q=0$: from \eref{Fa'}, we know that 
\begin{equation*}
  \digamma'(A)=\frac{\mathcal{K}}{\rho A_0}\Big[m\big(\frac{A}{A_0}\big)^{m-1}-n\big(\frac{A}{A_0}\big)^{n-1}\Big]\geq0, \quad \forall A\geq0,
\end{equation*}
which implies that $\digamma(A)$ is increasing in $A$. Furthermore, 
    \begin{equation*}
      \lim\limits_{A\rightarrow+\infty}\digamma(A)=+\infty,\quad \digamma'(A^*=0)=0.
    \end{equation*}
    Therefore, $\digamma(A)$ has a unique root in $(0,+\infty)$ if $\digamma(A^*)\leq0$; see Figure \ref{Phi}--(a) for an illustration.
\item Case $Q\neq0$: from \eref{Fa}, we know that
      \begin{equation*}
        \lim\limits_{A\rightarrow0^+}\digamma(A)=+\infty,\quad \lim\limits_{A\rightarrow+\infty}\digamma(A)=+\infty.
       \end{equation*}
From \eref{Fa'}, we observe that
\begin{equation*}
  \lim\limits_{A\rightarrow0^+}\digamma'(A)=-\infty,\quad \lim\limits_{A\rightarrow+\infty}\digamma'(A)=+\infty.
\end{equation*}
Thus, $\digamma(A)$ attains a minimum at $A=A^*$. It has a root in $(0,+\infty)$ if $\digamma(A^*)\leq0$. This behaviour is illustrated in Figure \ref{Phi}--(b).
\end{itemize}

\begin{figure}[ht!]
\centerline{\subfigure[$\digamma(A)$ with $Q=0$]{\includegraphics[trim=0.005cm 0.25cm 0.05cm 0.02cm,clip,width=4.5cm]{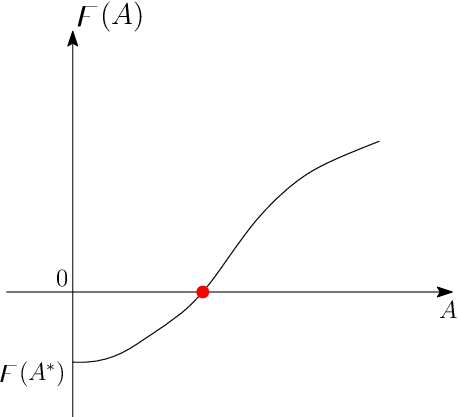}}\hspace*{2.0cm}
\subfigure[$\digamma(A)$ with $Q\neq0$]{\includegraphics[trim=0.005cm 0.25cm 0.05cm 0.02cm,clip,width=4.5cm]{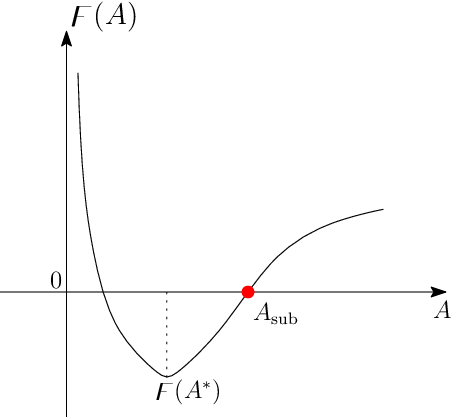}}}
\caption{\sf Sketch of function $\digamma(A)$.\label{Phi}}
\end{figure}

\end{rmk}
}





\subsection{WB update of point values}\label{sec35}
In this subsection, we follow the approach outlined in \cite{AL_SW} and solve the system \eref{2.3} to evolve the point value $\bm V_\jph$ in time. The semi-discrete form is given by
\begin{equation}\label{2.17}
  \frac{\rm d}{{\rm d}t}\,{\bm V}_\jph=-(\overleftarrow{\Phi}_{j+1}^{\bm V}+\overrightarrow{\Phi}_j^{\bm V}),
\end{equation}
where $\overleftarrow{\Phi}_{j+1}^{\bm V}+\overrightarrow{\Phi}_j^{\bm V}$ is a consistent approximation of $\frac{\partial \bm E}{\partial x}(x_\jph)$ and defined as
\begin{equation}\label{2.18}
\overleftarrow{\Phi}_{j+1}^{\bm V}=(\widetilde J(\bm V_\jph))^-\delta_\jph^-{\bm E},\quad \overrightarrow{\Phi}_j^{\bm V}=(\widetilde J(\bm V_\jph))^+\delta_\jph^+{\bm E},
\end{equation}
where $\widetilde J(\bm V_j)^\pm$ is the ``sign'' of the Jacobian defined as follows:
\begin{equation}\label{J_sign}
\widetilde J(\bm V_\jph)^\pm=\Upsilon_\jph\begin{pmatrix}\frac{\lambda_1^\pm}{\lambda_1}&0\\0&\frac{\lambda_2^\pm}{\lambda_2}\end{pmatrix}\Upsilon_\jph^{-1}
  :=\Upsilon_\jph\widetilde{\Lambda}^\pm\Upsilon_\jph^{-1},
\end{equation}
with $\lambda_1=u_\jph-\sqrt{\frac{\beta}{2}}(A_\jph)^{\frac{1}{4}}$, $\lambda_2=u_\jph+\sqrt{\frac{\beta}{2}}(A_\jph)^{\frac{1}{4}}$, and the corresponding eigenvector matrix   
\begin{equation}\label{R}
  \Upsilon_\jph=\left(
      \begin{array}{cc}
       -\vla{\sqrt{\frac{\rho AA_0}{\mathcal{K}\phi'\big(\frac{A}{A_0}\big)}}} & \vla{\sqrt{\frac{\rho AA_0}{\mathcal{K}\phi'\big(\frac{A}{A_0}\big)}}} \\
         1 & 1 
      \end{array}
    \right).
\end{equation}
In order to complete the computation in \eref{2.18}, we still need to compute the left- and right-biased finite difference (FD) approximations $\delta_\jph^\pm{\bm E}$. This can be achieved by taking the derivative of $R_{\bm E}$ in either \eref{2.10}, \eref{2.10a}, or \eref{2.11} at $x_\jph$ from the left- and right-side cells. Consequently, we obtain 
\begin{equation}\label{dE3}
\begin{aligned}
  &\delta_\jph^+ \bm E=\frac{1}{\dx}\big(\bm E_\jmh-4\bm E_j+3\bm E_\jph\big),\\ &\delta_\jph^- \bm E=\frac{1}{\dx}\big(-3\bm E_\jph+4\bm E_{j+1}-\bm E_{j+\frac{3}{2}}\big)
  \end{aligned}
\end{equation}
for the third-order scheme, 
  \begin{equation}\label{dE4}
  \begin{aligned}
   &\delta_\jph^+ \bm E=\frac{1}{\dx}\Big(-\bm E_\jmh+c_1\bm E_{j-\sqrt{\frac{1}{20}}}-c_2\bm E_{j+\sqrt{\frac{1}{20}}}+6\bm E_\jph\Big),\\
   &\delta_\jph^- \bm E=\frac{1}{\dx}\Big(-6\bm E_\jph+c_2\bm E_{j+1-\sqrt{\frac{1}{20}}}-c_1\bm E_{j+1+\sqrt{\frac{1}{20}}}+\bm E_{j+\frac{3}{2}}\Big),\\
   &\quad c_1=\frac{5(\sqrt{5}-1)}{2},\quad c_2=\frac{5(\sqrt{5}+1)}{2},
  \end{aligned}
\end{equation}
for the fourth-order scheme, and 
\begin{equation}\label{dE5}
\begin{aligned}
  &\delta_\jph^+ \bm E=\frac{1}{\dx}\Big(\bm E_\jmh+c_3\bm E_{j-\sqrt{\frac{3}{28}}}+\frac{16}{3}\bm E_j-c_4\bm E_{j+\sqrt{\frac{3}{28}}}+10\bm E_\jph\Big),\\
   &\delta_\jph^- \bm E=\frac{1}{\dx}\Big(-10\bm E_\jph+c_4\bm E_{j+1-\sqrt{\frac{3}{28}}}-\frac{16}{3}\bm E_{j+1}-c_3\bm E_{j+1+\sqrt{\frac{3}{28}}}-\bm E_{j+\frac{3}{2}}\Big),\\
   &\quad c_3=\frac{7(\sqrt{21}-7)}{6},\quad c_4=\frac{7(\sqrt{21}+7)}{6},
  \end{aligned}
\end{equation}
for the fifth-order scheme.

\vla{In what follows, we present a pseudocode algorithm for the one time step evolution of the proposed arbitrarily high-order WB scheme; see \cref{alg:HOWB}.}
\begin{algorithm}
\color{red}
\caption{\vla{One time step evolution of high-order WB scheme}}
\label{alg:HOWB}
\renewcommand{\algorithmicrequire}{\textbf{Input:}}
\renewcommand{\algorithmicensure}{\textbf{Output:}}
\begin{algorithmic}[1]
\REQUIRE{$\bm U_j^{(\ell)}$, $\bm V_{j\pm\frac{1}{2}}$ at time $t=t^n$ and spatially varying parameters $\mathcal{K}_j^{(\ell)}$, $(A_0)_j^{(\ell)}$, $(p_{\rm ext})_j^{(\ell)}$, $\mathcal{K}_{j\pm\frac{1}{2}}$, $(A_0)_{j\pm\frac{1}{2}}$, $(p_{\rm ext})_{j\pm\frac{1}{2}}$.}
\STATE{Compute $\bm U_{j\pm\frac{1}{2}}=\Psi^{-1}(\bm V_{j\pm\frac{1}{2}})$.}
\STATE{Construct the polynomial space $\widetilde{\bm U}$ using \eref{2.9}.}
\STATE{Compute $\bm U_k$ at the $r+1$ Gauss--Lobatto points using \eref{Um}.}
\STATE{Compute $\mathcal{K}_k$, $(A_0)_k$, and $(p_{\rm ext})_k$ using \eref{KA0pm}.}
\STATE{Compute $\bm E_k$ using \eref{Em} and \eref{KA0pm}.}
\STATE{Compute the right-hand-side terms of the semi-discrete ODEs \eref{2.13} and \eref{2.17}:
\begin{itemize}
  \item Compute $\bm F_{j\pm\frac{1}{2}}$ using \eref{2.14} and evaluate $\bm F_j^{(\ell)}+\bm S_j^{(\ell)}$ using \cref{alg:findl} and \eref{newsource}.
  \item Compute $\overleftarrow{\Phi}_{j+1}^{\bm V}+\overrightarrow{\Phi}_j^{\bm V}$ using \eref{2.18}, \eref{J_sign}, \eref{R}, and one of \eref{dE3}, or \eref{dE4}, or \eref{dE5}.
\end{itemize}
}
\STATE{Solve \eref{2.13} and \eref{2.17} simultaneously to update $\bm U_j^{(\ell)}$ and $\bm V_{j\pm\frac{1}{2}}$.}
\ENSURE{$\bm U_j^{(\ell)}$ and $\bm V_{j\pm\frac{1}{2}}$ at the new time level $t^{n+1}$.}
\end{algorithmic}
\end{algorithm}

\subsection{Well-balanced \vla{and exact conservation properties}}
In this subsection, we show the WB \vla{and exact conservation properties} of the proposed high-order scheme by proving the following propositions.
\begin{proposition}
In each cell $K_j$, consider the numerical initial data given by the moments 
\begin{equation*}
  \begin{aligned}
  A^{(\ell)}_j&\stackrel{\eref{2.6}}{=} C_\ell\int_{K_j}A(x)b_\ell(x-x_j){\rm d}x\approx C_\ell\sum_{k=1}^{r+1}\omega_kA_kb_\ell(X_k-x_j),\\
   Q^{(\ell)}_j&\stackrel{\eref{2.6}}{=} C_\ell\int_{K_j}Q(x)b_\ell(x-x_j){\rm d}x\approx C_\ell\sum_{k=1}^{r+1}\omega_kQ_kb_\ell(X_k-x_j),
  \end{aligned}\quad \ell\geq0,
\end{equation*}
and the point values
\begin{equation*}
  \begin{aligned}
 &A_\jmh=A_1,\quad A_\jph=A_{r+1},\\
 &u_\jmh=\frac{Q_1}{A_1},\quad u_\jph=\frac{Q_{r+1}}{A_{r+1}},
  \end{aligned}
\end{equation*}
where $(X_k,\omega_k)_{k=1}^{r+1}$ are the Gauss--Lobatto quadrature pairs in cell $K_j$ and $(A_k, Q_k)$ are the discrete initial point values of $\bm U$ at the $r+1$ Gauss--Lobatto points, i.e., $A_k\approx A(X_k)$ and $Q_k\approx Q(X_k)$, fulfill
\begin{align}\label{3.19}
  Q_k&=\rm{Const}, & E_k&=\vla{\frac{1}{2}\Big(\frac{Q_k}{A_k}\Big)^2+\frac{1}{\rho}\Bigg(\mathcal{K}_k\phi\Big(\frac{A_k}{(A_0)_k}\Big)+(p_{\rm ext})_k\Bigg)}= \rm{Const}.
\end{align}
Then, upon usage of the numerical method \vla{sketched in \cref{alg:HOWB}}, the moments $\bm U_j^{(\ell)}=(A_j^{(\ell)},Q_j^{(\ell)})$ and the point values $\bm V_{j\pm\frac{1}{2}}=(A_{j\pm\frac{1}{2}},u_{j\pm\frac{1}{2}})$ remain stationary. Namely, $\frac{{\rm d}}{{\rm d}t}\bm U_j^{(\ell)}=0$ and $\frac{{\rm d}}{{\rm d}t}\bm V_\jph=0$ hold at the steady state.
\end{proposition}
\begin{proof}
According to the definition of local reference \vla{steady-state solution} in \cref{sec24} and \eref{3.19}, we can find that $\bm U_k=\widehat{\bm U}_k$. Therefore, after we apply the Gauss--Lobatto quadrature to approximate the integrals in \eref{newsource}, the right-hand-side of \eref{2.13} reduces to
\begin{equation*}
  \bm F_j^{(\ell)}+\bm S_j^{(\ell)}=C_\ell\Big[\widehat{\bm F}_\jph b_\ell\big(\frac{\dx}{2}\big)-\widehat{\bm F}_\jmh b_\ell\big(-\frac{\dx}{2}\big)\Big]
\end{equation*}
Substituting this into \eref{2.13} and using the fact that $\bm U_\jmh=\bm U_1=\bm \widehat{\bm U}_1=\widehat{\bm U}_\jmh$ as well as $\bm U_\jph=\bm U_{r+1}=\widehat{\bm U}_{r+1}=\widehat{\bm U}_\jph$, we can obtain $\frac{{\rm d}}{{\rm d}t}\bm U_j^{(\ell)}=0$. Furthermore, since \eref{3.19} holds at all Gauss--Lobatto points, all the point values of the equilibrium variables $\bm E$ used in \cref{sec35} are constants, we easily get $\delta_\jph^{\pm}\bm E=0$. Therefore, $\overleftarrow{\Phi}_{j+1}^{\bm V}=0$ and $\overrightarrow{\Phi}_j^{\bm V}=0$, which implies that $\frac{{\rm d}}{{\rm d}t}\bm V_\jph=0$.
\end{proof}

\vla{
\begin{proposition}
When the space-varying parameters $\mathcal{K}$, $A_0$, and $p_{\rm ext}$ are constants, \eref{2.1} reduces to the following system of hyperbolic conservation laws:
\begin{equation}\label{eq:3.31}
 \left\{ \begin{aligned}
  &A_t+Q_x=0,\\
  &Q_t+\Big(\frac{Q^2}{A}+\frac{\mathcal{K}A_0}{\rho}\widetilde{\Phi}\big(\frac{A}{A_0}\big)\Big)_x=0.
  \end{aligned}\right.
\end{equation}
In this case, the conservative scheme for \eref{eq:3.31}, used to update cell averages (the zero-th moments), is recovered. Specifically, in the zero-th moment system derived from \eref{2.13},
\begin{equation*}
 \frac{{\rm d}\bm U^{(0)}_j}{{\rm d}t}+C_0\Big[\bm{F}_\jph -\bm{F}_\jmh \Big]={\bm F}^{(0)}_j+{\bm S}^{(0)}_j,
\end{equation*}
the term ${\bm F}^{(0)}_j+{\bm S}^{(0)}_j$ simplifies to zero. 
\end{proposition}

\begin{proof}
  From \eref{newsource}, we have 
  \begin{equation*}
    {\bm F}^{(0)}_j+{\bm S}^{(0)}_j=C_0(\widehat{\bm{F}}_\jph-\widehat{\bm{F}}_\jmh)
    +C_0\int\limits_{K_j}\big(\bm S(\bm U, x)-\bm S(\widehat{\bm U}, x)\big)\,{\rm d}x.
  \end{equation*}
Since $\bm S(\bm U, x)=0$ and $\bm S(\widehat{\bm U}, x)=0$ when $\mathcal{K}$, $A_0$, and $p_{\rm ext}$ are constants, it follows that the approximation term $\bm S_j^{(0)}=0$. From \eref{2.15bb}, we observe that $\widehat{\bm U}_k$ are constants if the local reference steady-state solutions exist or $\widehat{\bm U}_k=0$ if they do not. In either case, $\widehat{\bm F}_\jph=\widehat{\bm F}_\jmh$, which implies $\bm F_j^{(0)}=0$.  Consequently, ${\bm F}^{(0)}_j + {\bm S}^{(0)}_j = 0$.
\end{proof}
}

\begin{rmk}[Nonlinear stability]
\vla{Any high-order scheme is susceptible to the Gibbs phenomenon. In our case, this manifests itself as spurious oscillations in the numerical solutions near strong discontinuities. These oscillations can lead to non-physical values, such as negative cross-sectional area $A$ in blood flow simulations and make the code crash. Moreover, negative values of $A$ will make the eigenvalues in \eref{evalue} become imaginary, thus undermining the equations' hyperbolicity. To address these challenges, a stabilization technique is necessary. As} it was done in \cite[Section 2.1]{AL_SW}, we adopt the MOOD paradigm to detect the troubled cells and replace the high-order scheme with a lower-order scheme to recompute the solutions in the troubled cells. \vla{The MOOD criteria are identical to those in \cite[Section 2.1]{AL_SW} and are omitted here for brevity.} The lowest-order scheme is the first-order Local Lax--Friedrichs scheme, \vla{which is guaranteed not to produce any spurious oscillations. Its positivity-preserving property can be verified using a procedure similar to that in \cite[Theorem 2.2]{AL_SW}. To implement the MOOD loop, we must define an ordered list of schemes, which is referred to as the ``cascade''. In this work, we employ the following cascade}: 5th-order $\rightarrow$ 4th-order $\rightarrow$ 3rd-order $\rightarrow$ 1st-order. \vla{The MOOD procedure is illustrated in Figure \ref{MOOD}.} 
\begin{figure}[ht!]
\centerline{\includegraphics[trim=0.01cm 0.01cm 0.05cm 0.02cm,clip,width=13.0cm]{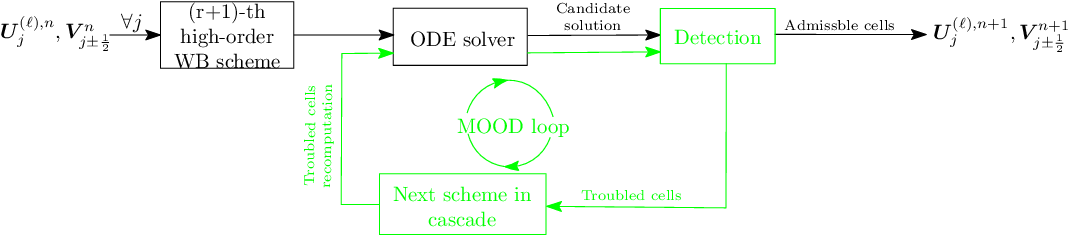}}
\caption{\sf Sketch of the hybrid finite element--finite volume scheme stabilized by a MOOD loop.\label{MOOD}}
\end{figure}
\end{rmk}

\section{Numerical Examples}\label{sec4}
In this section, we provide several numerical examples to demonstrate the performance of the proposed hybrid finite element-finite volume WB schemes of third-, fourth-, and fifth-order of accuracy. In all numerical experiments, unless specified otherwise, we use zero-order extrapolation boundary conditions, \vla{i.e., homogeneous Neumann boundary conditions,} and plot the cell average values in all Figures. We consider spatial discretizations of third, fourth, and fifth order of accuracy. The semi-discrete ODE systems are integrated using the three-stage third-order strong stability preserving (SSP) Runge--Kutta method with an adaptive time step computed at every time level based on CFL numbers of \bla{$0.4$, $0.2$, and $0.1$ for the third-, fourth-, and fifth-order schemes, respectively}. \vla{In certain examples, we compare the performance of the high-order well-balanced (WB) schemes with their non well-balanced (NWB) counterparts to explicitly highlight the importance and advantages of well-balancing. The non well-balanced schemes are constructed by setting the local reference steady-state solutions $\widehat{\bm U}=(\widehat A,\widehat Q)=\bm 0$ in \eref{newsource} at all Gauss--Lobatto points. Additionally, for some results, we report the discrete $L^1$- and $L^{\infty}$-errors between two vectors $\mu$ and $\nu$, defined as
\begin{equation*}
  \|\mu-\nu\|_{L^1}=\dx\sum_{j=1}^{N}\vert\mu_j-\nu_j\vert
\end{equation*}
and
\begin{equation*}
  \|\mu-\nu\|_{L^{\infty}}=\max_j \vert\mu_j-\nu_j\vert,
\end{equation*}
respectively.}

\vla{In Examples 1--7, we consider the simplest tube law in an artery with parameters $m=\frac{1}{2}$ and $n=0$ in \eref{phiamn} and \eref{Phiamn}. For other parameters in \eref{AQ}, we use $p_{\rm ext}={\rm Const}$, $\rho=1060$, and $\mathcal{K}=\frac{\kappa}{\sqrt{\pi}}\sqrt{A_0}$, where $\kappa=10^8$ and $A_0(x)$ is a function of $x$. In Examples 8 and 9, we take $\rho=1050$ and examine a general tube law with all space-varying parameters for arteries ($m=\frac{1}{2}$, $n=0$) and veins ($m=10$, $n=-\frac{3}{2}$), respectively.} 

\subsection*{Example 1---Accuracy Test}
In the first example, we evaluate the accuracy of the proposed high-order hybrid finite element-finite volume WB schemes on a problem with smooth solutions. To this end, we consider the initial conditions defined over the interval $[0, 10]$ as follows:
\begin{equation*}
  A(x,0)=\sin(0.2\pi x)+10,\quad Q(x,0)=e^{\cos(0.2\pi x)}, \quad A_0(x)=\frac{1}{2}\cos^2(0.2\pi x)+5.
\end{equation*} 
Periodic boundary conditions are employed and the numerical solutions are computed until the final time $t=0.01$. We measure the discrete $L^1$-errors and the experimental convergence rate using the Runge formulae, as detailed in \cite{CK23_blood}, and then report the obtained results in Table \ref{tab1}. It is evident from the findings that, for each $r$-degree polynomial space, an accuracy of order $(r+1)$ is achieved.

\begin{table}[!ht]
\caption{\sf Example 1: $L^1$-errors and experimental convergence rates computed with parabolic, cubic, and quartic polynomials.\label{tab1}}
\begin{center}
\begin{tabular}{c| c| c c c c c c}\hline
\multicolumn{1}{c|}{\multirow{2}{*}{Var.}} &\multicolumn{1}{c|}{\multirow{2}{*}{$N$}} &\multicolumn{2}{c}{$r=2$} &\multicolumn{2}{c}{$r=3$} &\multicolumn{2}{c}{$r=4$}\\ \cline{3-8}
 &  & $L^1$-error   & rate  & $L^1$-error  &rate  & $L^1$-error  & rate\\ \hline
 \multicolumn{1}{c|}{\multirow{4}{*}{$A$}} &40 &4.07e-03 & 2.88 & 8.22e-06 & 4.54 &5.44e-07 &5.04 \\ 
 &80 &4.63e-04 & 3.00  & 5.45e-07& 4.23&1.73e-08&5.01 \\
 &160 &5.60e-05 & 3.02 &3.53e-08 & 4.09&7.34e-10&4.78 \\
 &320 &6.98e-06 & 3.01 &2.21e-09 &4.05 &4.82e-11&4.37 \\ \hline
 \multicolumn{1}{c|}{\multirow{4}{*}{$Q$}} &40 &7.08e-01 & 2.60 & 3.54e-03& 4.02&1.14e-04&4.86 \\ 
 &80 &9.01e-02 & 2.77  &2.22e-04 &4.01 &3.26e-06&4.99 \\
 &160 &1.09e-02 & 2.90 & 1.37e-05&4.01 &9.99e-08&5.01 \\
 &320 &1.35e-03 & 2.95 &8.39e-07 &4.02 &3.24e-09&4.98 \\ \hline
\end{tabular}
\end{center}
\end{table}

\bla{Furthermore, Figure \ref{Ex1_fig1} shows the $L^1$-errors in $Q$ versus CPU times for the third-, fourth-, and fifth-order implementations of the proposed well-balanced schemes. The CPU times are that resulting from a sequence of successively refined meshes. It is seen that, for example, to achieve an acceptable error of $10^{-8}$, the fifth-order scheme is approximately $2.5$ times more efficient than the fourth-order scheme and nearly $10$ times more efficient than the third-order scheme. This leads to the following conclusion:
\begin{rmk}
On a given grid, high-order methods are slower than first-order ones, however, they also give better results. Their error also decreases much more quickly upon grid refinement. Therefore for a given error, they tend to outperform first-order method. The very fine grids that these latter need to be run on additionally entail small time steps due to the CFL condition.
\end{rmk}

\begin{figure}[ht!]
\centerline{\includegraphics[trim=0.02cm 0.05cm 0.2cm 0.2cm,clip,width=6.5cm]{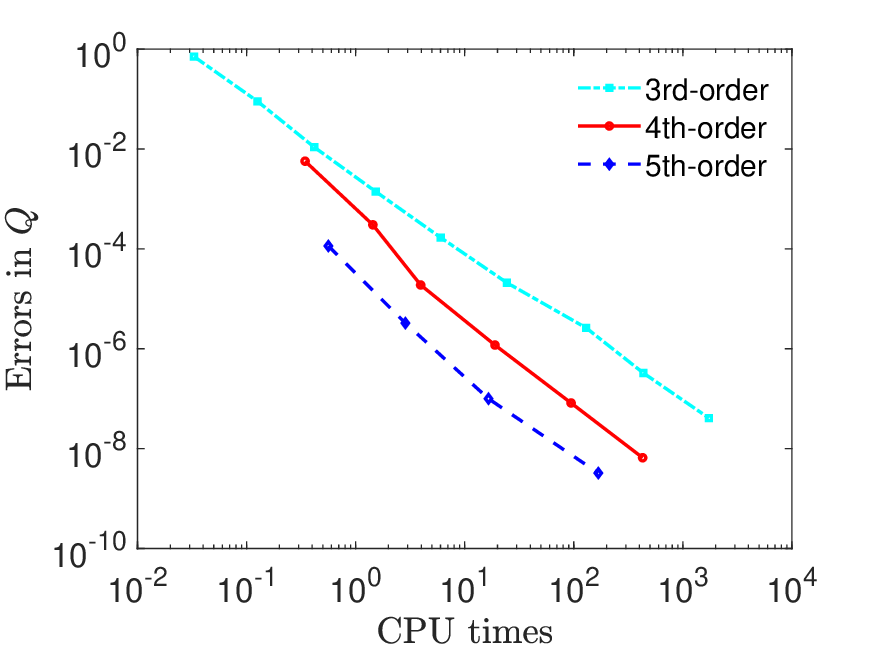}}
\caption{\sf Example 1: Errors versus computational CPU times.\label{Ex1_fig1}}
\end{figure}
}

\subsection*{Example 2---\vla{``blood-at-rest'' steady-state solution}}
In the second example, we showcase the capability of the proposed high-order WB scheme to exactly preserve the ``blood-at-rest'' steady-state solution, as denoted by \eref{s3}. To this end, we consider a cross-sectional area at rest given by:
\begin{equation*}\label{4.1}
  A_0(x)=\pi \big(R_0(x)\big)^2,
\end{equation*}
where
\begin{equation}\label{r01}
  R_0(x)=\left\{
  \begin{aligned}
  &\widetilde{R}, && x\in[0,x_1]\cup[x_4,L],\\
  &\widetilde{R}+\frac{\Delta R}{2}\left[\sin\Big(\frac{x-x_1}{x_2-x_1}\pi-\frac{\pi}{2}\Big)+1\right], &&x\in[x_1,x_2],\\
  &\widetilde{R}+\Delta R, && x\in[x_2,x_3],\\
  &\widetilde{R}+\frac{\Delta R}{2}\left[\cos\Big(\frac{x-x_3}{x_4-x_3}\pi\Big)+1\right], &&x\in[x_3,x_4],
  \end{aligned}\right.
\end{equation}
for the \vla{unloaded configuration where $p=p_{\rm ext}$}, and
\begin{equation}\label{r02}
 R_0(x)=\left\{
  \begin{aligned}
  &\widetilde{R}+\Delta R, && x\in[0,x_1]\cup[x_4,L],\\
  &\widetilde{R}-\frac{\Delta R}{2}\left[\sin\Big(\frac{x-x_1}{x_2-x_1}\pi-\frac{\pi}{2}\Big)-1\right], &&x\in[x_1,x_2],\\
  &\widetilde{R}, && x\in[x_2,x_3],\\
  &\widetilde{R}-\frac{\Delta R}{2}\left[\cos\Big(\frac{x-x_3}{x_4-x_3}\pi\Big)-1\right], &&x\in[x_3,x_4],
  \end{aligned}\right.
\end{equation}
\vla{for the loaded configuration.} The parameters in \eref{r02} are detailed in Table \ref{tab20}, while the initial conditions for the cross-sectional area and discharge are determined by the steady-state solution \eref{s3}, expressed as:
\begin{equation}\label{Ex2_IC}
  A(x,0)=\left\{\begin{aligned}
  &\pi \big(R_0(x)\big)^2,&&\mbox{\vla{for unloaded configuration}},\\
  &\big(0.001+\sqrt{\pi}R_0(x)\big)^2,&&\mbox{\vla{for loaded configuration}},
  \end{aligned}\right.
  \quad Q(x,0)\equiv0.
\end{equation}

\begin{table}[!ht]
\caption{\sf Example 2: parameters in the profile of $R_0(x)$.\label{tab20}}
\begin{center}
\begin{tabular}{c|c|c|c|c|c|c|c}
\hline
$r_0(x)$ & $\widetilde{R}$ & $\Delta R$ & $L$ & $x_1$ & $x_2$ & $x_3$ & $x_4$\\ \hline
\eref{r01}&\multirow{2}{*}{$0.004m$}& \multirow{2}{*}{$0.001m$}&\multirow{2}{*}{$0.14m$}& $0.01m$& $0.0305m$& $0.0495m$ & $0.07m$\\\cline{1-1} \cline{5-8}
\eref{r02} & & & & $0.0315m$& $0.035m$& $0.105m$ & $0.1085m$\\
\hline
\end{tabular}
\end{center}
\end{table}

The numerical solutions for the aforementioned two cases are computed using the proposed arbitrarily high-order schemes until reaching a final time $t=5$ on mesh size of $50$ uniform cells in the computational domain $[0,L]$. The errors in $A$ \vla{obtained by well-balanced and non well-balanced schemes} are reported in Tables \ref{tab21} and \ref{tab22} for unloaded and loaded configurations, respectively. These results clearly show that the developed arbitrarily high-order WB schemes are capable of exactly preserving the steady-state solutions at the discrete level, even with a coarse mesh resolution. \vla{However, it is obvious that the non well-balanced schemes fail to preserve these steady-state solutions.}

\begin{table}[!ht]
\color{red}
\caption{\sf Example 2: Errors in $A$ for the unloaded configuration and computed with well-balanced (WB) and non well-balanced (NWB) schemes using parabolic, cubic, and quartic polynomials.\label{tab21}}
\begin{center}
\begin{tabular}{c| c c c c c c}\hline
\multicolumn{1}{c|}{\multirow{2}{*}{Schs.}}  &\multicolumn{2}{c}{$r=2$} &\multicolumn{2}{c}{$r=3$} &\multicolumn{2}{c}{$r=4$}\\ \cline{2-7}
   & $L^1$-error   & $L^\infty$-error  & $L^1$-error  &$L^\infty$-error  & $L^1$-error  & $L^\infty$-error\\ \hline
 \multicolumn{1}{c|}{\multirow{1}{*}{WB}}  &1.16e-21 & 5.42e-20 &2.68e-21 &1.08e-19 &6.74e-17 & 5.62e-16 \\ \hline
 \multicolumn{1}{c|}{\multirow{1}{*}{NWB}}  &1.88e-12 & 4.61e-11 &1.12e-13 &1.91e-11 &7.37e-15 &1.77e-12 \\ \hline
\end{tabular}
\end{center}
\end{table}

\begin{table}[!ht]
\color{red}
\caption{\sf Example 2: Same as in Table \ref{tab21} but for \vla{the loaded configuration}.\label{tab22}}
\begin{center}
\begin{tabular}{c| c c c c c c}\hline
\multicolumn{1}{c|}{\multirow{2}{*}{Schs.}}  &\multicolumn{2}{c}{$r=2$} &\multicolumn{2}{c}{$r=3$} &\multicolumn{2}{c}{$r=4$}\\ \cline{2-7}
   & $L^1$-error   & $L^\infty$-error  & $L^1$-error  &$L^\infty$-error  & $L^1$-error  & $L^\infty$-error\\ \hline
 \multicolumn{1}{c|}{\multirow{1}{*}{WB}}  &4.17e-22 & 4.07e-20 &2.66e-22 &2.71e-20 &1.19e-16 &9,37e-16 \\ \hline
 \multicolumn{1}{c|}{\multirow{1}{*}{NWB}}  &2.41e-09 & 3.24e-08 &1.12e-10 &1.73e-08 &7.84e-11 &4.44e-09 \\ \hline
\end{tabular}
\end{center}
\end{table}

\subsection*{Example 3---Perturbation of a \vla{loaded} ``blood-at-rest'' \vla{steady-state solution}}
In the third example, we assess the ability of the proposed schemes to accurately track the propagation of an initial small perturbation from the non-zero pressure ``blood-at-rest'' steady state given in Example 2. Maintaining the same radius at rest and initial discharge as specified in \eref{r02} and \eref{Ex2_IC}, respectively, we introduce a small perturbation at the cross-sectional area given by:
\begin{equation*}
  A(x,0)=\left\{\begin{aligned}
  &A_{\rm eq}(x)\Big[1-10^{-3}\sin\Big(\frac{500}{7}\pi(x-0.063)\Big)\Big]^2&&\mbox{if}~x\in[0.063,0.077],\\
  &A_{\rm eq}(x)&&\mbox{otherwise},
  \end{aligned}\right.
\end{equation*} 
where $A_{\rm eq}(x)=\big(0.001+\sqrt{\pi}R_0(x)\big)^2$ with $R_0(x)$ defined as \eref{r02}.

We compute the solution until a final time $t=0.0016$ using two different mesh sizes, $N=50$ and $N=200$ uniform cells. The time snapshots of the variation $A(x,t)-A_{\rm eq}(x)$ are plotted in Figures \ref{Ex3_fig1}. \vla{For all WB methods, as} anticipated, the initial perturbation imposed at the center of the artery divides into two humps propagating in opposite directions \vla{and} the propagation process is accurately captured. When employing a coarse mesh with $N=50$ uniform cells, we observe that among the three studied \vla{WB} schemes with different orders of accuracy, the fifth-order scheme yields the most accurate results, with the fourth-order scheme outperforming the third-order scheme. \vla{However, the non well-balanced schemes generate relatively large spurious oscillations.} Upon refining the mesh to $N=200$ uniform cells, the results obtained by the six different schemes are nearly identical, with minor distinctions discernible only through local zooms. \vla{We observe that the non-physical oscillations produced by the non well-balanced schemes are effectively eliminated by mesh refinement; however, this approach incurs significantly higher computational costs. The results clearly demonstrate the superiority of the well-balanced schemes over their non-well-balanced counterparts.} \bla{Additionally, when comparing the performance of non well-balanced schemes on a coarse mesh, we find that high-order schemes also help mitigate non-physical waves, which aligns with the findings reported in \cite{Veiga_HOWB}.}    

\begin{figure}[ht!]
\color{red}
\centerline{\subfigure[WB, $N=50$]{\includegraphics[trim=0.8cm 0.25cm 1.1cm 0.2cm,clip,width=3.2cm]{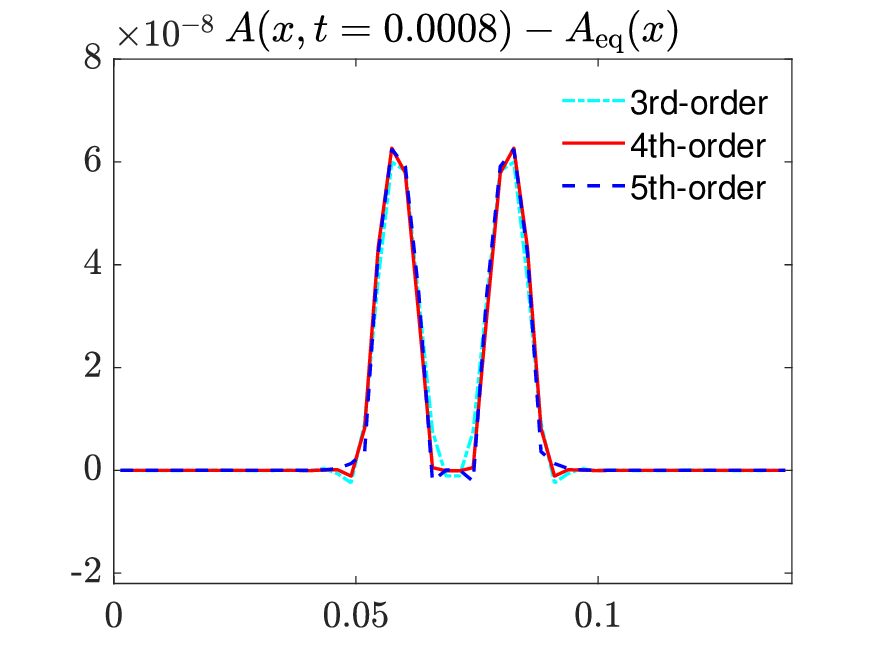}}\hspace*{0.005cm}
\subfigure[NWB, $N=50$]{\includegraphics[trim=0.8cm 0.25cm 1.1cm 0.2cm,clip,width=3.2cm]{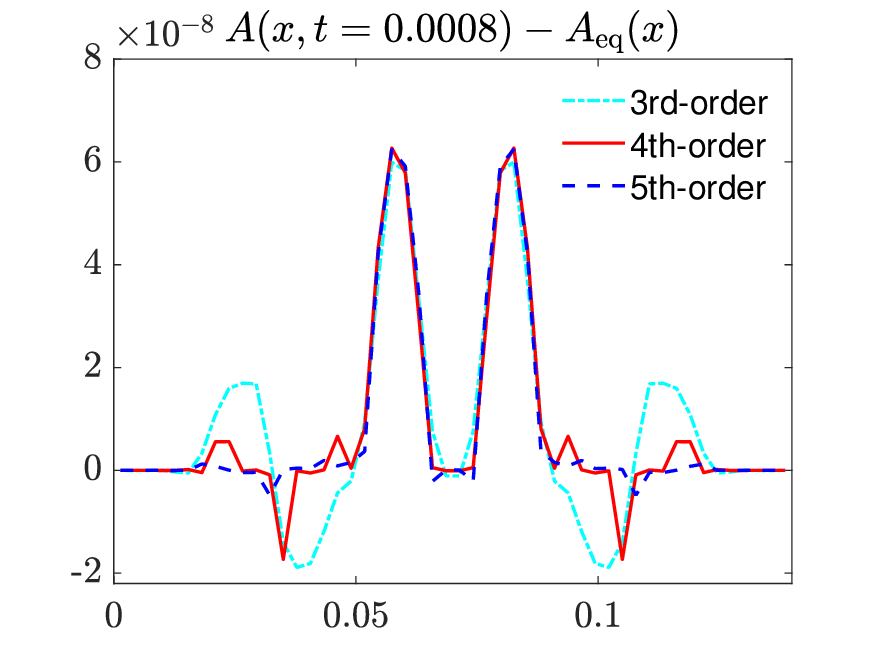}}\hspace*{0.005cm}
\subfigure[WB, $N=50$]{\includegraphics[trim=0.8cm 0.25cm 1.1cm 0.2cm,clip,width=3.2cm]{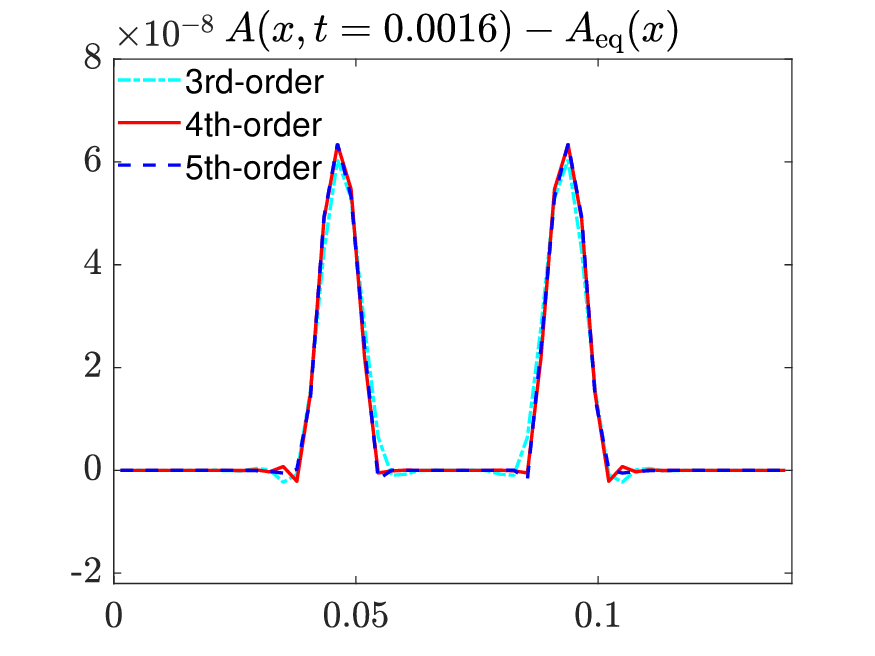}}\hspace*{0.005cm}
\subfigure[NWB, $N=50$]{\includegraphics[trim=0.8cm 0.25cm 1.1cm 0.2cm,clip,width=3.2cm]{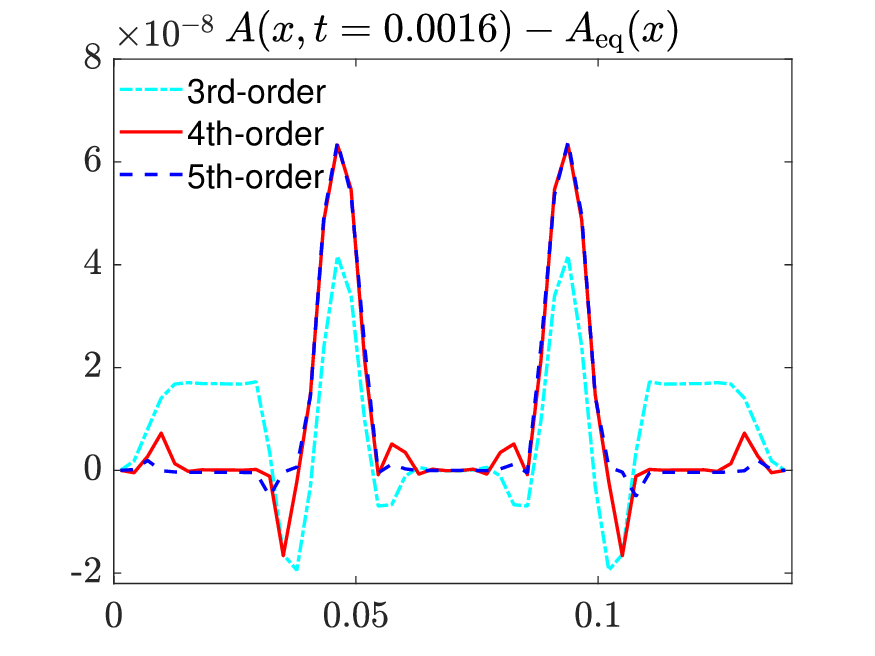}}}
\vskip5pt
\centerline{\subfigure[WB, $N=200$]{\includegraphics[trim=0.8cm 0.25cm 1.1cm 0.2cm,clip,width=3.2cm]{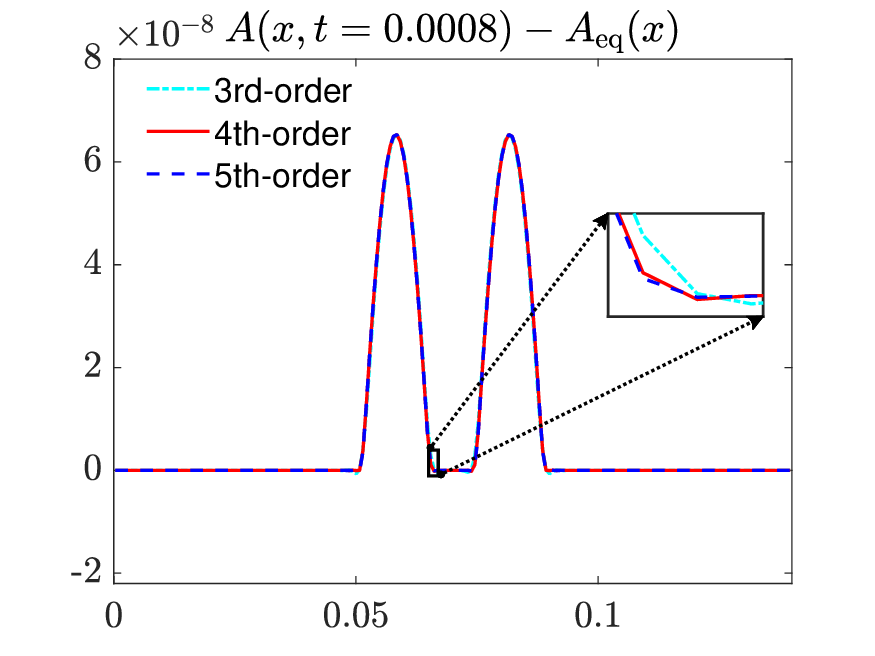}}\hspace*{0.005cm}
\subfigure[NWB, $N=200$]{\includegraphics[trim=0.8cm 0.25cm 1.1cm 0.2cm,clip,width=3.2cm]{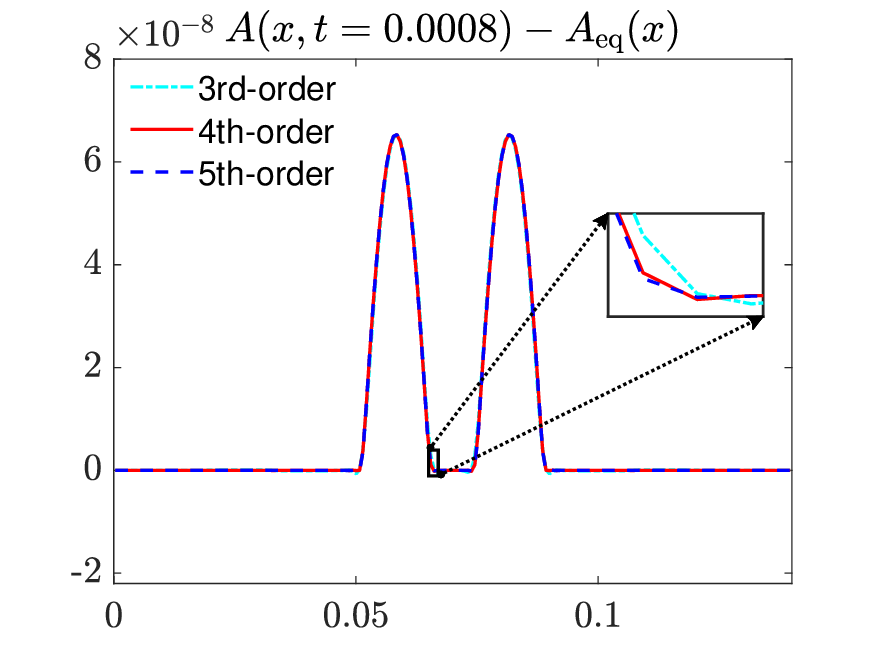}}\hspace*{0.005cm}
\subfigure[WB, $N=200$]{\includegraphics[trim=0.8cm 0.25cm 1.1cm 0.2cm,clip,width=3.2cm]{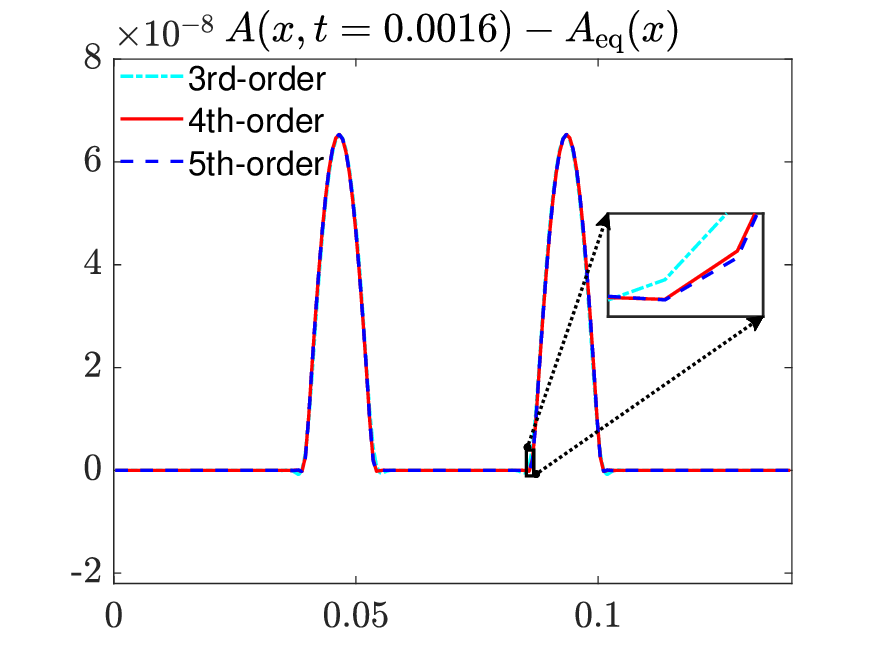}}\hspace*{0.005cm}
\subfigure[NWB, $N=200$]{\includegraphics[trim=0.8cm 0.25cm 1.1cm 0.2cm,clip,width=3.2cm]{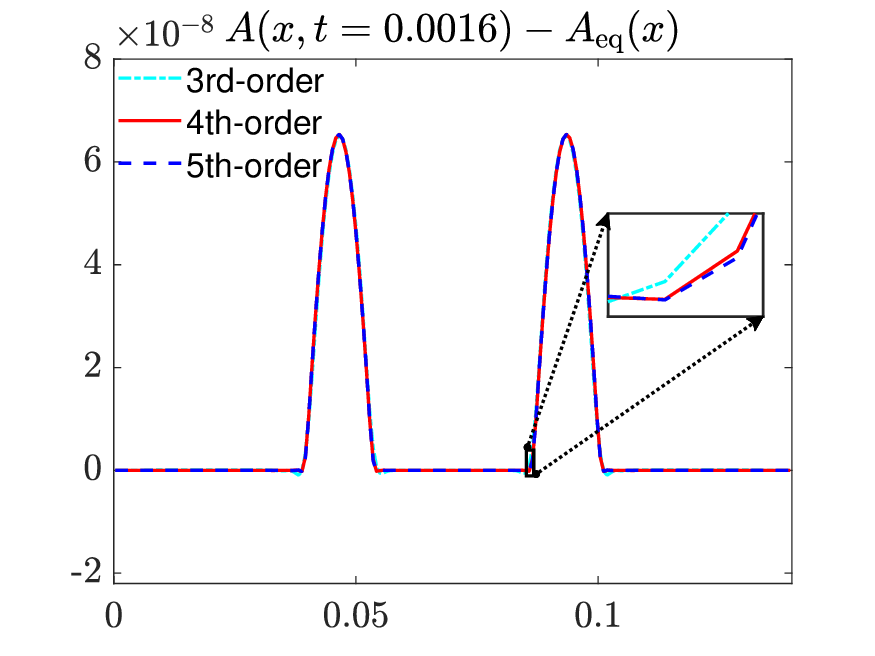}}}
\caption{\sf Example 3: Time snapshots ($t=0.0008$ and $t=0.0016$) of the difference $A(x,t)-A_{\rm eq}(x)$ computed by WB and non WB schemes.\label{Ex3_fig1}}
\end{figure}

\subsection*{Example 4---Non-zero-velocity ``moving-blood'' steady states}
In the fourth example, we demonstrate how the proposed hybrid WB schemes preserve the non-zero steady state \eref{s1}. We examine three distinct cases corresponding to the physiological conditions of an aneurysm, a stenosis, and a decreasing step. The initial conditions for each case are determined from the steady-state condition \eref{s1}, which is expressed as
\begin{equation}\label{Ex4_IC}
  Q_{\rm s}=Q_{\rm in},\quad E_{\rm s}=\frac{(Q_{\rm in})^2}{2(A_{\rm out})^2}+\beta(\sqrt{A_{\rm out}}-\sqrt{A_0(L)}),
\end{equation} 
where the subscripts ``in'' and ``out'' denote values at the inlet (left side) and outlet (right side) of the domain, respectively, with $L$ representing the length of the artery (set to $0.16$ in this example). In \eref{Ex4_IC}, the values of $A_{\rm in}$, $A_{\rm out}$, and $Q_{\rm in}$ are given by
\begin{equation*}
  A_{\rm in}=A_0(0)[1+S_{\rm in}]^2,\quad A_{\rm out}=A_0(L)[1+S_{\rm in}]^2,\quad Q_{\rm in}=A_{\rm in}S_{\rm in}C_{\rm in},
\end{equation*} 
where $S_{\rm in}$ is the Shapiro number at the inlet, equivalent to the Froude number for the shallow water equations. For all test cases considered here, we set $S_{\rm in}=\{0.5,0.1,0.01\}$, given that blood flow is typically subcritical. Additionally, $C_{\rm in}$ denotes the Moens-Korteweg velocity at the inlet, defined as
\begin{equation*}
  C_{\rm in}=\sqrt{\frac{\kappa\sqrt{A_{\rm in}}}{2\rho\sqrt{\pi}}}.
\end{equation*}


\subsubsection*{Test 1: An aneurysm}
In this test, we simulate an artery with an aneurysm by defining the cross-sectional radius at rest (see Figure \ref{Ex4_fig1}, left) as 
\begin{equation*}
  R_0(x)=\left\{
  \begin{aligned}
  &R_{\rm in}, && x\in[0,x_1]\cup[x_4,L],\\
  &R_{\rm in}+\frac{\Delta R}{2}\left[1-\cos\Big(\frac{x-x_1}{x_2-x_1}\pi\Big)\right], &&x\in[x_1,x_2],\\
  &R_{\rm in}+\Delta R, && x\in[x_2,x_3],\\
  &R_{\rm in}+\frac{\Delta R}{2}\left[1+\cos\Big(\frac{x-x_3}{x_4-x_3}\pi\Big)\right], &&x\in[x_3,x_4],
  \end{aligned}\right.
\end{equation*}
where $R_{\rm in}=4\times10^{-3}$, $\Delta R=1\times10^{-3}$, $x_1=\frac{9L}{40}$, $x_2=\frac{L}{4}$, $x_3=\frac{3L}{4}$, $x_4=\frac{31L}{40}$, and the cross-sectional area at rest is given by $A_0(x)=\pi(R_0(x))^2$. The initial profile of cross-sectional area $A(x,0)$ is determined by solving \eref{Ex4_IC}. With the obtained $A(x,0)$ and $Q(x,0)=Q_{\rm in}$, we run the simulation on a uniform mesh of $50$ cells until a final time $t=5$. The $L^1$- and $L^{\infty}$-errors \vla{in $A$, computed using both well-balanced and non well-balanced schemes,} are shown in Table \ref{tab41}. \vla{The results clearly indicate that only the errors computed by the well-balanced schemes are within machine accuracy}, confirming the well-balanced property for ``moving-blood'' flow with an aneurysm is successfully maintained \vla{by the well-balanced schemes, whereas it is not preserved by the non well-balanced schemes.} 

\begin{table}[!ht]
\color{red}
\caption{\sf Example 4---Test 1: $L^1$- and $L^\infty$-errors in $A$ computed with well-balanced and non well-balanced schemes using parabolic, cubic, and quartic polynomials for the steady state in an aneurysm.\label{tab41}}
\begin{center}
\begin{tabular}{c| c| c c c c c c}\hline
\multicolumn{1}{c|}{\multirow{2}{*}{$S_{\rm in}$}} &\multicolumn{1}{c|}{\multirow{2}{*}{Schs.}}  &\multicolumn{2}{c}{$r=2$} &\multicolumn{2}{c}{$r=3$} &\multicolumn{2}{c}{$r=4$}\\ \cline{3-8}
 & & $L^1$-error   & $L^\infty$-error  & $L^1$-error  &$L^\infty$-error  & $L^1$-error  & $L^\infty$-error\\ \hline
{\multirow{2}{*}{$0.5$}}& \multicolumn{1}{c|}{\multirow{1}{*}{WB}}  &3.90e-22 & 4.07e-20 &2.26e-21 &2.71e-20 &2.66e-16 & 2.01e-15 \\ \cline{2-8}
& \multicolumn{1}{c|}{\multirow{1}{*}{NWB}}  &2.02e-10 & 2.52e-09 &2.84e-10 &2.04e-08 &2.77e-10 &7.86e-09 \\ \hline
{\multirow{2}{*}{$0.1$}}& \multicolumn{1}{c|}{\multirow{1}{*}{WB}}  &1.28e-21 & 2.71e-20 &1.04e-21 &1.36e-20 &7.72e-17 &5.31e-16 \\ \cline{2-8}
& \multicolumn{1}{c|}{\multirow{1}{*}{NWB}}  &3.89e-09 & 4.57e-08 &1.05e-10 &1.47e-08 &6.11e-11 &4.34e-09 \\ \hline
{\multirow{2}{*}{$0.01$}}& \multicolumn{1}{c|}{\multirow{1}{*}{WB}}  &5.42e-22 & 2.71e-20 &1.08e-22 &1.36e-20 &6.85e-17 &4.60e-16 \\ \cline{2-8}
& \multicolumn{1}{c|}{\multirow{1}{*}{NWB}}  &4.32e-09 & 5.08e-08 &1.04e-10 &1.44e-08 &5.48e-11 &4.22e-09 \\ \hline
\end{tabular}
\end{center}
\end{table}

\subsubsection*{Test 2: A Stenosis}
In this test, we simulate an artery with a stenosis by defining the cross-sectional radius at rest (see Figure \ref{Ex4_fig1}, middle) as 
\begin{equation*}
  R_0(x)=\left\{
  \begin{aligned}
  &R_{\rm in}, && x\in[0,x_1]\cup[x_2,L],\\
  &R_{\rm in}-\frac{\Delta R}{4}\bigg[1-\cos\Big(2\pi\frac{x-x_1}{x_2-x_1}\Big)\bigg], &&x\in[x_1,x_2],
  \end{aligned}\right.
\end{equation*}
where $x_1=\frac{3L}{10}$ and $x_2=\frac{7L}{10}$. We repeat the simulation as in Test 1, and show the $L^1$- and $L^\infty$-errors in Table \ref{tab42}, which validates the well-balanced property \vla{of the proposed high-order schemes} for ``moving-blood'' flow with a stenosis.

\begin{table}[!ht]
\color{red}
\caption{\sf Example 4---Test 2: Same as in Table \ref{tab41} but for the flow in a stenosis.\label{tab42}}
\begin{center}
\begin{tabular}{c| c| c c c c c c}\hline
\multicolumn{1}{c|}{\multirow{2}{*}{$S_{\rm in}$}} &\multicolumn{1}{c|}{\multirow{2}{*}{Schs.}}  &\multicolumn{2}{c}{$r=2$} &\multicolumn{2}{c}{$r=3$} &\multicolumn{2}{c}{$r=4$}\\ \cline{3-8}
 & & $L^1$-error   & $L^\infty$-error  & $L^1$-error  &$L^\infty$-error  & $L^1$-error  & $L^\infty$-error\\ \hline
{\multirow{2}{*}{$0.5$}}& \multicolumn{1}{c|}{\multirow{1}{*}{WB}}  &1.04e-21 & 5.42e-20 &7.81e-22 &4.07e-20 &1.80e-16 &1.30e-15 \\ \cline{2-8}
& \multicolumn{1}{c|}{\multirow{1}{*}{NWB}}  &1.93e-13 & 1.00e-11 &3.11e-15 &7.94e-14 &1.82e-16 &1.32e-15 \\ \hline
{\multirow{2}{*}{$0.1$}}& \multicolumn{1}{c|}{\multirow{1}{*}{WB}}  &7.16e-22 & 2.03e-20 &4.55e-22 &2.03e-20 &6.25e-17 &4.13e-16 \\ \cline{2-8}
& \multicolumn{1}{c|}{\multirow{1}{*}{NWB}}  &4.65e-14 & 1.43e-12 &1.04e-15 &3.46e-14 &6.28e-17 &4.21e-16 \\ \hline
{\multirow{2}{*}{$0.01$}}& \multicolumn{1}{c|}{\multirow{1}{*}{WB}}  &3.86e-21 & 1.56e-19 &3.04e-21 &4.07e-20 &5.36e-17 &3.49e-16 \\ \cline{2-8}
& \multicolumn{1}{c|}{\multirow{1}{*}{NWB}}  &5.18e-14 & 1.59e-12 &1.01e-15 &3.38e-14 &5.50e-17 &3.72e-16 \\ \hline
\end{tabular}
\end{center}
\end{table}

\subsubsection*{Test 3: A decreasing step}
In this test, we examine a scenario where the artery's radius undergoes an instantaneous reduction, representing the transition from a parent to a daughter artery. This idealized transition results in a sudden change in the cross-sectional radius at a specific location. The radius at rest (see Figure \ref{Ex4_fig1}, right) is defined as 
\begin{equation*}
  R_0(x)=\left\{
  \begin{aligned}
  &R_{\rm in}, && x<\frac{L}{2},\\
  &R_{\rm in}-\frac{\Delta R}{2}, &&x\geq\frac{L}{2}.
  \end{aligned}\right.
\end{equation*}
Following the procedures of Tests 1 and 2, we compute the numerical solution up to a final time $t=5$ and present the $L^1$ and $L^\infty$ errors in Table \ref{tab43}. These results confirm the exact preservation of the ``moving-blood'' flow during the decreasing step.  

\begin{table}[!ht]
\color{red}
\caption{\sf Example 4---Test 3: Same as in Tables \ref{tab41} and \ref{tab42} but for the flow in a decreasing step.\label{tab43}}
\begin{center}
\begin{tabular}{c| c| c c c c c c}\hline
\multicolumn{1}{c|}{\multirow{2}{*}{$S_{\rm in}$}} &\multicolumn{1}{c|}{\multirow{2}{*}{Schs.}}  &\multicolumn{2}{c}{$r=2$} &\multicolumn{2}{c}{$r=3$} &\multicolumn{2}{c}{$r=4$}\\ \cline{3-8}
 & & $L^1$-error   & $L^\infty$-error  & $L^1$-error  &$L^\infty$-error  & $L^1$-error  & $L^\infty$-error\\ \hline
{\multirow{2}{*}{$0.5$}}& \multicolumn{1}{c|}{\multirow{1}{*}{WB}}  &1.82e-21 & 4.07e-20 &0 &0 &2.67e-16 &1.97e-15 \\ \cline{2-8}
& \multicolumn{1}{c|}{\multirow{1}{*}{NWB}}  &8.49e-09 & 8.20e-08 &1.28e-08 &3.51e-07 &1.50e-08 &2.79e-07 \\ \hline
{\multirow{2}{*}{$0.1$}}& \multicolumn{1}{c|}{\multirow{1}{*}{WB}}  &6.51e-23 & 1.36e-20 &0 &0 &6.40e-17 &4.36e-16 \\ \cline{2-8}
& \multicolumn{1}{c|}{\multirow{1}{*}{NWB}}  &2.19e-09 & 1.52e-08 &1.27e-09 &1.22e-07 &3.65e-10 &7.86e-08 \\ \hline
{\multirow{2}{*}{$0.01$}}& \multicolumn{1}{c|}{\multirow{1}{*}{WB}}  &0 & 0 &2.17e-23 &6.78e-21 &6.49e-17 &4.42e-16 \\ \cline{2-8}
& \multicolumn{1}{c|}{\multirow{1}{*}{NWB}}  &2.48e-09 & 1.61e-08 &1.57e-09 &1.17e-07 &4.62e-10 &7.52e-08 \\ \hline
\end{tabular}
\end{center}
\end{table}

\begin{figure}[ht!]
\centerline{\includegraphics[trim=0.8cm 0.25cm 1.1cm 0.2cm,clip,width=4.0cm]{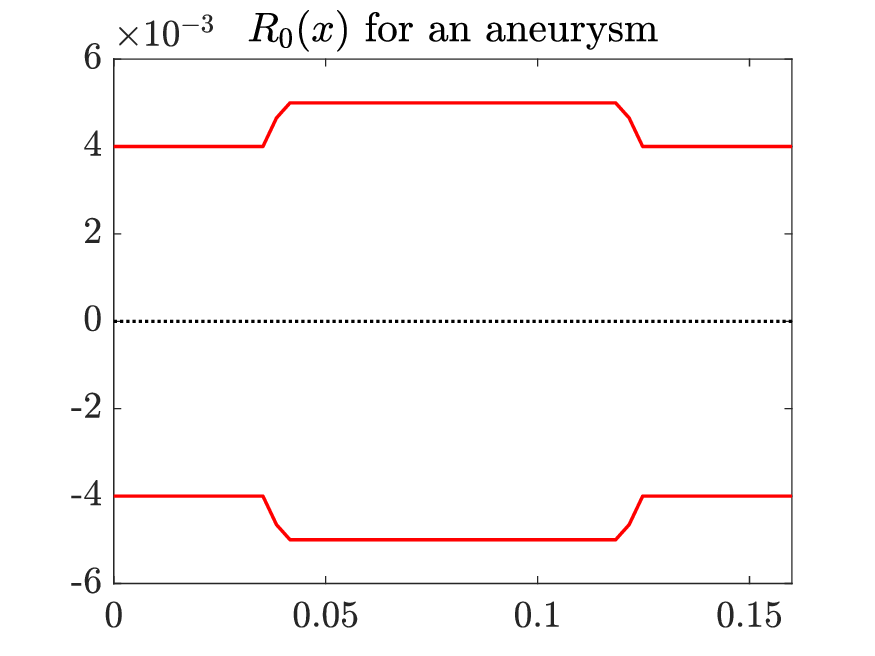}\hspace*{0.15cm}
\includegraphics[trim=0.8cm 0.25cm 1.1cm 0.2cm,clip,width=4.0cm]{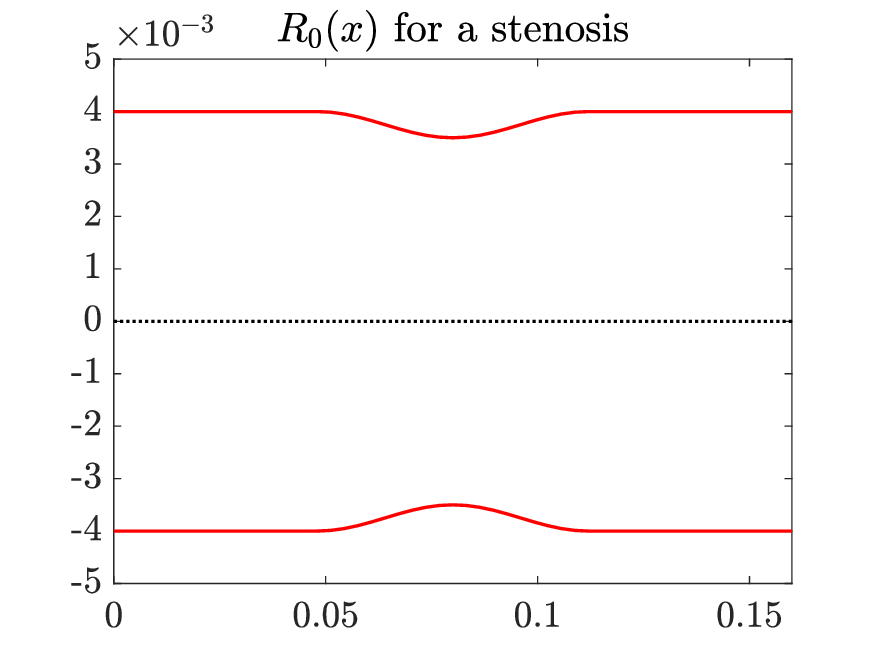}\hspace*{0.15cm}
\includegraphics[trim=0.8cm 0.25cm 1.1cm 0.2cm,clip,width=4.0cm]{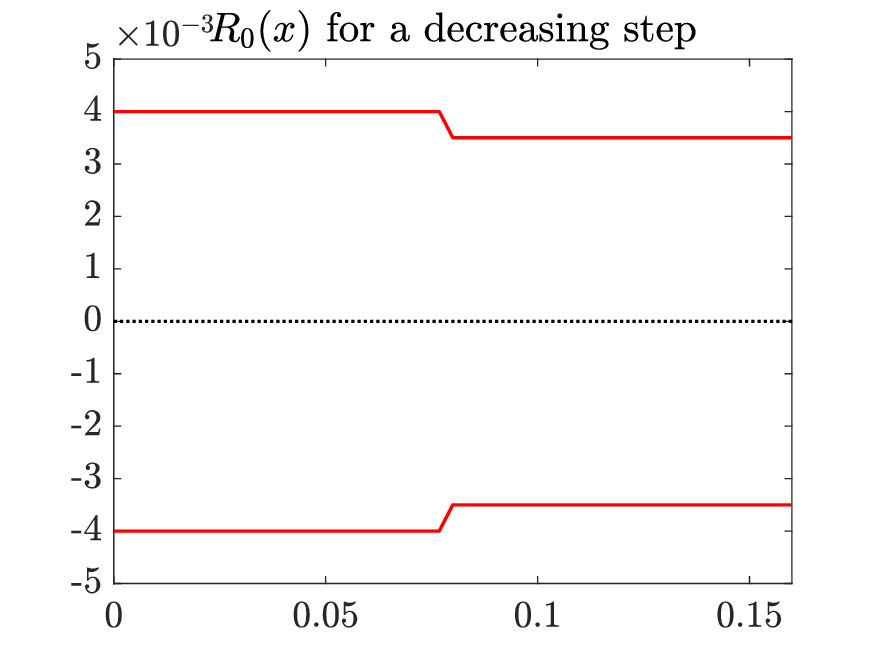}}
\caption{\sf Example 4: Radii at rest for the artery with an aneurysm (left), a stenosis (middle), and a decreasing step (right). \label{Ex4_fig1}}
\end{figure}

\subsection*{Example 5---Small perturbation of a ``moving-blood'' \vla{steady-state solution} in an aneurysm}
In the fifth example, we demonstrate that the ability of the proposed arbitrarily high-order schemes in handling small perturbation to ``moving-blood'' steady state flow in an aneurysm. The background steady-state solution is given in Example 4, Test 1. We will denote the steady-state cross-sectional area as $A_{\rm eq}(x)$ and the steady-state discharge as $Q_{\rm eq}(x)$ here. We then introduce the perturbed initial conditions as follows:
\begin{equation*}
  A(x,0)=A_{\rm eq}(x)+\left\{\begin{aligned}
  &\varepsilon\pi\cos^2\Big(\frac{125\pi}{2}x\Big) &&\mbox{if}~ x\in[0.072, 0.088],\\
  &0 &&\mbox{otherwise},
  \end{aligned}\right. \quad Q(x,0)=Q_{\rm eq}(x),
\end{equation*}
defined within a computational domain $[0,0.16]$. Here, $\varepsilon=2.5\times10^{-9}$ is a small magnitude number.

We compute the numerical solution by the third-, fourth-, and fifth-order schemes with either 50 or 200 uniform cells. The time snapshots of solution at $t=0.0025$ and $t=0.005$ are presented in Figures \ref{Ex5_fig1}, \ref{Ex5_fig2}, and \ref{Ex5_fig3}, corresponding to $S_{\rm in}=0.5$, $0.1$, and $0.01$, respectively. As one can see, the proposed arbitrarily high-order well-balanced schemes accurately capture the perturbations, consistent with results reported in \cite{BX_blood}. Differences between the schemes are more pronounced with coarse mesh resolutions but diminish as the mesh is refined. \vla{In contrast, when examining the results obtained from the non well-balanced schemes, it is evident that in the latter two cases with lower Shapiro numbers, the third-order non well-balanced scheme fails to accurately resolve small perturbations, even with a much finer mesh.}

\begin{figure}[ht!]
\color{red}
\centerline{\subfigure[WB, $N=50$]{\includegraphics[trim=0.8cm 0.25cm 1.1cm 0.2cm,clip,width=3.2cm]{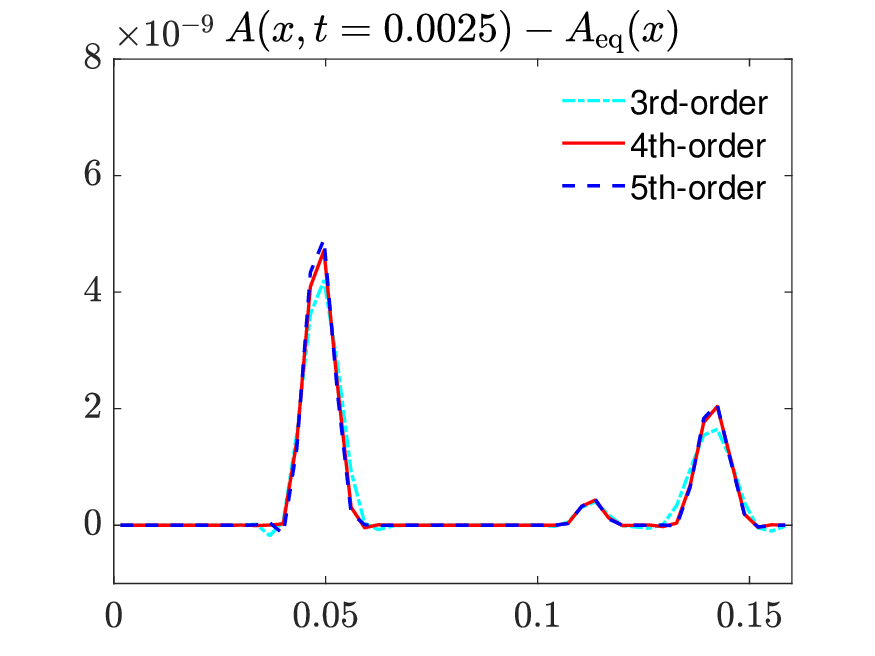}}\hspace*{0.005cm}
\subfigure[NWB, $N=50$]{\includegraphics[trim=0.8cm 0.25cm 1.1cm 0.2cm,clip,width=3.2cm]{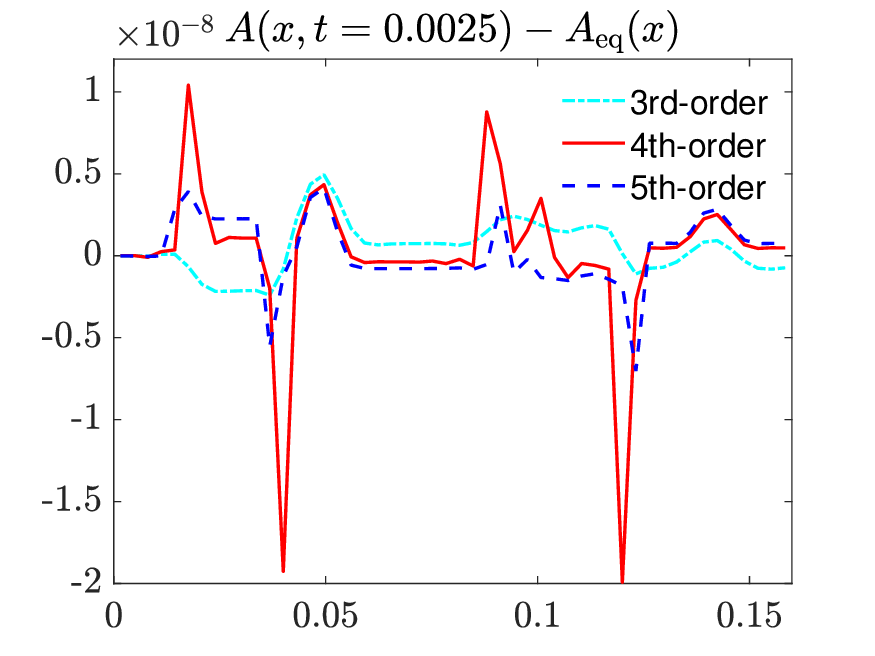}}\hspace*{0.005cm}
\subfigure[WB, $N=50$]{\includegraphics[trim=0.8cm 0.25cm 1.1cm 0.2cm,clip,width=3.2cm]{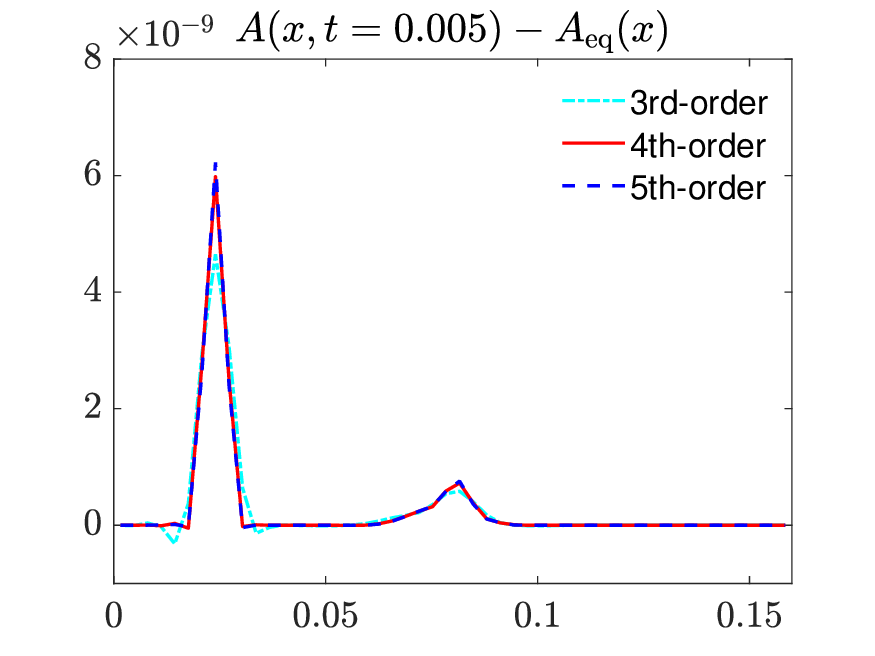}}\hspace*{0.005cm}
\subfigure[NWB, $N=50$]{\includegraphics[trim=0.8cm 0.25cm 1.1cm 0.2cm,clip,width=3.2cm]{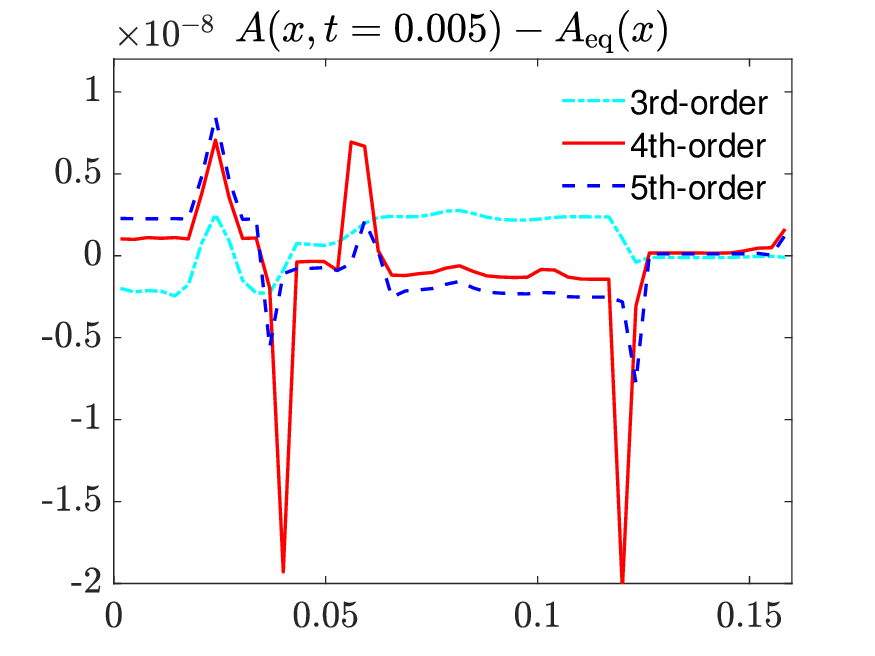}}}
\vskip5pt
\centerline{\subfigure[WB, $N=200$]{\includegraphics[trim=0.8cm 0.25cm 1.1cm 0.2cm,clip,width=3.2cm]{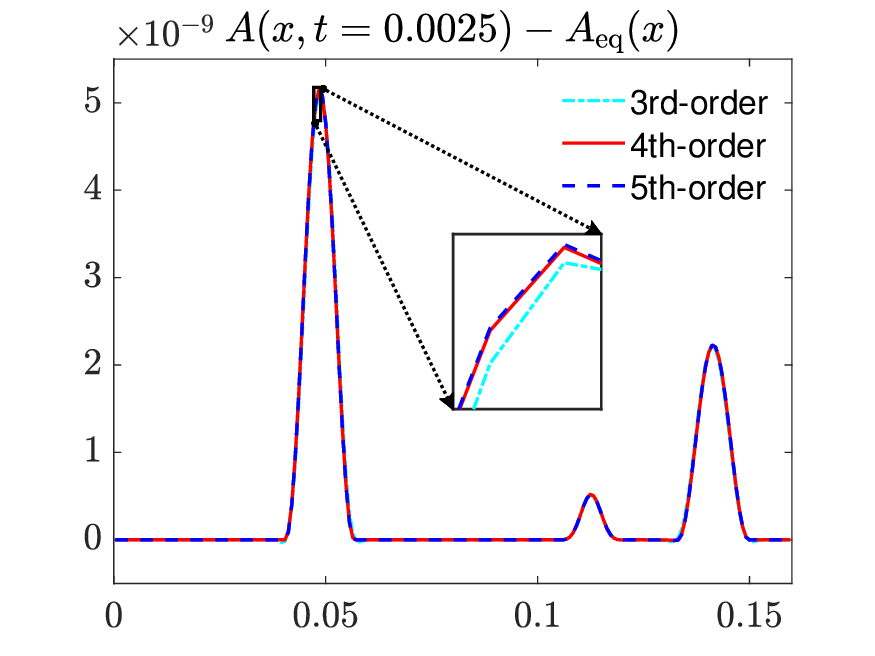}}\hspace*{0.005cm}
\subfigure[NWB, $N=200$]{\includegraphics[trim=0.8cm 0.25cm 1.1cm 0.2cm,clip,width=3.2cm]{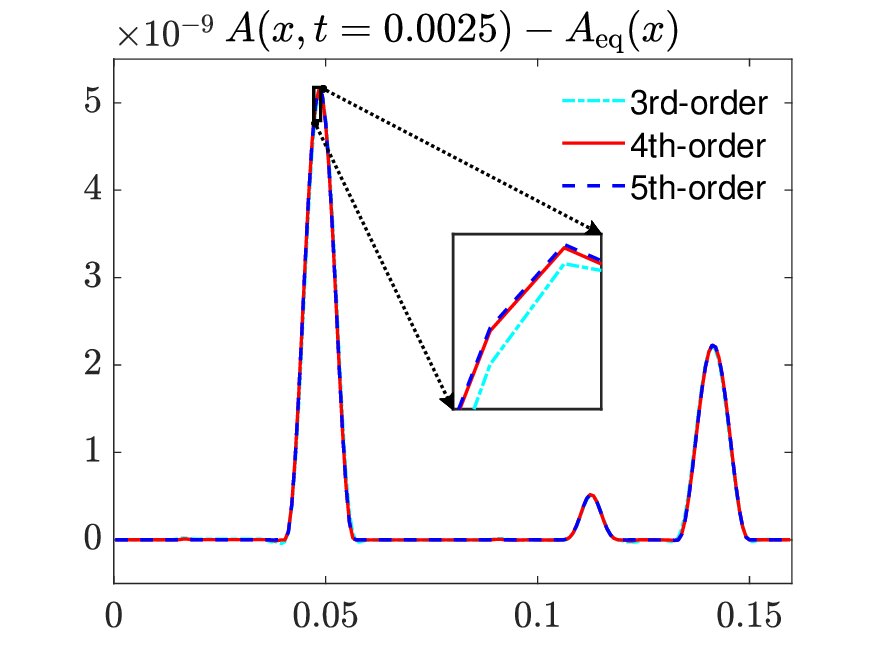}}\hspace*{0.005cm}
\subfigure[WB, $N=200$]{\includegraphics[trim=0.8cm 0.25cm 1.1cm 0.2cm,clip,width=3.2cm]{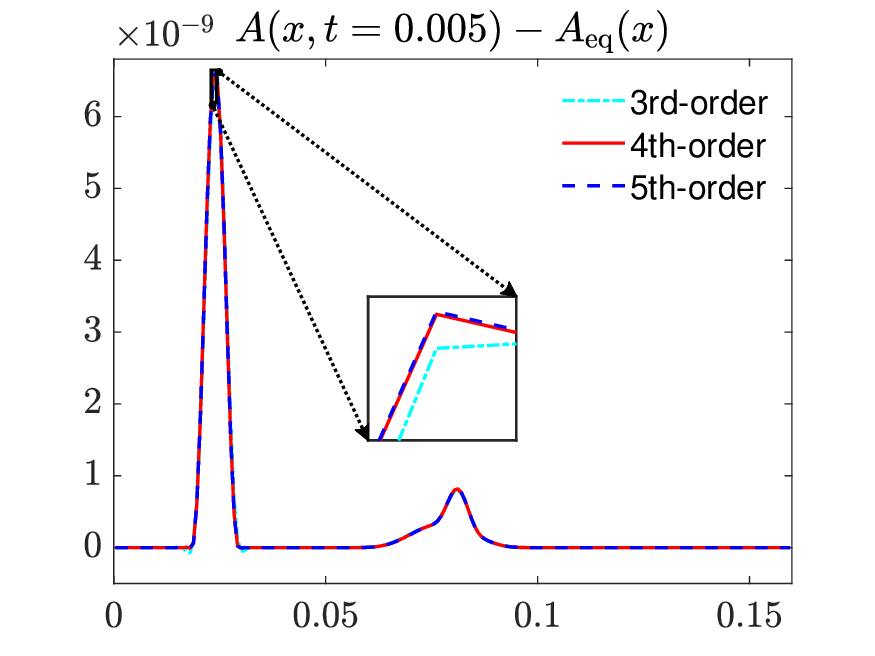}}\hspace*{0.005cm}
\subfigure[NWB, $N=200$]{\includegraphics[trim=0.8cm 0.25cm 1.1cm 0.2cm,clip,width=3.2cm]{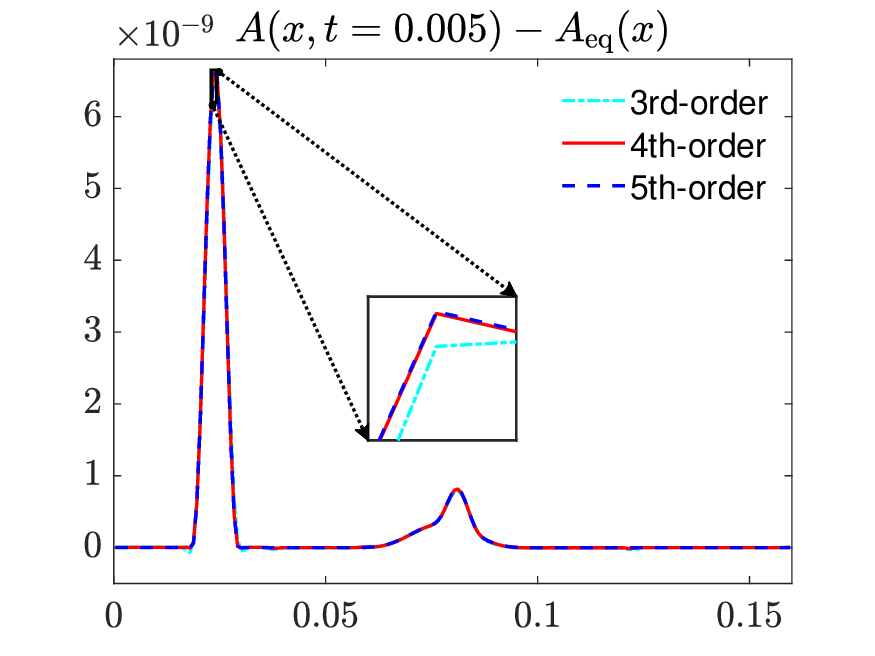}}}
\caption{\sf Example 5: Time snapshots ($t=0.0025$ and $t=0.005$) of the difference $A(x,t)-A_{\rm eq}(x)$ for $S_{\rm in}=0.5$ computed by WB and non WB schemes.\label{Ex5_fig1}}
\end{figure}

\begin{figure}[ht!]
\color{red}
\centerline{\subfigure[WB, $N=50$]{\includegraphics[trim=0.8cm 0.25cm 1.1cm 0.2cm,clip,width=3.2cm]{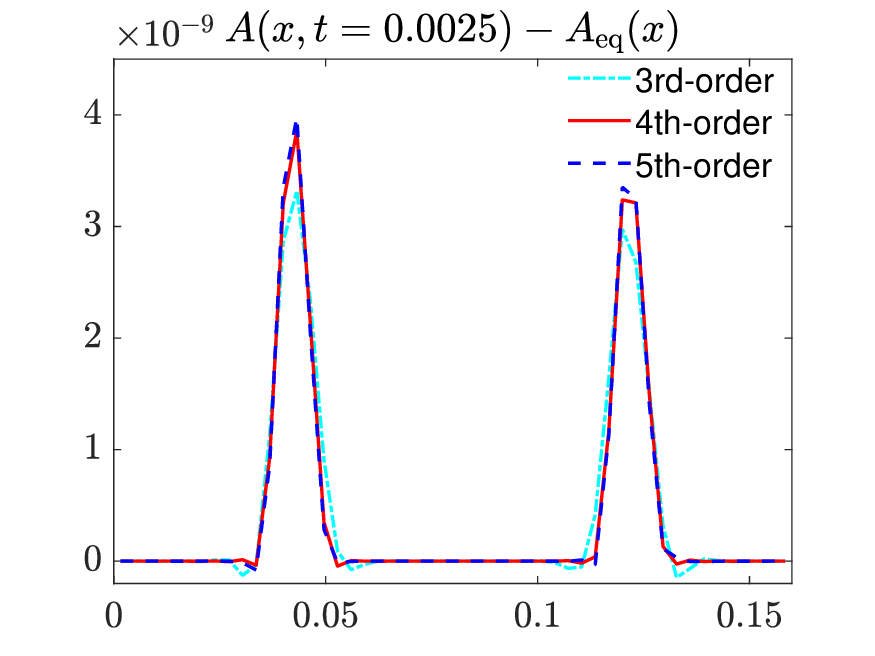}}\hspace*{0.005cm}
\subfigure[NWB, $N=50$]{\includegraphics[trim=0.8cm 0.25cm 1.1cm 0.2cm,clip,width=3.2cm]{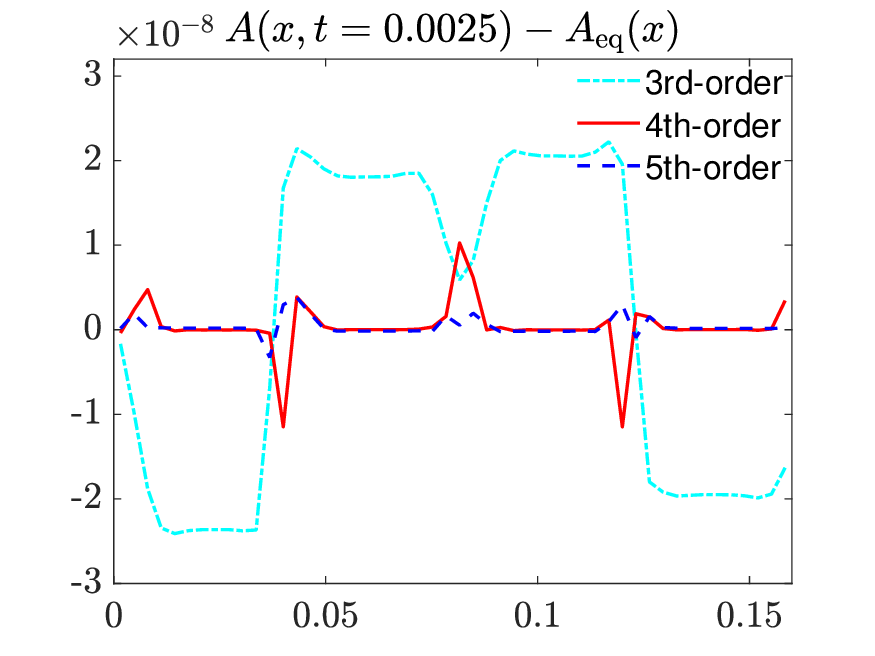}}\hspace*{0.005cm}
\subfigure[WB, $N=50$]{\includegraphics[trim=0.8cm 0.25cm 1.1cm 0.2cm,clip,width=3.2cm]{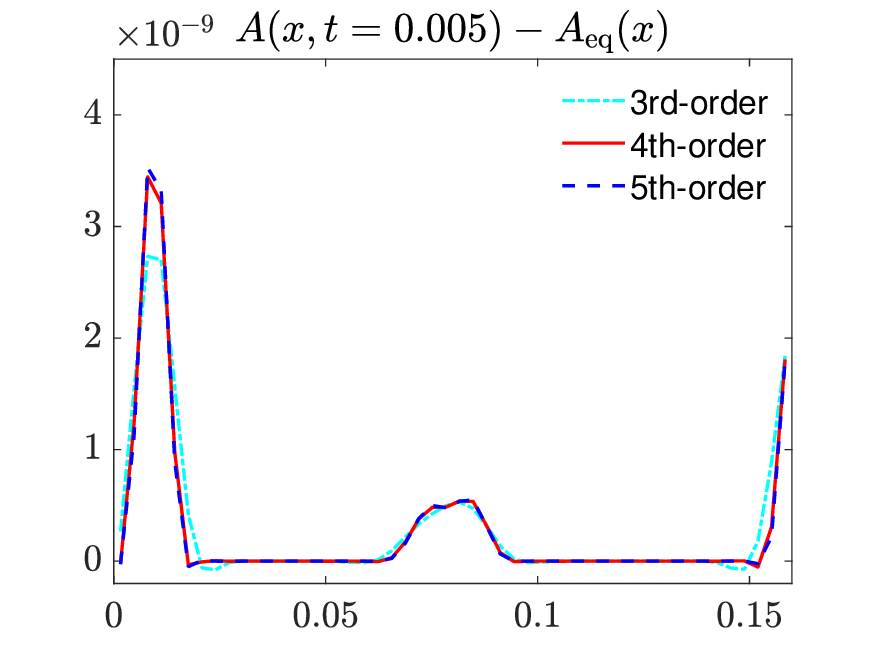}}\hspace*{0.005cm}
\subfigure[NWB, $N=50$]{\includegraphics[trim=0.8cm 0.25cm 1.1cm 0.2cm,clip,width=3.2cm]{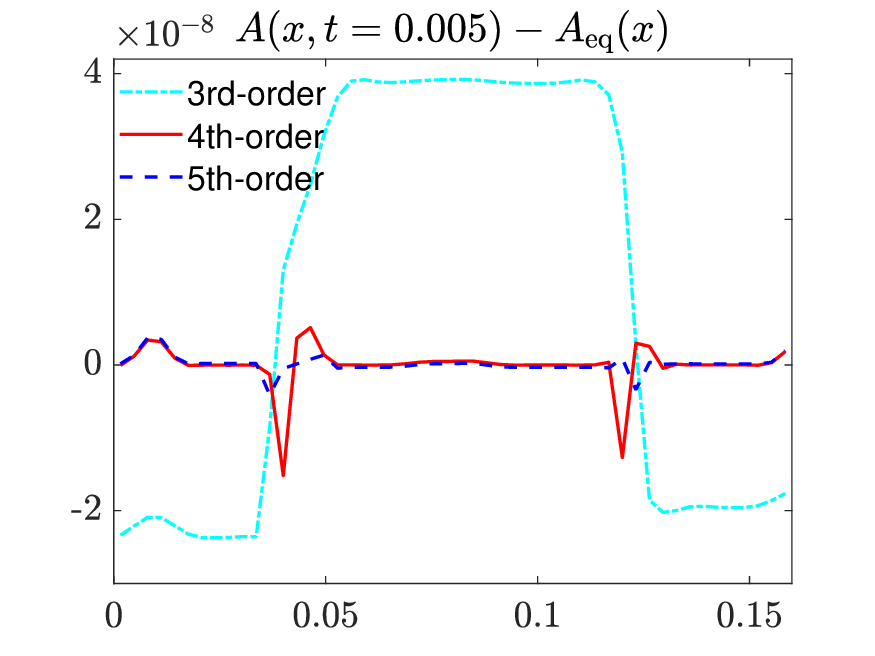}}}
\vskip5pt
\centerline{\subfigure[WB, $N=200$]{\includegraphics[trim=0.8cm 0.25cm 1.1cm 0.2cm,clip,width=3.2cm]{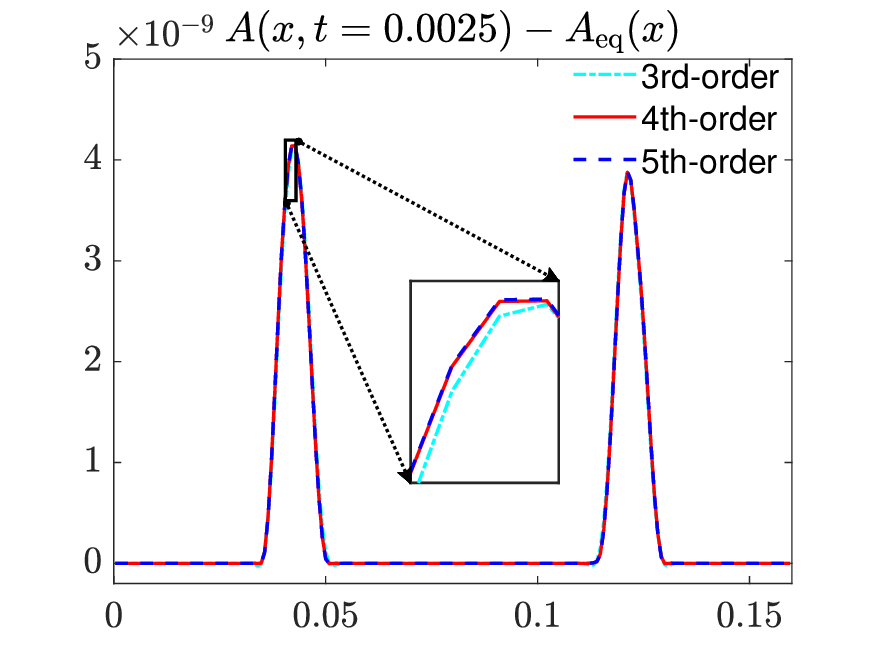}}\hspace*{0.005cm}
\subfigure[NWB, $N=200$]{\includegraphics[trim=0.8cm 0.25cm 1.1cm 0.2cm,clip,width=3.2cm]{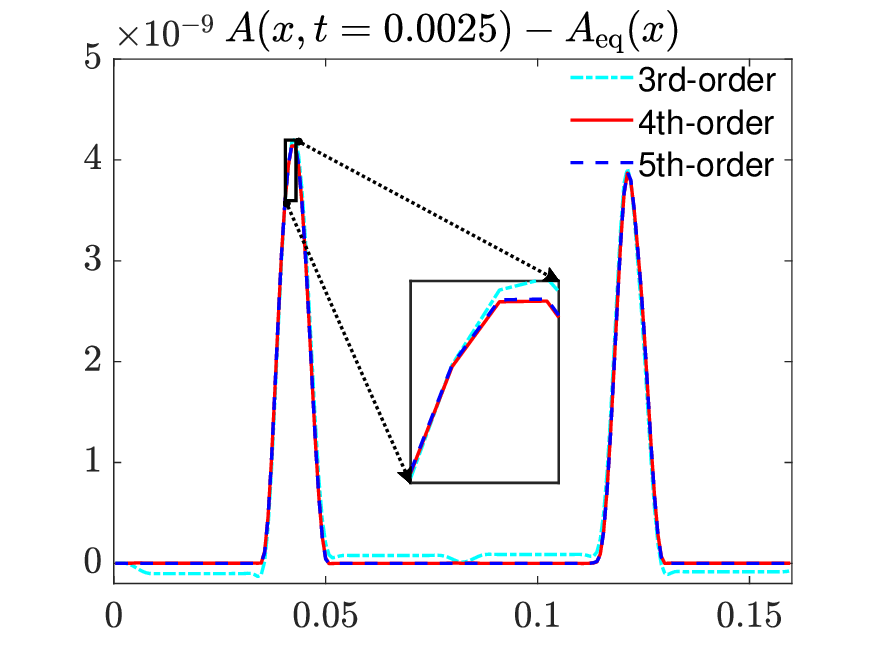}}\hspace*{0.005cm}
\subfigure[WB, $N=200$]{\includegraphics[trim=0.8cm 0.25cm 1.1cm 0.2cm,clip,width=3.2cm]{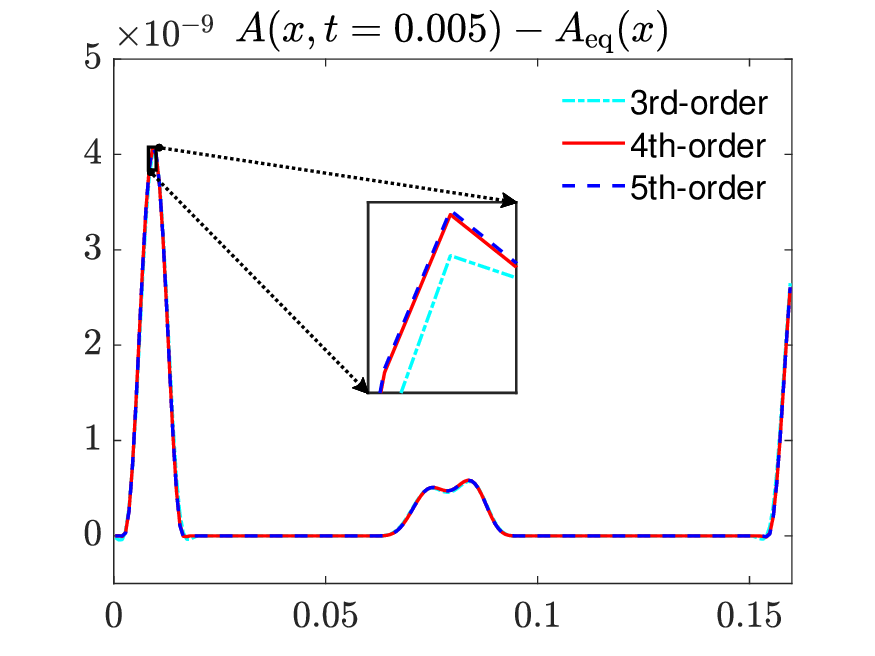}}\hspace*{0.005cm}
\subfigure[NWB, $N=200$]{\includegraphics[trim=0.8cm 0.25cm 1.1cm 0.2cm,clip,width=3.2cm]{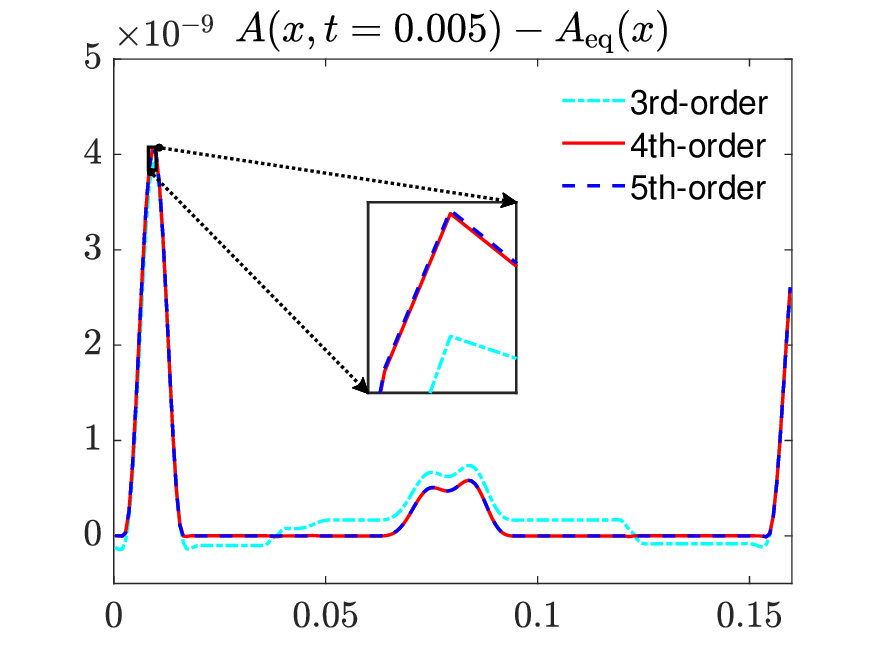}}}
\caption{\sf Example 5: Same as in Figure \ref{Ex5_fig1} but for $S_{\rm in}=0.1$.\label{Ex5_fig2}}
\end{figure}

\begin{figure}[ht!]
\color{red}
\centerline{\subfigure[WB, $N=50$]{\includegraphics[trim=0.8cm 0.25cm 1.1cm 0.2cm,clip,width=3.2cm]{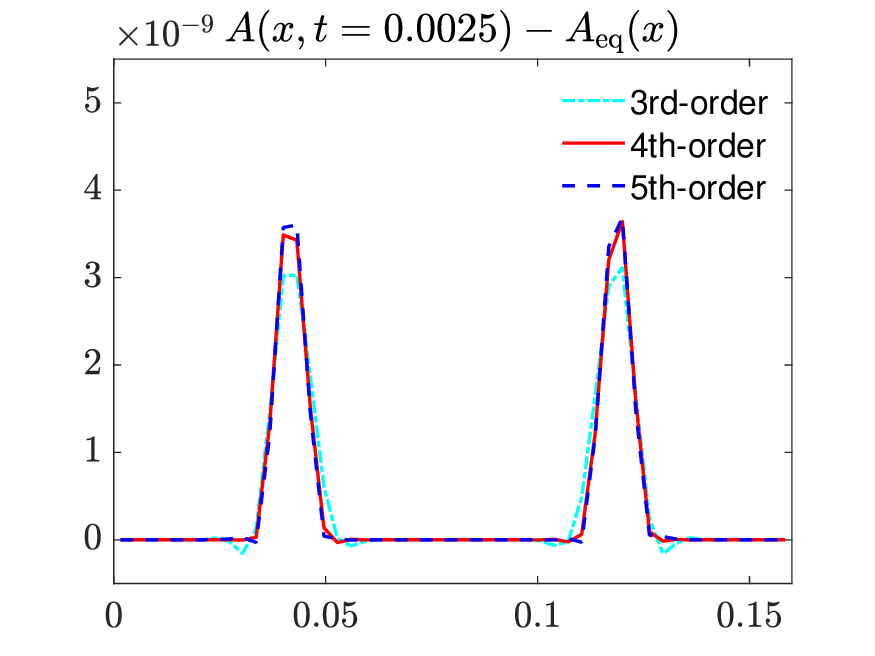}}\hspace*{0.005cm}
\subfigure[NWB, $N=50$]{\includegraphics[trim=0.8cm 0.25cm 1.1cm 0.2cm,clip,width=3.2cm]{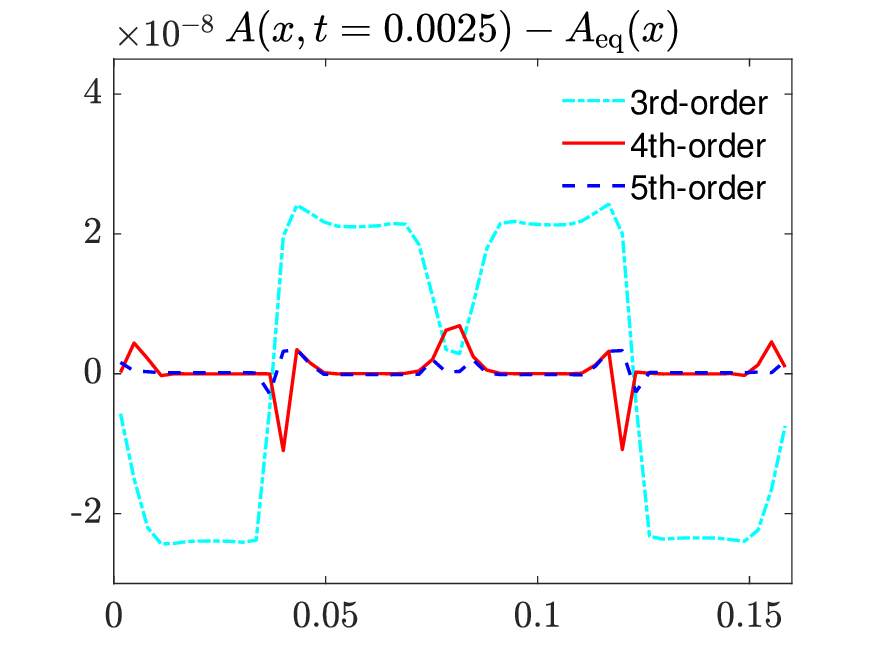}}\hspace*{0.005cm}
\subfigure[WB, $N=50$]{\includegraphics[trim=0.8cm 0.25cm 1.1cm 0.2cm,clip,width=3.2cm]{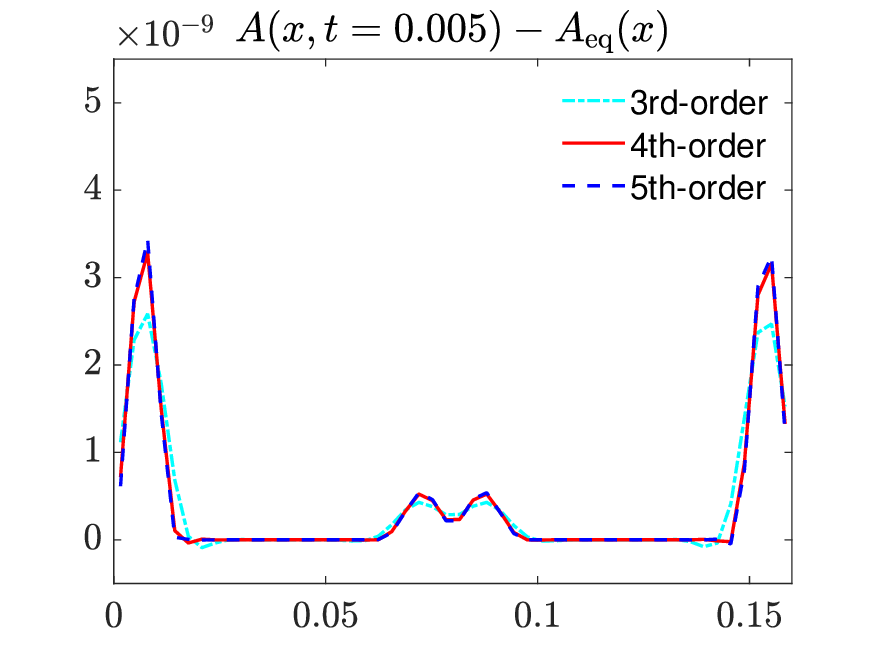}}\hspace*{0.005cm}
\subfigure[NWB, $N=50$]{\includegraphics[trim=0.8cm 0.25cm 1.1cm 0.2cm,clip,width=3.2cm]{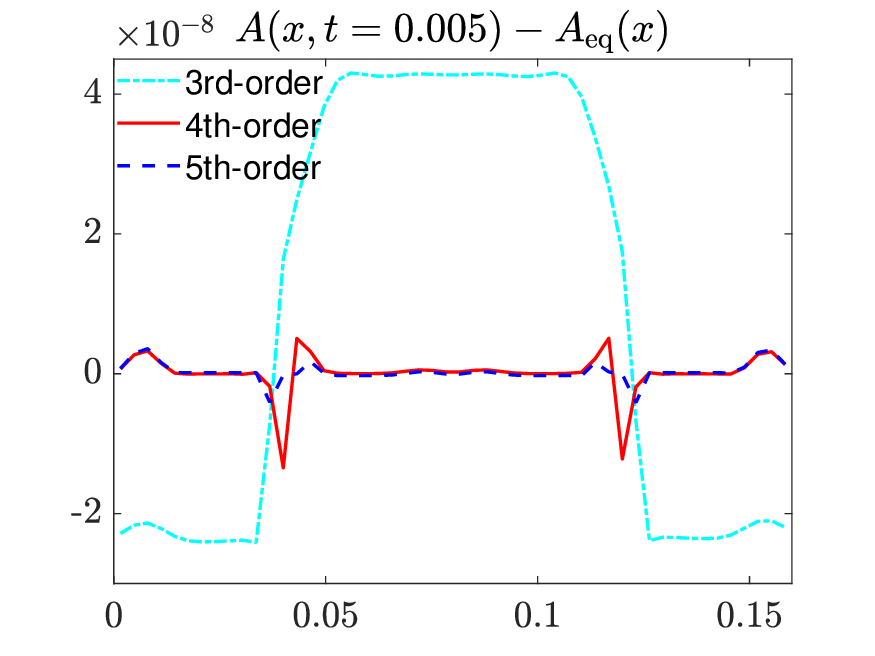}}}
\vskip5pt
\centerline{\subfigure[WB, $N=200$]{\includegraphics[trim=0.8cm 0.25cm 1.1cm 0.2cm,clip,width=3.2cm]{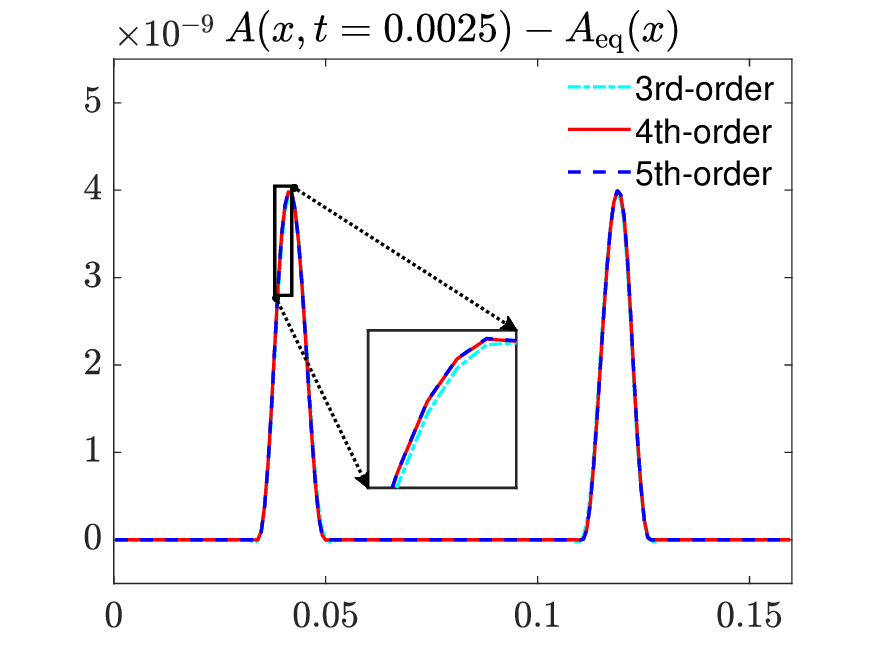}}\hspace*{0.005cm}
\subfigure[NWB, $N=200$]{\includegraphics[trim=0.8cm 0.25cm 1.1cm 0.2cm,clip,width=3.2cm]{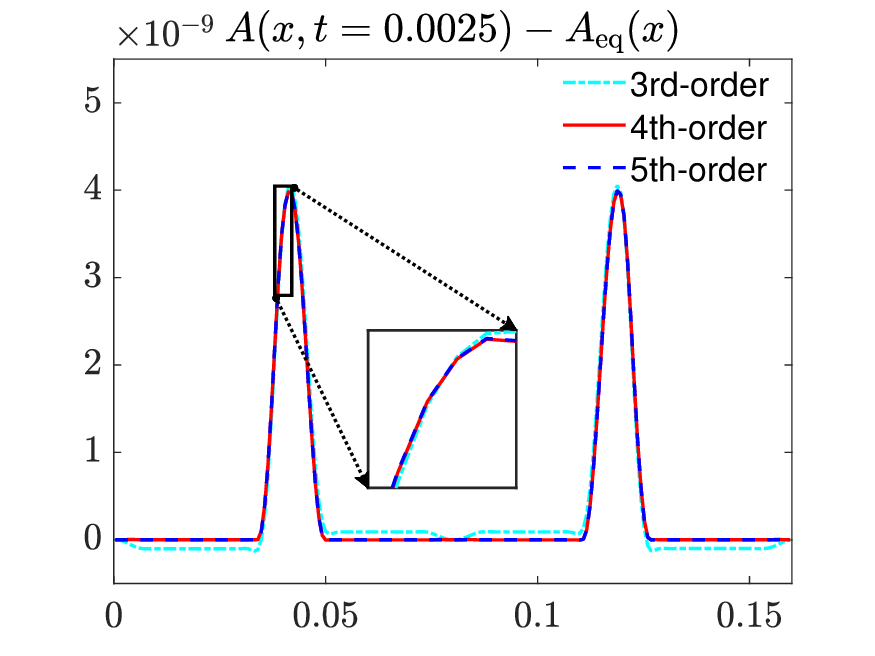}}\hspace*{0.005cm}
\subfigure[WB, $N=200$]{\includegraphics[trim=0.8cm 0.25cm 1.1cm 0.2cm,clip,width=3.2cm]{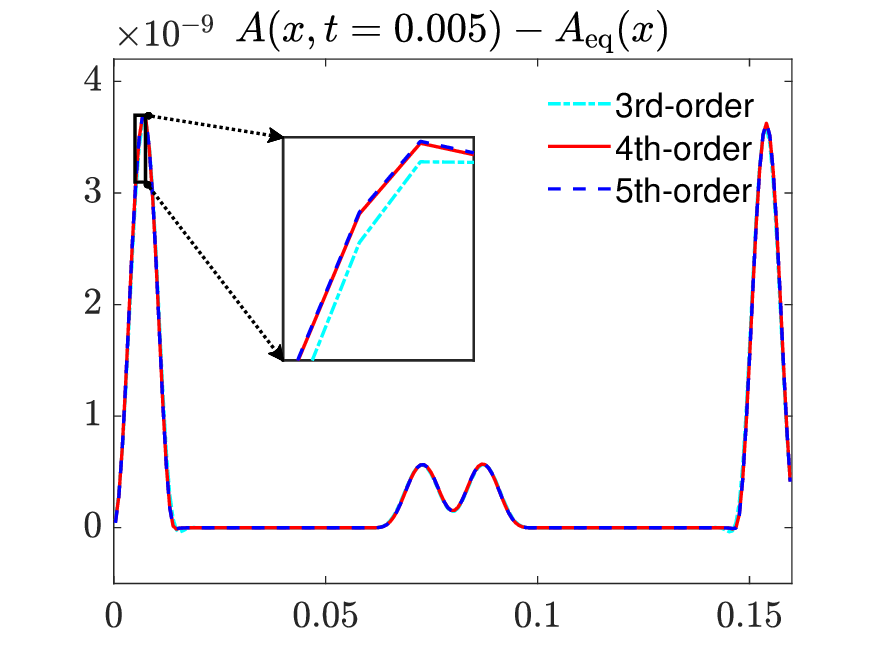}}\hspace*{0.005cm}
\subfigure[NWB, $N=200$]{\includegraphics[trim=0.8cm 0.25cm 1.1cm 0.2cm,clip,width=3.2cm]{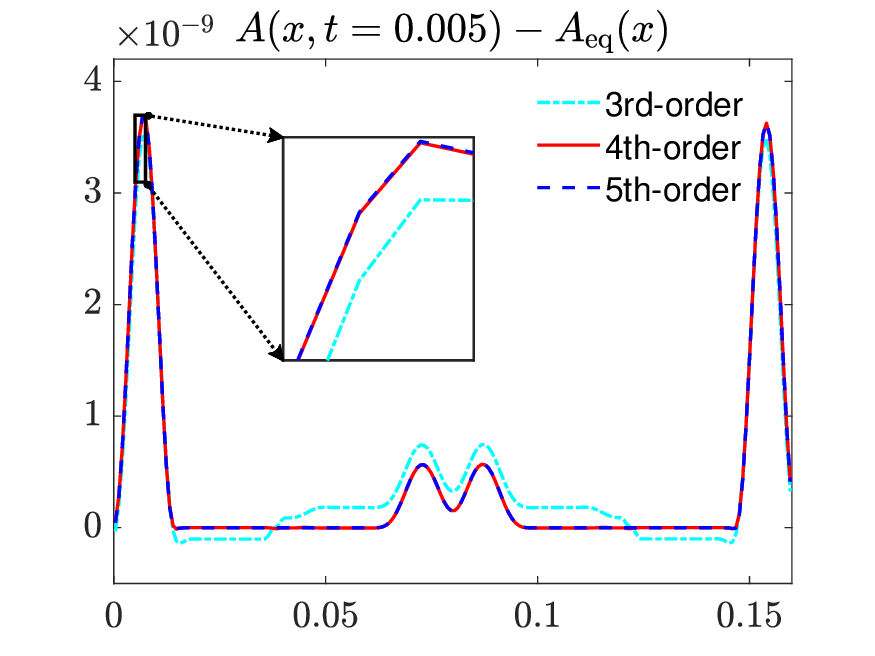}}}
\caption{\sf Example 5: Same as in Figures \ref{Ex5_fig1} and \ref{Ex5_fig2} but for $S_{\rm in}=0.01$.\label{Ex5_fig3}}
\end{figure}

\subsection*{Example 6---An ideal tourniquet problem}
In the sixth example, we investigate the dynamics of an ideal tourniquet problem in blood flow. This is a scenario where a tourniquet is promptly applied and then subsequently removed. It is similar to the dam break problem in shallow water equations and the Sod tube problem in compressible gas dynamics. The computational domain spans $[-0.04,0.04]$, with the arterial stiffness number set to $\kappa=10^7$, and the cross-sectional area at rest set as $A_0(x)=0$. The initial conditions are given by
\begin{equation*}
  A(x,0)=\left\{
  \begin{aligned}
  &\pi(5\times10^{-3})^2,&&x\in[-0.04,0],\\
  &\pi(4\times10^{-3})^2,&&x\in[0,0.04],
  \end{aligned}\right.
  \quad Q(x,0)=0.
\end{equation*}
We compute the numerical solutions using the proposed third-, fourth-, and fifth-order schemes up to a final time $t=0.005$, employing a mesh size consisting of 50 uniform cells. Additionally, \bla{an exact solution is also given on the same grids}. Figure \ref{Ex6_fig1} presents the solutions, showcasing the well capture of the right-going shock wave and the left-going rarefaction. Moreover, the superiority of the high-order scheme becomes apparent upon closer inspection of the smooth rarefaction waves in the zoomed-in views.

\begin{figure}[ht!]
\centerline{\subfigure[A(x,t)]{\includegraphics[trim=0.8cm 0.25cm 0.9cm 0.01cm,clip,width=4.5cm]{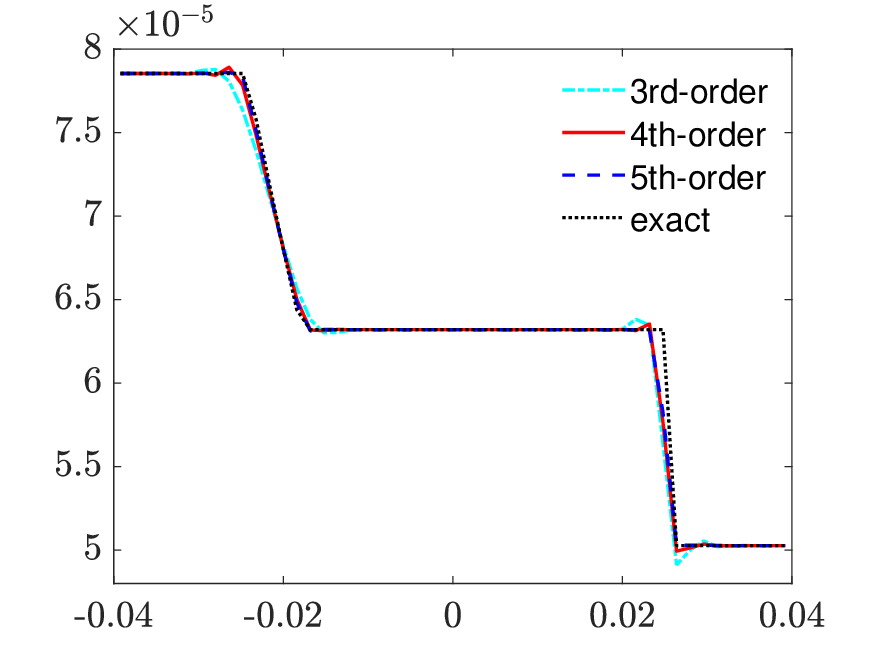}}\hspace*{0.45cm}
\subfigure[zoom of A(x,t)]{\includegraphics[trim=0.8cm 0.25cm 0.9cm 0.01cm,clip,width=4.5cm]{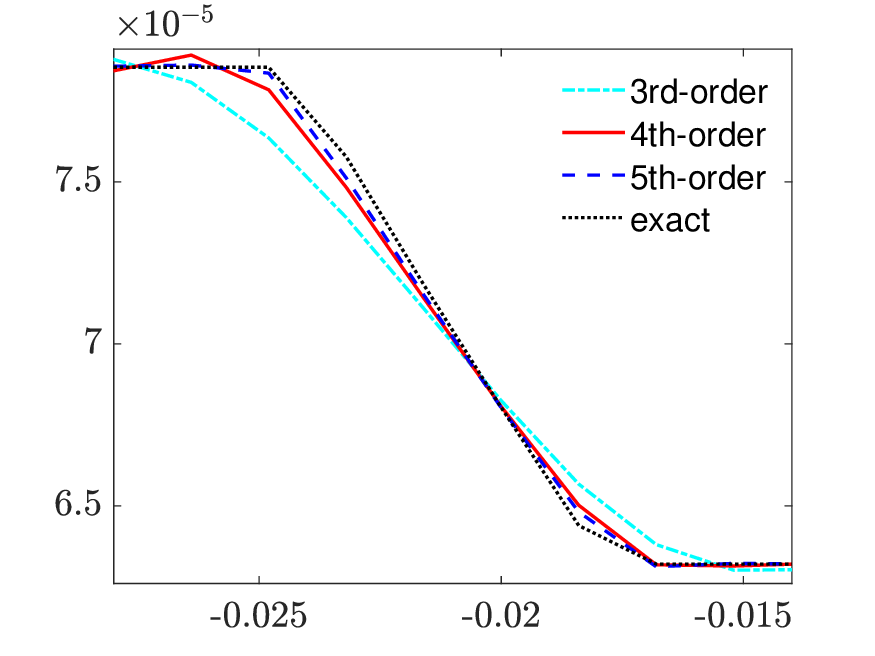}}}
\vskip5pt
\centerline{\subfigure[Q(x,t)]{\includegraphics[trim=0.8cm 0.25cm 0.9cm 0.01cm,clip,width=4.5cm]{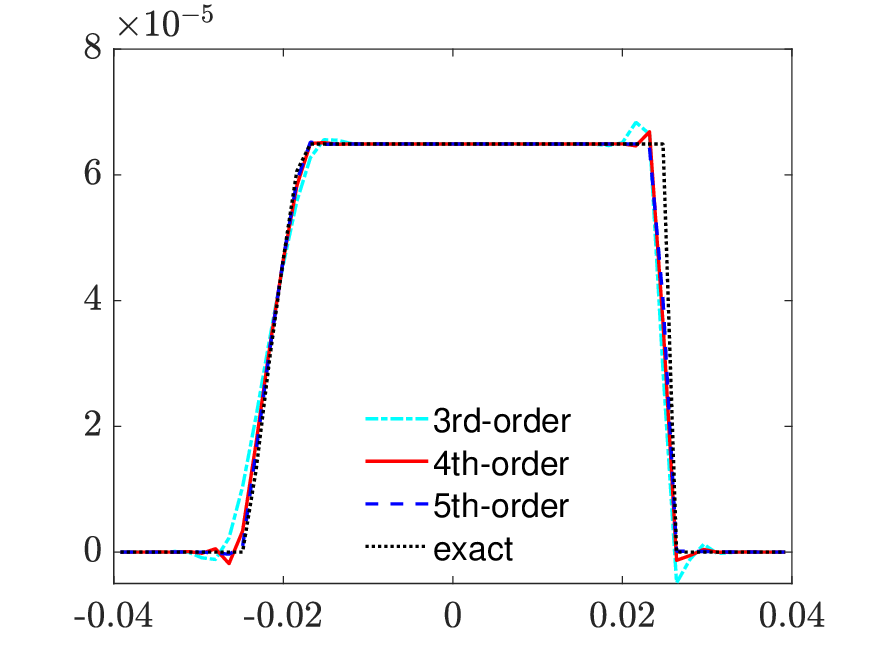}}\hspace*{0.45cm}
\subfigure[zoom of Q(x,t)]{\includegraphics[trim=0.8cm 0.25cm 0.9cm 0.01cm,clip,width=4.5cm]{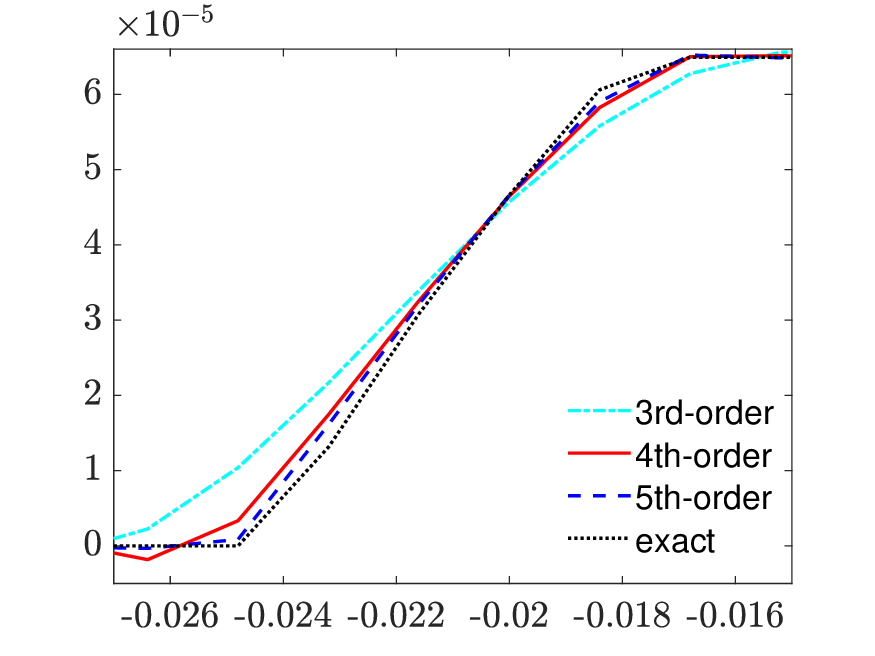}}}
\caption{\sf Example 6: Numerical solutions of the cross-sectional area $A(x,t)$ (top row) and discharge $Q(x,t)$ (bottom row). \label{Ex6_fig1}}
\end{figure}
 
\subsection*{Example 7---Riemann problems}
In the seventh example, we take the arterial stiffness number as $\kappa=3.31\times10^6$ and conduct numerical computations for two Riemann problems within the computational domain $[0,0.2]$. The cross-sectional area at rest is $A_0(x)=0$ and the initial conditions are given by 
\begin{equation}\label{Ex7_IC1}
A(x,0)=6.28\times10^{-4},\quad Q(x,0)=\left\{
  \begin{aligned}
  &-6.28\times10^{-4},&&x\in[0,0.1],\\
  &6.28\times10^{-4},&&x\in[0.1,0.2],
  \end{aligned}\right.
\end{equation}
and 
\begin{equation}\label{Ex7_IC2}
A(x,0)=6.28\times10^{-4},\quad Q(x,0)=\left\{
  \begin{aligned}
  &6.28\times10^{-4},&&x\in[0,0.1],\\
  &-6.28\times10^{-4},&&x\in[0.1,0.2].
  \end{aligned}\right.
\end{equation}
We perform numerical simulations using $100$ uniform cells for each of these two Riemann problems at final times $t=0.009$ and $t=0.012$, correspondingly. The computed results are presented in Figures \ref{Ex7_fig1} and \ref{Ex7_fig2}, along with \bla{the exact solutions computed on the same grids}. In Figure \ref{Ex7_fig1}, the solution contains two rarefaction waves. Upon closer examination in the zoomed-in sections (right column of Figure \ref{Ex7_fig1}), the superiority of high-order schemes is once again underscored. In Figure \ref{Ex7_fig2}, the solution consists of two shock waves which are well captured without generating spurious oscillations by all three schemes.

\begin{figure}[ht!]
\centerline{\subfigure[A(x,t)]{\includegraphics[trim=0.8cm 0.25cm 0.9cm 0.01cm,clip,width=4.5cm]{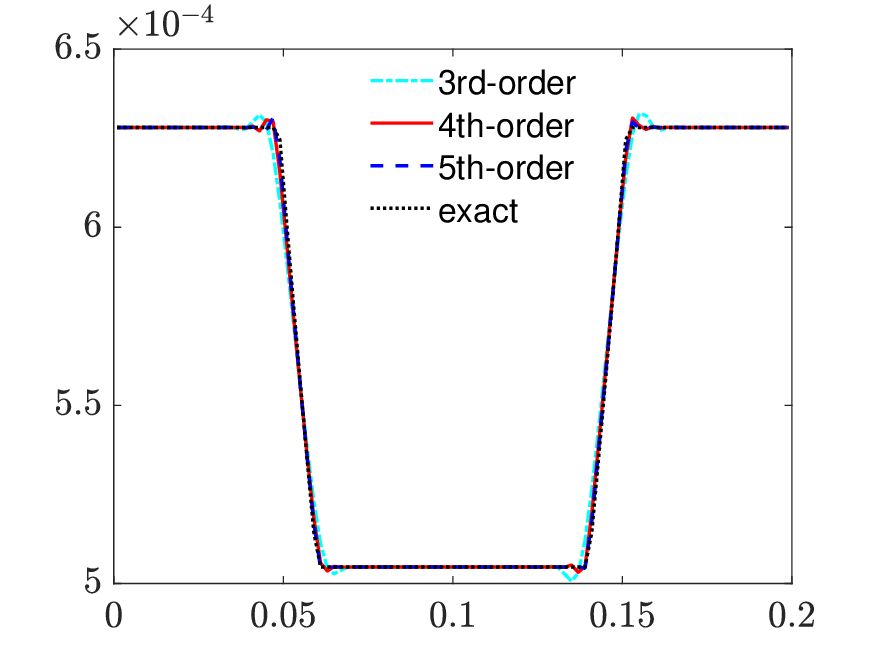}}\hspace*{0.45cm}
\subfigure[zoom of A(x,t)]{\includegraphics[trim=0.8cm 0.25cm 0.9cm 0.01cm,clip,width=4.5cm]{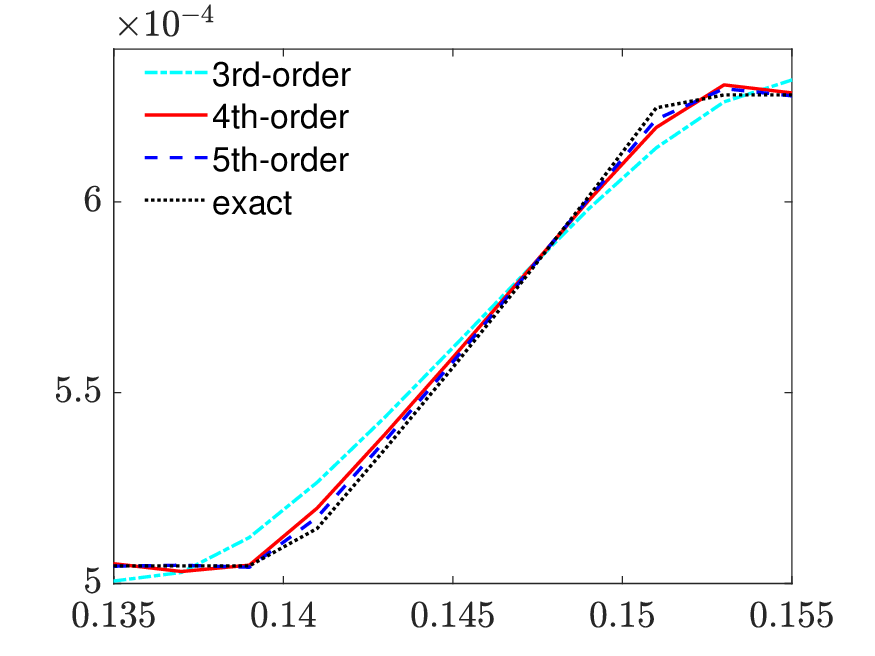}}}
\vskip5pt
\centerline{\subfigure[Q(x,t)]{\includegraphics[trim=0.8cm 0.25cm 0.9cm 0.01cm,clip,width=4.5cm]{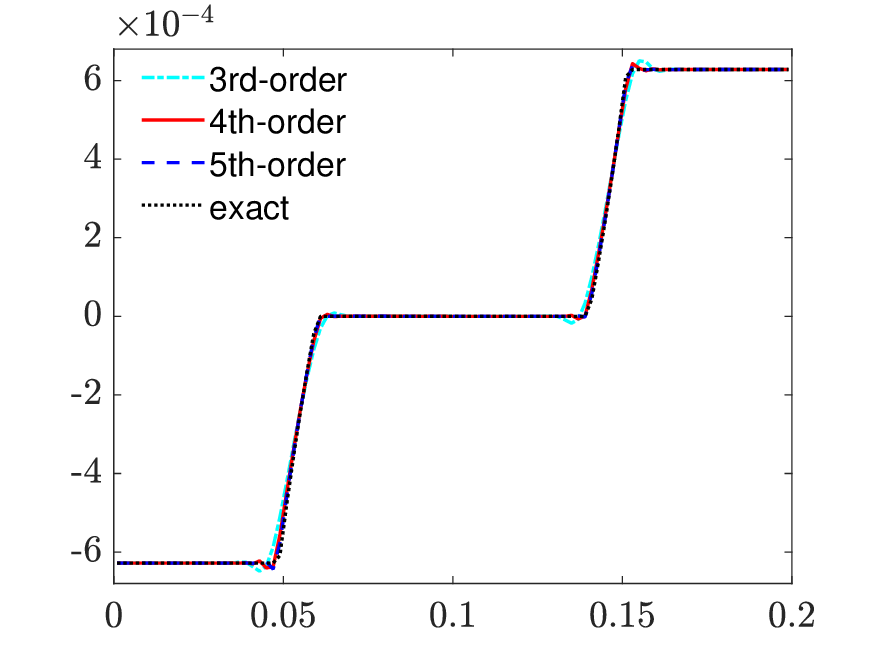}}\hspace*{0.45cm}
\subfigure[zoom of Q(x,t)]{\includegraphics[trim=0.8cm 0.25cm 0.9cm 0.01cm,clip,width=4.5cm]{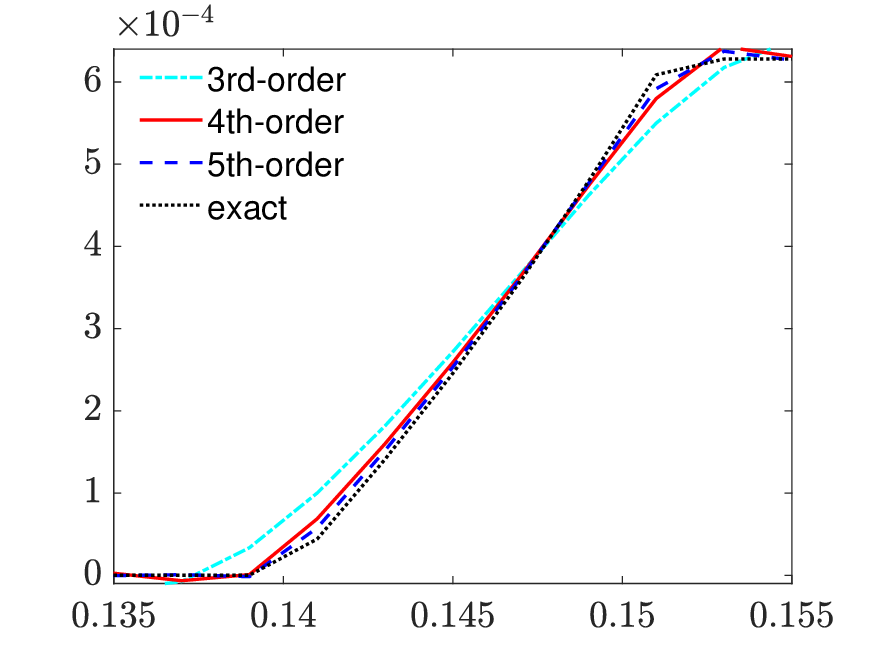}}}
\caption{\sf Example 7: Numerical solutions of the cross-sectional area $A(x,t)$ (top row) and discharge $Q(x,t)$ (bottom row) for initial conditions \eref{Ex7_IC1}. \label{Ex7_fig1}}
\end{figure}

\begin{figure}[ht!]
\centerline{\subfigure[A(x,t)]{\includegraphics[trim=0.8cm 0.25cm 0.9cm 0.01cm,clip,width=4.5cm]{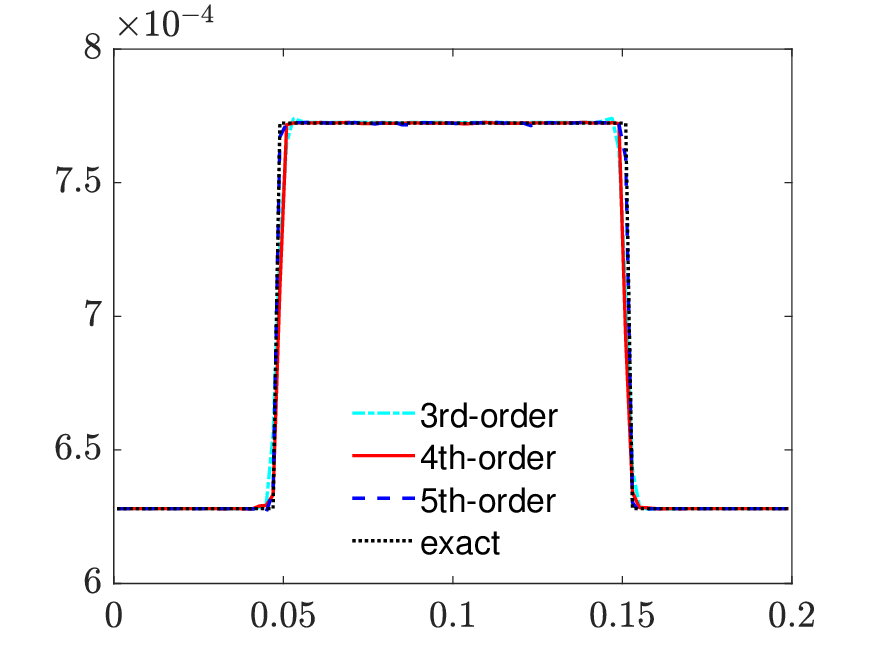}}\hspace*{0.45cm}
\subfigure[Q(x,t)]{\includegraphics[trim=0.8cm 0.25cm 0.9cm 0.01cm,clip,width=4.5cm]{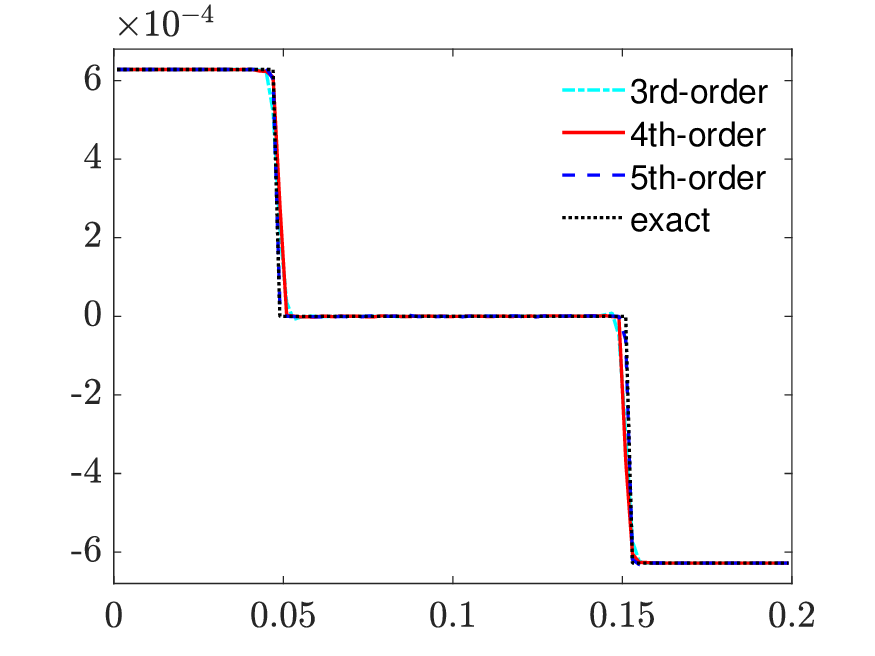}}}
\caption{\sf Example 7: Numerical solutions of the cross-sectional area $A(x,t)$ (left) and discharge $Q(x,t)$ (right) for initial conditions \eref{Ex7_IC2}. \label{Ex7_fig2}}
\end{figure}

{\color{red}
\subsection*{Example 8---Smooth steady-state solution in an artery with space-varying parameters}
In the eighth example, we consider a smooth steady-state solution in the domain $[0,5]$ with the following geometrical and mechanical parameters:
\begin{equation*}
  \mathcal{K}(x)=58725+100\mathcal{E}(x),\quad A_0(x)=0.0005+0.0001\mathcal{E}(x),\quad  p_{\rm ext}(x)=10000+100\mathcal{E}(x),
\end{equation*}
where $\mathcal{E}(x)=e^{-10(x-2.5)^2}$. The initial data is a subcritical stationary solution satisfying the steady-state condition \eref{s1}, given by
\begin{equation*}
\begin{aligned}
  &Q_{\rm s}=A(x)u(x)=1.0228\times10^{-3},\\
   &E_{\rm s}=\frac{(Q_{\rm s})^2}{2(A(x))^2}+\frac{1}{\rho}\Big(\mathcal{K}(x)\phi\big(\frac{A(x)}{A_0(x)}\big)+p_{\rm ext}(x)\Big)=\frac{1}{2}+\frac{1}{\rho}\Big(\mathcal{K}(0)\phi\big(\frac{Q_{\rm s}}{A_0(0)}\big)+p_{\rm ext}(0)\Big).
   \end{aligned}
\end{equation*}
 
We first check the ability of the high-order well-balanced schemes to exactly preserve the smooth steady-state solution. We compute the numerical solution using both WB and non WB schemes until a final time $t=5$ on a uniform mesh with $50$ cells. As shown in Table \ref{tab8}, only the WB methods maintain the steady-state solution within machine precision.

\begin{table}[!ht]
\color{red}
\caption{\sf Example 8: Errors in $A$ computed with WB and non WB (NWB) schemes using parabolic, cubic, and quartic polynomials.\label{tab8}}
\begin{center}
\begin{tabular}{c| c c c c c c}\hline
\multicolumn{1}{c|}{\multirow{2}{*}{Schs.}}  &\multicolumn{2}{c}{$r=2$} &\multicolumn{2}{c}{$r=3$} &\multicolumn{2}{c}{$r=4$}\\ \cline{2-7}
   & $L^1$-error   & $L^\infty$-error  & $L^1$-error  &$L^\infty$-error  & $L^1$-error  & $L^\infty$-error\\ \hline
 \multicolumn{1}{c|}{\multirow{1}{*}{WB}}  &2.82e-18 & 3.25e-18 &1.15e-18 &1.52e-18 &1.10e-15 & 3.11e-16 \\ \hline
 \multicolumn{1}{c|}{\multirow{1}{*}{NWB}}  &4.96e-03 & 1.18e-03 &3.94e-05 &3.59e-05 &3.94e-05 &3.47e-05 \\ \hline
\end{tabular}
\end{center}
\end{table}

To further illustrate the superiority of WB methods, we introduce a small perturbation $10^{-7}e^{-40(x-1)^2}$ to the underlying steady-state solution $A(x)$ and show the differences between the numerical solutions and the underlying stationary solution at times $t=0.1$ and $t=0.4$. It can be observed that, the evolution of the perturbation is better captured as the order of the method increases, and the left-going wave generated by the initial perturbation leaves the domain at $t=0.4$.

\begin{figure}[ht!]
\color{red}
\centerline{\subfigure[$t=0.1$]{\includegraphics[trim=0.8cm 0.25cm 0.9cm 0.01cm,clip,width=4.5cm]{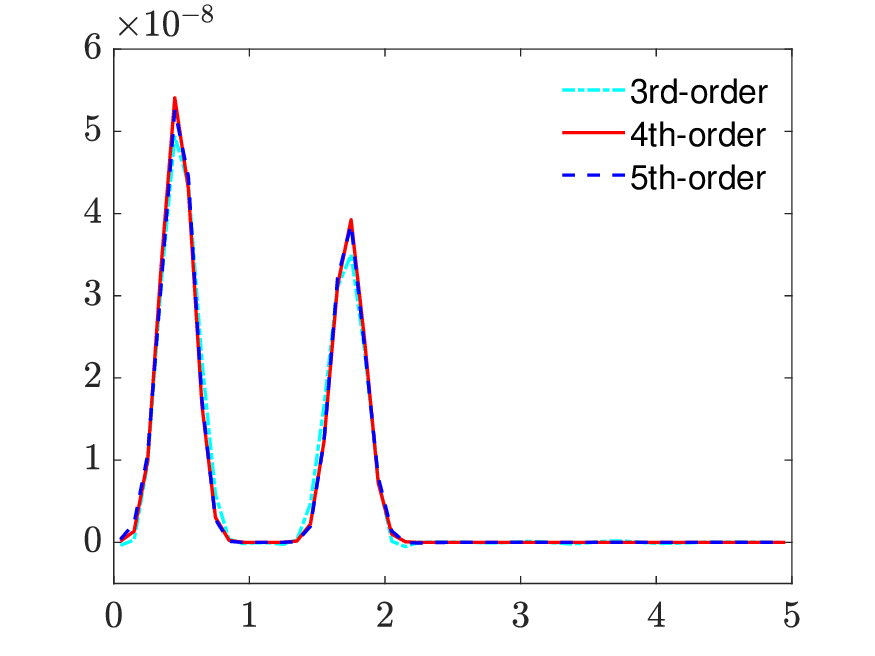}}\hspace*{0.45cm}
\subfigure[$t=0.4$]{\includegraphics[trim=0.8cm 0.25cm 0.9cm 0.01cm,clip,width=4.5cm]{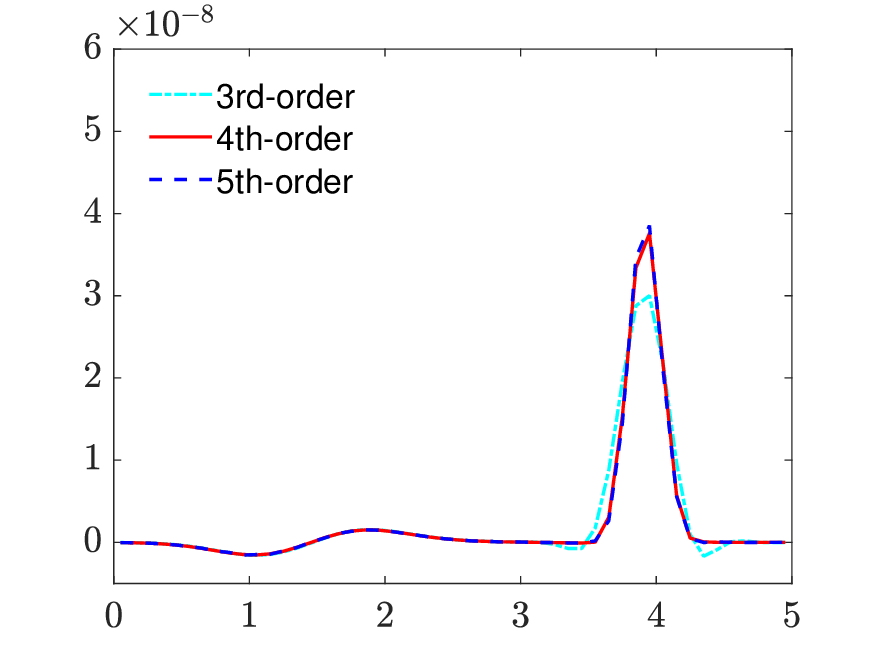}}}
\caption{\sf Example 8: plot of the perturbation, i.e., difference between the numerical solution and the initial steady-state solution. \label{Ex8_fig}}
\end{figure}

\subsection*{Example 9---Discontinuous steady-state solution in veins}
In the final example, we consider a solution linked by a stationary contact discontinuity with $u\neq0$. The initial conditions in the domain $[0,0.2]$ are given by
\begin{equation*}
  (A(x,t=0),u(x,t=0))=\left\{
  \begin{aligned}
  &(6.41356968\times 10^{-4}, 1),&&x<0.1,\\
  &(3.109988229063683\times 10^{-4},2.06224886),&&x>0.1,
  \end{aligned}\right.
\end{equation*}
with geometrical and mechanical parameters defined as
\begin{equation*}
  (A_0(x),\mathcal{K}(x),p_{\rm ext}(x))=\left\{
  \begin{aligned}
  &(6.2706\times 10^{-4},58725,9999.15),&&x<0.1,\\
  &(3.1353\times 10^{-4},587250,78001.73870735058),&&x>0.1.
  \end{aligned}\right.
\end{equation*}

We repeat the simulation as in Example 8 and first examine the WB property by computing the numerical solution up to a final time $t=1$ on a uniform mesh with $50$ cells. The obtained errors, reported in Table \ref{tab9}, show that only the WB methods capture the steady-state solution within machine accuracy, even in the presence of a non smooth profile.

\begin{table}[!ht]
\color{red}
\caption{\sf Example 9: Errors in $A$ computed with WB and non WB (NWB) schemes using parabolic, cubic, and quartic polynomials.\label{tab9}}
\begin{center}
\begin{tabular}{c| c c c c c c}\hline
\multicolumn{1}{c|}{\multirow{2}{*}{Schs.}}  &\multicolumn{2}{c}{$r=2$} &\multicolumn{2}{c}{$r=3$} &\multicolumn{2}{c}{$r=4$}\\ \cline{2-7}
   & $L^1$-error   & $L^\infty$-error  & $L^1$-error  &$L^\infty$-error  & $L^1$-error  & $L^\infty$-error\\ \hline
 \multicolumn{1}{c|}{\multirow{1}{*}{WB}}  &1.11e-20 & 1.08e-19 &1.08e-20 &1.08e-19 &1.81e-15 & 1.70e-14 \\ \hline
 \multicolumn{1}{c|}{\multirow{1}{*}{NWB}}  &5.70e-06 & 2.46e-04 &5.61e-06 &2.46e-04 &5.58e-06 &2.46e-04 \\ \hline
\end{tabular}
\end{center}
\end{table}

To further assess the performance of WB schemes, we add a small perturbation $10^{-5}e^{-20000(x-0.5)^2}$ to the steady-state solution $A(x)$ and compute the difference between the numerical solution and the underlying steady-state solution $A(x)$ at $t=0.001$ and $t=0.002$. The results, plotted in Figure \ref{Ex9_fig}, again demonstrate the ability of WB schemes to accurately capture the propagation of a small perturbation.
\begin{figure}[ht!]
\color{red}
\centerline{\subfigure[$t=0.001$]{\includegraphics[trim=0.8cm 0.25cm 0.9cm 0.01cm,clip,width=4.5cm]{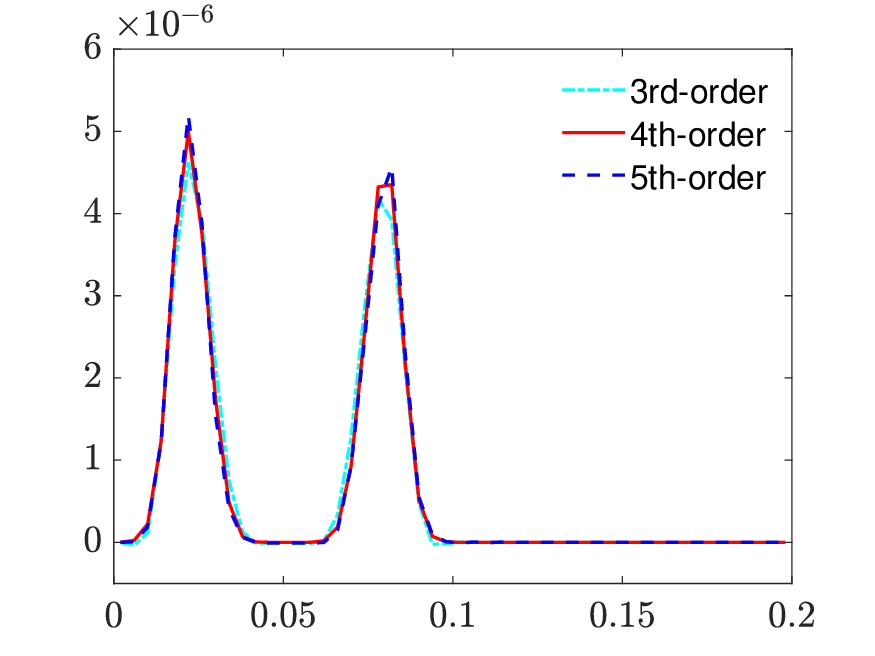}}\hspace*{0.45cm}
\subfigure[$t=0.002$]{\includegraphics[trim=0.8cm 0.25cm 0.9cm 0.01cm,clip,width=4.5cm]{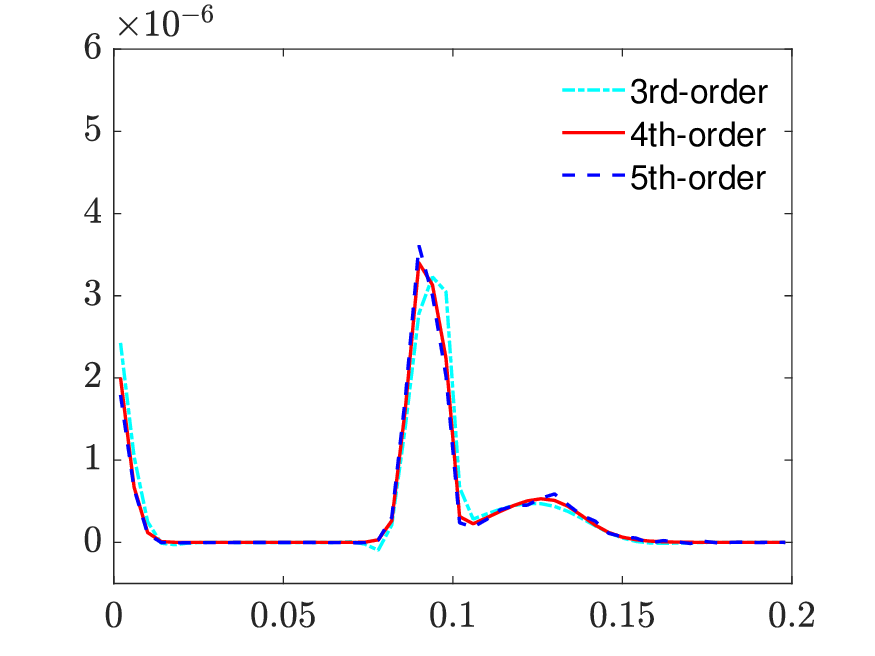}}}
\caption{\sf Example 9: plot of the perturbation, i.e., difference between the numerical solution and the initial steady-state solution. \label{Ex9_fig}}
\end{figure}
}

\section*{Acknowledgements} We would like to thank, warmly, the two anonymous referees for their insightful and constructive comments, which have significantly improved this work.

\bibliographystyle{siamplain}
\bibliography{reference}

\end{document}

%% file: ex_shared.tex
\usepackage{lipsum}
\usepackage{amsfonts}
\usepackage{graphicx,subfigure}
\usepackage{epstopdf}
\usepackage{algorithmic}
\usepackage{xr-hyper}
\usepackage{cleveref}
\usepackage{nicefrac}
\usepackage{bm}
\usepackage{bbm}
\usepackage{amssymb}
\usepackage{amsmath}
\usepackage{mathabx}
\usepackage{mathrsfs}
\usepackage{mathtools}
\usepackage{empheq}

\usepackage{multirow}

\ifpdf
  \DeclareGraphicsExtensions{.eps,.pdf,.png,.jpg}
\else
  \DeclareGraphicsExtensions{.eps}
\fi

\def\softd{{\leavevmode\setbox1=\hbox{d}%
          \hbox to 1.05\wd1{d\kern-0.4ex{\char039}\hss}}}

\newcommand\eref[1]{(\ref{#1})}

\newcommand{\jph}{{j+\frac{1}{2}}}
\newcommand{\jmh}{{j-\frac{1}{2}}}
\newcommand{\dx}{\Delta x}

\newcommand*\xbar[1]{%
  \hbox{%
    \vbox{%
      \hrule height 0.5pt 
      \kern0.4ex
      \hbox{%
        \kern-0.05em
        \ensuremath{#1}%
        \kern-0.00em
      }%
    }%
  }%
}

\newsiamremark{rmk}{Remark}
\newsiamthm{thm}{Theorem}
\newsiamremark{hypothesis}{Hypothesis}
\crefname{hypothesis}{Hypothesis}{Hypotheses}
\newsiamthm{claim}{Claim}
\newsiamthm{prop}{Proposition}
\newsiamthm{defn}{Definition}
\headers{High-order Well-balanced Scheme for Blood Flow}{Y. Liu and W. Barsukow}

\graphicspath{{figures/}{./}}

\title{An Arbitrarily High-Order Fully Well-balanced Hybrid Finite Element--Finite Volume Method for a One-dimensional Blood Flow Model\thanks{Submitted to the editors DATE.
\funding{The work of Y. Liu was supported in part by SNSF grant 200020\_204917, FZEB-0-166980, and UZH Postdoc Grant, 2024 / Verf\"{u}gung Nr. FK-24-110.}}}

\author{Yongle Liu\thanks{Corresponding author. Institute of Mathematics, University of Z\"{u}rich, Winterthurerstrasse 190, 8057 Z\"{u}rich, Switzerland
  (\email{yongle.liu@math.uzh.ch; liuyl2017@mail.sustech.edu.cn}).}
\and Wasilij Barsukow\thanks{Bordeaux Institute of Mathematics, Bordeaux University and CNRS/UMR 5251, Talence, 33405,
France (\email{wasilij.barsukow@math.u-bordeaux.fr}).}}

\usepackage{amsopn}
